\documentclass[11pt,a4paper]{article}
\usepackage[]{graphicx,float,latexsym,times}
\usepackage{amsfonts,amstext,amsmath,amssymb,amsthm}

 \RequirePackage[OT1]{fontenc}
\long\def\symbolfootnote[#1]#2{\begingroup%
\def\thefootnote{\fnsymbol{footnote}}\footnote[#1]{#2}\endgroup}

\usepackage{amssymb}
\usepackage{xspace}
\usepackage{geometry}
\usepackage{bm}
\usepackage{natbib}
\usepackage{bbm}
\usepackage{color}
\usepackage{hyperref}
 \usepackage[utf8]{inputenc}    
\usepackage{supertabular}
\usepackage{multirow,multicol}
 
\DeclareMathAccent{\widehat}{\mathord}{largesymbols}{"62}
\DeclareMathAccent{\widetilde}{\mathord}{largesymbols}{"65}
\def\pth#1{\left(#1\right)}
\def\acc#1{\left\{#1\right\}}

\def\eeX{\mathbb{X}}

\def\e1{1\!\!1}

\def\ebo{\textrm{\mathversion{bold}$\mathbf{\beta^0}$\mathversion{normal}}}

\def\eb{\textrm{\mathversion{bold}$\mathbf{\beta}$\mathversion{normal}}}  
\def\ed{\textrm{\mathversion{bold}$\mathbf{\delta}$\mathversion{normal}}}

\def\eU{\textrm{\mathversion{bold}$\mathbf{\Upsilon}$\mathversion{normal}}} 

\def\eE{\mathbb{E}}

\def\e1{1\!\!1}

\def\eeX{\mathbb{X}}

\theoremstyle{plain}

\newtheorem{theorem}{Theorem}[section]

\newtheorem{lemma}{Lemma}[section]

\newtheorem{remark}{Remark}[section]

\newcommand{\beqn}{\begin{eqnarray*}}
\newcommand{\eeqn}{\end{eqnarray*}}
  
\def\ee1{\textrm{\mathversion{bold}$\mathbf{\varepsilon}$\mathversion{normal}}}

\def\eu{\mathbf{{u}}}

\newcommand{\N}{\mathbb{N}}
\newcommand{\R}{\mathbb{R}}
\newcommand{\PP}{\mathbb{P}}

\def\eX{\mathbf{X}}

\def\ex{\mathbf{x}}

\newcommand{\Var}{\mathbb{V}\mbox{ar}\,}

\def\argmin{\mathop{\mathrm{arg\,min}}} 

\begin{document}

\title {Automatic selection by penalized asymmetric $L_q$-norm in an high-dimensional model with grouped variables }
 \date{}
\maketitle

\author{
\begin{center}
 Angelo ALCARAZ and Gabriela CIUPERCA
 \footnote{\small {CONTACT: angelo.alcaraz@ens-paris-saclay.fr,  Gabriela.Ciuperca@univ-lyon1.fr}}\\
\small{ENS Paris-Saclay, 4 Av. des Sciences, 91190 Gif-sur-Yvette
	\\
	Institut Camille Jordan,  UMR 5208, Universit\'e Claude Bernard Lyon 1, France}
\end{center}
}


\begin{abstract}
  The paper focuses on the automatic selection of the grouped explanatory variables in an high-dimensional model, when the model errors are asymmetric. After introducing the model and notations, we define the adaptive group LASSO expectile estimator for which we prove the oracle properties: the sparsity and the asymptotic normality. 
 Afterwards, the results are generalized by considering the asymmetric $L_q$-norm loss function. The theoretical results are obtained in several cases with respect  to the number of variable groups. This number can be fixed or dependent on the sample size $n$,  with the possibility that it is of the same order as $n$. Note that these  new estimators allow us to consider weaker assumptions on the data and on the model errors than the usual ones. Simulation study demonstrates the competitive performance of the proposed penalized expectile regression, especially when the   samples size  is close to the number of explanatory variables and model  errors are asymmetrical. An application on air pollution data is considered.
\end{abstract}



\section{Introduction}
\label{Intro}
With advances in computing and data collection, we are increasingly faced with handling problems of high-dimensional models with grouped explanatory variables for which an automatic selection of relevant groups must be performed. 
For many real applications, the selection of the relevant grouped variables  is very important to the prediction performance of the response variable and of the model parameters.  
Several methods have been proposed in the literature for automatic selection of the variables and then of the groups of relevant variables, by penalizing the loss function with an adaptive penalty of  LASSO type. The loss term is to be chosen according to the assumptions on the model errors while the penalty term aims to select significant (group of) explanatory variables.  Let us then give some recent bibliographical  references on this topic. 
 For the least square (LS) loss function with adaptive group LASSO penalty, \cite{Wang-Leng-08} prove the convergence rate and the oracle properties of the associated estimator when the group number of explanatory variables is fixed. These results were extended in \cite{Zhang-Xiang-16} where the number $p$ of groups depends on the sample size $n$ but with $p$ of order strictly less than $n$. For a quantile model  with grouped variables, \cite{Ciuperca.2016}, \cite{Ciuperca.2020}  automatically select the groups of relevant variables by adaptive LASSO and adaptive elastic-net methods, respectively. Based on the PCA and PLS $n^{1/2}$-consistent estimators corresponding to adaptive weights, \cite{Mendez.2021} study the sparsity of the adaptive group LASSO quantile estimator. On the other hand, \cite{Wang-Wang.14} study the adaptive LASSO estimators for generalized linear model (GLM), while \cite{Wang-Tian-19} consider the grouped variables for a GLM, their results including the case $p>n$. \cite{Zhou.Zhou.2019} prove the oracle inequalities for the estimation and prediction error of overlapping group Lasso method in the GLMs. 
 Concerning the computational aspects, when the LS loss function is penalized with $L_1$-norm for the subgroup of coefficients and with $L_2$-norm or $L_1$-norm for fused coefficient subgroups, \cite{Dondelinger.2000} present two co-ordinate descendent algorithm for calculating the corresponding estimators. \\
 Emphasize now that for a model with asymmetric errors, the LS estimation method is not appropriate, while that quantile  makes the inference more difficult because of the non differentiability of the loss function. Combining the ideas of the LS method to those of the   quantile,   the expectile estimation method can be considered, with  the advantage  that the loss
 function is differentiable and then the theoretical study is more amenable,   the numerical computation being simplified (a  comparison between quantile  and expectile models can be found in  \cite{Waltrup.15}).  Remark also that the quantile model is a generalization of the median model and characterizes the tail behaviour of a distribution, while the expectile method is a generalization  of least squares. The reader can  found asymptotic properties of the sample expectiles in   \cite{Holzmann.2016}. The automatic selection of the relevant variables in a model with ungrouped explanatory variables was realized in \cite{Liao.2018}, \cite{Ciuperca.2021} by the adaptive LASSO expectile estimation method, results generalized afterwards by \cite{Hu.Chen.2021} for the asymmetric $L_q$-norm loss function.\\
 The present paper generalizes the these last three listed papers, for a model with grouped explanatory variables, the number $p$ of groups being either fixed or dependent on $n $. The convergence rates and the asymptotic distributions   of the estimators depend on the size order of  $p$ with respect to $n$. Note that our results remain valid when $p$ is of the same order as $n$. It should also be noted that the proposed methods  encounter fewer numerical problems than for  to the non-derivable  quantile loss function, proposed and studied in \cite{Ciuperca.2016}, \cite{Ciuperca.2020}. The simulation study and application on  real data show that  our proposed adaptive group  LASSO methods have good performance.\\
 The remainder of the paper is organized as follows. Section \ref{Sect_model} introduces the model and general notations. Section \ref{Sect_expectile} defines and studies the adaptive group LASSO expectile estimator firstly when the number $p$ of groups is fixed and afterwards when $p$ depends on the sample size $n$. The method and the results are generalized in Section \ref{Sect_generalisation} for asymmetric $L_q$-norm loss function, where convergence rate and oracle properties are stated for the adaptive group LASSO $L_q$-estimator. Section \ref{Sect_simu} presenting simulation results is followed by an application on real data in Section \ref{Sect_appli}. The theoretical result proofs are relegated in Section \ref{Sect_proofs}.
\section{Model and notations}
\label{Sect_model}
In this section, we introduce the studied statistical model and some general notations.\\
We give some notations before presenting the statistical model. All vectors and matrices are denoted by bold symbols and all vectors are column. For a vector $\bm{v}$, we denote by $\bm{v}^\top$ its transposed, by $\| \bm{v} \|_2 $ its euclidian norm, by $\|\textbf{v}\|_1$   and $\|\textbf{v}\|_\infty$, its $L_1$,   $L_\infty$ norms, respectively.
For a positive definite matrix $\bm{M}$, we denote by $\delta_{\min}(\bm{M})$ and $\delta_{\max}(\bm{M})$ its smallest and largest eigenvalues, respectively.
We will also use the following notations: if $U_n$ and $V_n$ are random variable sequences, $V_n=o_\PP(U_n)$ means that $\underset{n \rightarrow \infty}{\text{lim}} \mathbb{P}(|  {U_n}/{V_n} | >e)=0 $ for all $ e >0$ .  Moreover, $V_n=O_\PP(U_n)$ mean that there exists $ C >0 \, \text{so that} \, \underset{n \rightarrow \infty}{\text{lim}} \mathbb{P}(|  {U_n}/{V_n} | >C )<e  $ for all $ e>0$.  If $(a_n)_{n \in \N }$, $(b_n)_{n \in \N}$ are two positive deterministic sequences we denote by $a_n \gg b_n$ when $\lim_{n \rightarrow \infty} a_n/b_n=\infty$. We use $\overset{\mathcal{L}}{\underset{n \rightarrow \infty}{\longrightarrow}}$,$\overset{\mathbb{P}}{\underset{n \rightarrow \infty}{\longrightarrow}}$ to represent the convergence in distribution and in probability.    For an event $E$, $\e1_{E}$ denotes the indicator function that the event $E$ happens. For a real $x$, $\lfloor x \rfloor$ is the integer part of $x$.  For an index set ${\cal A}$, let us denote either by $|{\cal A}|$ or by $Card({\cal A})$  the cardinality of ${\cal A}$ and by  ${\cal A}^c$ its complementary set. 
Throughout this paper, $C$ will always denote a generic constant, not depending on the size $n$ and its value is not of interest.
We will denote by $\textbf{0}_k$ the zero $k$-vector. When it is not specified, the convergence is as $n \rightarrow \infty$.
\newline
We consider the following model with $p$ groups of explanatory variables:
\begin{equation}
Y_i=\sum_{j=1}^p \textbf{X}_{ij}^\top\bm{\beta_j}+\varepsilon_i=\mathbb{X}_i^\top\eb + \varepsilon_i, \qquad i=1, \cdots, n,
\label{eq0}
\end{equation}
where  $Y_i$ is  the response variable, $\varepsilon_i$ is the model error, both being random variables.  For each group $j \in \{1,...,p\}$, the  vector of parameters is $\eb_j=(\beta_{j1},...,\beta_{jd_j}) \in \mathbb{R}^{d_j}$ and the design for observation $i $ is the $d_j$-vector $\textbf{X}_{ij}$. In the other words, the vector with all the coefficients is $\eb =(\eb_1,...,\eb_p)$ related to the vector with all explanatory variables $\mathbb{X}_i=(\textbf{X}_{i1},...,\textbf{X}_{ip})$.
State now that for $j=1, \cdots, p$, we denote by $\eb_j^0=(\beta_{j1}^0,...,\beta_{jd_j}^0)$ 
the true (unknown) value of the coefficient parameters $\eb_j$ corresponding to the $j$-th group of explanatory variables. For observation $i$, we denote by $X_{ij,k}$ the $k$th variable of the $j$th group. Emphasize that the relevant groups of explanatory variables correspond to the non-zero vectors. More precisely, without loss of generality, we suppose that the first $p^0 \, ( p^0 \leq p)$  groups of variables are relevant:
\begin{equation}
\label{E2.1}
\| \eb_j^0 \|_2 \ne 0 ,\quad \forall j \leq p^0 \: \text{and} \:  \| \eb_j^0 \|_2 = 0 ,\quad \forall j > p^0 . 
\end{equation}
Let's take: $r\equiv \sum_{j=1}^p d_j$, $ r^0\equiv \sum_{j=1}^{p^0} d_j $. 
So, $r$ is the total number of parameters. The numbers  $r^0$ and $p^0$ are unknown, while $r$ and $p$ are known. For model (\ref{eq0}), we define the index set of non-zero true group parameters:
$$
\mathcal{A} \equiv\{ j \in \{1,...,p\}; \; \| \eb_j^0 \|_2 \ne 0 \}, \quad \eb_{\mathcal{A}} \equiv  \{ \eb_j, j \in \mathcal{A} \}.
$$
Taking into account model (\ref{E2.1}), we have $\mathcal{A}=\{1,...,p^0\}$. Throughout this paper, we denote by $\eb_{\mathcal{A}}$ the non-zero sub-vectors of $\eb$ which contains all sub-vectors $\eb_j$, with $j \in {\cal A}$.
For any  $j \in {\cal A}$, the corresponding group of explanatory variables  is relevant, while for $j \in {\cal A}^c$, the group of variables  is irrelevant. In the applications, since $\ebo$ is unknown, then  ${\cal A}$ is also unknown. 
We denote by  $\mathbb{X}_{i,\mathcal{A}}$  the $r^0$-vector which contains the elements $\textbf{X}_{ij}$ with $j \in \{1,...,p^0\}$.
\section{Adaptive group LASSO expectile estimator}
\label{Sect_expectile}
In this section we introduce and study the automatic selection of the relevant group of explanatory variables by penalizing the expectile loss function with an adaptive $L_{2,1}$-norm. Two cases will be considered: first when the number $p$ of variable groups is fixed and afterwards when it is dependent on the number $n$ of observations. 
For a fixed $\tau \in ]0,1[$, we define the expectile function $\rho_{\tau}: \R \rightarrow \mathbb{R}^+$ :
\begin{equation}
\label{function_expectile}
\rho_{\tau}(u)\equiv u^2 |\tau-\e1_{u<0} | .
\end{equation}
The value $\tau$ is called expectile index and $\rho_\tau$ is the expectile function of order $\tau$. We also define: $\psi_\tau(t) \equiv \rho_{\tau}(\varepsilon\def\e1{1\!\!1}-t) $, $
g_{\tau}(\varepsilon)\equiv \psi_\tau^{'}(0)=-2\varepsilon(\tau \e1_{\varepsilon\geq 0}+(1-\tau)\e1_{\varepsilon<0})$,  $ h_{\tau}(\varepsilon) \equiv\psi_\tau^{''}(0)=2(\tau \e1_{\varepsilon \geq 0}+(1-\tau)\e1_{\varepsilon<0})$, $\sigma^2_{g_{\tau}} \equiv \Var[g_{\tau}(\varepsilon)]$,  $ \mu_{h_{\tau}} \equiv\mathbb{E}[h_{\tau}(\varepsilon)]$.
We can then define the expectile estimator:
\begin{equation}
\label{E2.2}
\widetilde{\eb}_n \equiv \underset{\eb \in \mathbb{R}^r}{\text{arg min}} \;  \mathcal{G}_n(\eb), \quad \textrm{with} \quad \mathcal{G}_n(\eb) \equiv  \sum_{i=1}^n \rho_{\tau}(Y_i-\mathbb{X}_i^\top\eb).
\end{equation}
The expectile estimator $\widetilde{\eb}_n$ is written as $\widetilde{\eb}_n=(\widetilde{\eb}_{n;1}, \cdots , \widetilde{\eb}_{n;p})$. 
Remark that for $\tau=1/2$, we find the classical least square (LS) estimator.
Unfortunately, the estimator  $\widetilde{\eb}_n$ doesn't allow automatic selection of groups of variables. In order to select the significant explanatory variables, we would need to perform some hypothesis tests, which can be tedious if $p$ is large. In order to overcome this inconvenience, \cite{Zou.2006} proposed in the particular case $\tau=1/2$ and ungrouped explanatory variables, under specific assumptions on the design, the adaptive LASSO estimator which automatically selects variables.
Inspired by this idea, in this section   we would like to select the groups of relevant  variables. 
Then, we will introduce an estimator, denoted by $\widehat{\eb}_n=(\widehat{\eb}_{n;1}, \cdots, \widehat{\eb}_{n;p})$, minimizing the expectile process (\ref{E2.2}) penalized by a LASSO adaptive term.  The main objective of this section is to  study the asymptotic properties of adaptive group LASSO expectile estimator (\textit{ag}\_$\mathcal{E}$) defined by:
\begin{equation}
\label{bebe}
\widehat{\eb}_n \equiv \underset{\eb \in \mathbb{R}^r}{\argmin}\mathcal{Q}_n(\eb),
\end{equation}
with $\mathcal{Q}_n(\eb)$ the penalized expectile process:
$$
\mathcal{Q}_n(\eb) \equiv n^{-1} \sum_{i=1}^n \rho_{\tau}(Y_i-\mathbb{X}_i^\top\eb)+\lambda_n \sum_{j=1}^p \widehat{\omega}_{n;j} \| \eb_j \|_2,
$$
where the weights $\widehat{\omega}_{n;j}$ are: $
\widehat{\omega}_{n;j} \equiv\| \widetilde{\eb}_{n;j} \|_2 ^{-\gamma}, \quad \gamma >0$.
The weights $\widehat{\omega}_{n;j} $ depend on the expectile estimators corresponding to the $j$-th group of the explanatory variables. 
The sequence $(\lambda_n)_{n \in \mathbb{N}}$ is called tuning parameter and $\gamma$ is a known positive constant. 
Corresponding to $\widehat{\eb}_n$, we define the index set of the non null adaptive group LASSO expectile estimators:
$$
\widehat{\mathcal{A}}_n \equiv \{ j \in \{1,...,p\}; \; \| \widehat{\eb}_{n;j} \|_2 \ne 0 \}.
$$
The set $\widehat{\mathcal{A}}_n$ is an estimator of ${\cal A}$ and $|\widehat{\mathcal{A}}_n|$ an estimator for the number $p^0$ of significant groups of variables.
Two cases will be considered in this section: one where $p$ is constant and one where $p$ is depending of $n$.
Remark that    $\widehat{\eb}_n$  is a generalization of the estimator for non grouped explanatory variables ($d_j=1$ for any $j=1, \cdots , p$) proposed by  \cite{Liao.2018} when $p$ is fixed and by \cite{{Ciuperca.2020}} for $p$ depending on $n$.\\
We present now the assumptions, on the model errors and on the design, that will be necessary in this section. The assumptions on $(\lambda_n)$ and $\gamma$ will be presented in each of the next two subsections, depending on whether or not $p$ depends on $n$.\\
The model errors $(\varepsilon_i)_{1\leq i \leq n}$ satisfy the following assumption:
\begin{description}
	\item \textbf{(A1)} $(\varepsilon_i)_{i \leq n}$ are independent, identically distributed, having 0 as their $\tau$th expectile, with a positive density, continuous near 0.
	We also suppose $\mathbb{E}(\varepsilon_i^4)< \infty$. 
\end{description}
Two assumptions are considered for the deterministic design $(\mathbb{X}_i)_{1\leq i \leq n}$:
\begin{description}
	\item \textbf{(A2)}   
	$  \max_{1 \leqslant i \leqslant n} \| \eeX_i \|_\infty  \leq C_0$ for some constant $C_0>0$.
	\item \textbf{(A3)} We denote by $\eU_n \equiv n^{-1} \sum_{i=1}^n \mathbb{X}_i \mathbb{X}_i^\top$. 
	There exists two constants $ 0<m_0 \leq M_0 < \infty $, so that $m_0 \leq \delta_{\min}( \eU_n) \leq \delta_{\max}( \eU_n) \leq M_0$.
\end{description}
Assumption (A3) is classical for linear regression, while assumption (A1) is typical in the context of expectile regression (\cite{Liao.2018}, \cite{Ciuperca.2021}) and it implies $\eE[\varepsilon_i(\tau \e1_{\varepsilon_i \geq 0}+(1-\tau)\e1_{\varepsilon_i<0})]=-\eE[g_\tau(\varepsilon_i)]=0$. Hence, the expectile index $\tau$ will be fixed throughout in this section, such that $\eE[g_\tau(\varepsilon_i)]=0$.  Assumption (A2)   is  commonly considered   in high-dimensional models when the number of parameters diverges with $n$ (\cite{Wang-Wang.14}, \cite{Zhao-Chen-Zhang.18}, \cite{Wang-Tian-19}, \cite{Ciuperca.2021}, \cite{Hu.Chen.2021}, \cite{Zhou.Zhou.2019}).\\
Note that in the particular case $\tau=1/2$ of the LS loss function and $p$ fixed, we obtain the  adaptive group LASSO LS estimator proposed and studied by \cite{Wang-Leng-08}. 
The proofs of the results presented in the following two subsections can be found in Subsection \ref{proof_expectile}. 
\subsection{Fixed $p$ case}
\label{subsect_pfix}
In order to study the properties of $\widehat{\eb}_n$ when number $p$ of groups is fixed,  the following two conditions are considered for the tuning parameter sequence $(\lambda_n)_{n \in \N} $, $\lambda_n \underset{n \rightarrow +\infty}{\longrightarrow} 0$ and on  $ \gamma$:
\begin{equation}
\label{E3.1}
\begin{split}
 &(a) \quad n^{1/2} \lambda_n  \underset{n \rightarrow +\infty}{\longrightarrow} 0,  \\
&(b) \quad n^{(\gamma+1)/2} \lambda_n \underset{n \rightarrow +\infty}{\longrightarrow} \infty . 
\end{split}
\end{equation}
The two conditions of (\ref{E3.1}) are classical for adaptive LASSO penalties when the number of coefficients is fixed (see for example \cite{Wu.Liu.2009}, \cite{Ciuperca.2016}, \cite{Liao.2018}). 
Observe also that Assumption (A3) implies when $p$ is fixed that there exists  a   positive definite $r\times r$ matrix $\eU$  so that 
\begin{equation} n^{-1} \sum_{i=1}^n \mathbb{X}_i \mathbb{X}_i^\top \underset{n \rightarrow +\infty}{\longrightarrow} \eU. 
\label{Gg}
\end{equation}
Let us first find the  convergence rate of the adaptive group LASSO expectile estimator $\widehat{\eb}_n$.
\begin{lemma} 
\label{L1}
Under assumptions (A1), (A2), (A3) and (\ref{E3.1})(a), we have: $\| \widehat{\eb}_n - \eb^0 \|_2 = O_\PP(n^{-1/2})$.
\end{lemma}
\noindent The convergence rate of $\widehat{\eb}_n $ is of optimal order $n^{-1/2}$ and coincides with that obtained by \cite{Liao.2018} for the SCAD expectile estimator with ungrouped variables, or by \cite{Wang-Leng-08} for the adaptive group LASSO LS estimator, both when the number of model parameters is fixed. 
The following theorem shows the asymptotic normality of  $\widehat{\eb}_n$ corresponding to the relevant groups of variables. In this case, the non-zero parameter estimators have the same asymptotic distribution they would have if the true non-zero parameters were known.  
\begin{theorem} 
\label{T1}
Under assumptions (A1)-(A3) and (\ref{E3.1}), we have: $
n^{1/2}(\widehat{\eb}_n-\eb^0)_{\mathcal{A}} \overset{\mathcal{L}}{\underset{n \rightarrow \infty}{\longrightarrow}} \mathcal{N}\big(\textbf{0}_{r^0},  \sigma^{2}_{g_{\tau}} \mu^{-2}_{h_{\tau}} \eU_{\mathcal{A}}^{-1}\big)$,
with $\eU_{\mathcal{A}}$ the sub-matrix of $\eU$ with indexes in $\{1,...,d_1,d_1+1,...,d_1+d_2,...,\sum_{j=1}^{p^0-1} d_j+1, \cdots ,\sum_{j=1}^{p^0} d_j\}$.
\end{theorem}
\noindent The next theorem establishes the sparsity property of the estimator $\widehat{\eb}_n$. This result show that the estimators of the groups of non-zero parameters are indeed non-zero with a probability converging to 1 when $n$ converges to infinity. 
\begin{theorem} 
\label{T2}
Under assumptions of Theorem \ref{T1}, we have: $
\underset{n \rightarrow \infty}{\lim} \mathbb{P}[\widehat{\mathcal{A}}_n=\mathcal{A}]=1$.
\end{theorem}
\noindent An estimator which satisfies the sparsity property and is  asymptotically normal for the non-zero coefficient vector,  enjoys the oracle properties. 
\begin{remark}
	Considering the $\widehat{\omega}_{n;j}$ related to the $j$th group, the expectile estimator can be replaced by any estimator of $\eb_j$ converging consistently to $\eb^0_j$ with a rate  of $n^{-1/2}$.
\end{remark}


\subsection{Case when $p$ is depending on $n$}
\label{subsect_exp_pvar}
In this section, we study the asymptotic properties of the adaptive group LASSO expectile estimator $\widehat{\eb}_n$ defined by (\ref{bebe}) when $p$ depends on $n$: $p=p_n$ and $p_n \underset{n \rightarrow +\infty}{\longrightarrow} +\infty$. For readability, we keep the notation $p$ instead of $p_n$, even if $p$ depends on $n$. Obviously, $p^0$, $r$, $r^0$, ${\cal A}$ can also depend on $n$, but for the same readability reasons, the index $n$ does not appear. 
\newline
Since $p$ depends on $n$,  new assumptions are considered:
\begin{description}
\item \textbf{(A4)} $p=O(n^c)$ with $c \in [0,1/2)$.
\item \textbf{(A5)} For $h_0\equiv \underset{1 \leq j \leq p^0}{\min} \| \eb_j^0 \|_2$, there exist a constant $ K >0$ so that $K \leq n^{-\alpha} h_0$ and $\alpha > (c-1)/2$.
\end{description}
Assumptions (A4) is often used when $p$ is depending on $n$ (see \cite{Ciuperca.2021} and \cite{Ciuperca.2016}). Since $\| \ex \|_2 \leq r^{1/2} \| \ex \|_\infty$, then assumptions (A2) and (A4) imply 
\begin{eqnarray}
( {p}/{n})^{ {1}/{2}} \underset{1 \leq i \leq n }{\max} \| \mathbb{X}_i \|_2 \underset{n \rightarrow +\infty}{\longrightarrow} 0 ,
\label{tl}
\end{eqnarray}
 supposition considered by \cite{Ciuperca.2016} for the grouped   quantile case,  by \cite{Ciuperca.2021}, \cite{Hu.Chen.2021} for ungrouped expectile case and ungrouped $L_q$-norm case, respectively.
Assumption (A4) enable us to control the convergence rate of $p$, and ensure the convergence towards 0 of the sequence $p/n$. Assumption (A5) is classical for a model with the number of variable groups depends on the sample size (see \cite{Ciuperca.2016}, \cite{Zhang-Xiang-16}). 
This assumption implies that coefficients can be non-zero for a fixed $n$ but they converge to 0 when $n$ converges to infinity. 
By assumption (A3) we deduce that $r= \textrm{rank} (n^{-1}\sum^n_{i=1} \mathbb{X}_i \mathbb{X}_i ^\top)$. Moreover, since $\textrm{rank} (n^{-1}\sum^n_{i=1} \mathbb{X}_i \mathbb{X}_i ^\top) \leq n$, then $r \leq n$. Thus we will consider $p=O(n^c)$, with $c \in [0,1]$  such that $r \leq n$. 
In the study of the estimator $\widehat{\eb}_n$, two cases will be considered:   $c \in [0,1/2)$ and   $c \in [1/2,1]$. Note that the first case covers also the possibility that $p$ does not depend on $n$.  To shorten the paper, only case $c \in [0,1/2)$ will be presented for the expectile loss function. The   case $c \in [1/2,1]$ will be considered in Section \ref{Sect_generalisation} which will have the expectile function as a special case. 
\newline
 The tuning parameter $(\lambda_n)_{n \in \N}$, with $\lambda_n \underset{n \rightarrow +\infty}{\longrightarrow} 0$,  satisfies instead of conditions (\ref{E3.1})(a) and (\ref{E3.1})(b) the following two assumptions:
\begin{equation}
\label{E4.1}
    \lambda_n   n^{  1/2-\alpha \gamma} \underset{n \rightarrow +\infty}{\longrightarrow} 0,
\end{equation}
\begin{equation}
\label{E4.2}
    \lambda_n n^{ (1-c)(1+\gamma)/2} \underset{n \rightarrow +\infty}{\longrightarrow} +\infty.
\end{equation}
Obviously, for $c=0$ and $\alpha=0$ relations (\ref{E4.1}) and (\ref{E4.2})  become the two conditions of (\ref{E3.1}). Condition  (\ref{E4.2}) is  considered also by  \cite{Zhang-Xiang-16} for LS loss function ($\tau=1/2$) with adaptive group LASSO penalty, in the case of i.i.d. model errors of zero mean and finite variance. Condition (\ref{E4.1}) is weaker than $\lambda_n   n^{  (1+c)/2-\alpha \gamma} \underset{n \rightarrow +\infty}{\longrightarrow} 0$ the corresponding one of \cite{Zhang-Xiang-16}.   
 \\
The following theorem deals with the  convergence rate of $\widehat{\eb}_n$. This rate will be the same than for the expectile estimator $\widetilde{\eb}_n$ in the case $p$ depending on $n$ (see  \cite{Ciuperca.2021}). The same convergence rate was obtained for adaptive group LASSO estimators when the loss functions are: likelihood in  \cite{Wang-Tian-19}, LS in  \cite{Zhang-Xiang-16}, quantile in  \cite{Ciuperca.2016}. Obviously if $c = 0$ we get the results of the previous subsection.
\begin{theorem} 
\label{T3}
Under assumptions (A1)-(A5) and condition (\ref{E4.1}) for $(\lambda_n)_{n \in \N}$, we have: $\| \widehat{\eb}_n - \eb^0 \|_2 = O_\PP( (p/n )^{1/2})$.
\end{theorem}
\noindent The following theorem shows the first oracle property of  $\widehat{\eb}_n$: the sparsity.
\begin{theorem} 
\label{T4}
Under assumptions (A1)-(A5), conditions (\ref{E4.1}) and (\ref{E4.2}) for $(\lambda_n)_{n \in \N}$, we have: $\underset{n \rightarrow \infty}\lim \mathbb{P}[\widehat{\mathcal{A}}_n=\mathcal{A}]=1$.
\end{theorem}
\noindent The last theorem shows the second oracle property of $\widehat{\eb}_n$: the asymptotic normality. 
\begin{theorem} 
\label{T5}
Under the same assumptions as in Theorem \ref{T4}, for all vector $ \textbf{u} \in \mathbb{R}^{r^0}$  so that $ \| \textbf{u} \|_2 =1$, considering  $\eU_{n,\mathcal{A}}\equiv n^{-1} \sum_{i=1}^n \mathbb{X}_{i,\mathcal{A}} \mathbb{X}_{i,\mathcal{A}}^\top$, we have: $
 n^{1/2}(\textbf{u}^\top \eU_{n,\mathcal{A}}^{-1} \textbf{u})^{-1/2} \textbf{u}^\top (\widehat{\eb}_n-\eb^0)_{\mathcal{A}} \overset{\mathcal{L}}{\underset{n \rightarrow \infty}{\longrightarrow}} \mathcal{N}\big(0,  \sigma^2_{g_{\tau}} \mu^{-2}_{h_{\tau}}  \big)$.
\end{theorem}
\begin{remark}	
	(i) The expectile estimator in the weights $\widehat{\omega}_{n;j}$  can be replaced by any estimator of $\eb_j$ which converges consistently to $\eb^0_j$ with a  rate of $(p/n)^{1/2}$.\\
	(ii) In accordance with the results of  Subsection \ref{subsect_Lq2}, when $c \in [1/2,1]$ and for the particular case $q=2$, the convergence rate of $\widehat{\eb}_n$ to $\ebo$ is $\| \widehat{\eb}_n - \ebo \|_1=O_\PP( (p/n )^{1/2})$. Then, for the asymptotic normality,  the vector $\eu \in \R^{r^0}$ is such that $\| \eu \|_1=1$. This result is a generalization of that obtained by   \cite{Ciuperca.2021} for a model with ungrouped variables and estimated by adaptive LASSO expectile method.
\end{remark}

\section{Generalisation: adaptive group LASSO  $L_q$-estimator}
\label{Sect_generalisation}
In this section the results of Section \ref{Sect_expectile} are generalized by considering the asymmetric $L_q$-norm as the loss function, with $q > 1$. For the index $\tau \in (0,1)$, we consider now the following loss function: 
\begin{equation}
\label{def_Lq}
\rho_\tau(x;q) \equiv | \tau -\e1_{x < 0 }| \cdot |x|^q, \qquad x \in \R.
\end{equation}
  For $q=2$ we find the expectile regression studied in Section \ref{Sect_expectile}.  Remark also that if $q=1$    then we obtain the quantile function, which will not be considered in this section because this case was considered by \cite{Ciuperca.2016}. The properties of function (\ref{def_Lq}) can be found in \cite{Daouia.19} where it is specified for example that the choice $q \in (1,2)$ is preferable  for data with outliers in order to take into account both the robustness of the quantile method and the sensitivity of that  expectile. 
For $t \in \R$, consider the function: $\psi_\tau(t;q) \equiv \rho_\tau(\varepsilon -t;q)$ for which we denote: $g_\tau(\varepsilon; q) \equiv \frac{\partial \psi_\tau}{\partial t}(0;q) = - q \tau | \varepsilon|^{q-1} \e1_{\varepsilon \geq 0} + q(1-\tau) | \varepsilon|^{q-1} \e1_{\varepsilon <0}$ and $h_\tau(\varepsilon; q) \equiv \frac{\partial^2 \psi_\tau}{\partial t^2}(0;q) = q (q-1) \tau | \varepsilon|^{q-2} \e1_{\varepsilon \geq 0} + q (q-1)(1-\tau) | \varepsilon|^{q-2} \e1_{\varepsilon <0}$. \cite{Hu.Chen.2021} proved that $\mu_{h_{\tau}(q)} \equiv\eE[h_{\tau}(\varepsilon;q)] < \infty $ and $\sigma^2_{g_{\tau}(q)} \equiv\Var[g_{\tau}(\varepsilon;q)] < \infty$.   \\
For the asymmetric $L_q$-norm, relation (\ref{E2.2}) becomes 
\begin{equation*}
\label{E2.2q}
  \mathcal{G}_n(\eb;q) \equiv  \sum_{i=1}^n \rho_{\tau}(Y_i-\mathbb{X}_i^\top\eb;q), \quad \widetilde{\eb}_n(q) \equiv \underset{\eb \in \mathbb{R}^r}{\argmin}   \; \mathcal{G}_n(\eb;q)
\end{equation*}
and the corresponding penalized process is:
\begin{equation}
\label{upi}
\mathcal{Q}_n(\eb;q) \equiv n^{-1}\sum_{i=1}^n \rho_{\tau}(Y_i-\mathbb{X}_i^\top\eb;q)+\lambda_n \sum_{j=1}^p \widehat{\omega}_{n;j}(q) \| \eb_j \|_2,
\end{equation}
 with the weights :
\begin{equation}
\label{u1}
\widehat{\omega}_{n;j}(q)\equiv\| \widetilde{\eb}_{n;j}(q) \|_2 ^{-\gamma}, \quad \gamma >0.
\end{equation}
Since the quantities $\argmin_{t \in \R} \eE[\rho_\tau(\varepsilon - t ; q)]$ are called $L_q$-quantiles (see \cite{Chen.1996}, \cite{Daouia.19}),   $\widetilde{\eb}_n(q)$ is called the $L_q$-quantile estimator. 
Then, we define the adaptive group LASSO $L_q$-quantile   estimator by:
$$
\widehat{\eb}_n(q) \equiv \underset{\eb \in \mathbb{R}^r}{\argmin}\mathcal{Q}_n(\eb;q).
$$
For $\tau=1/2$ and $q=2$ we find the adaptive group LASSO LS method proposed and studied by \cite{Zhang-Xiang-16}. Let us also underline that in respect to present paper, \cite{Hu.Chen.2021} considers an asymmetric  $L_q$ regression with ungrouped variables,  where the explanatory variable  selection is carried out with SCAD and adaptive LASSO penalties.\\
Now, in this section, for the model errors  $(\varepsilon_i)_{1 \leqslant i \leqslant n}$, we suppose the following assumption, considered also by \cite{Hu.Chen.2021} for $L_q$-quantile regression with ungrouped variables and which for $q=2$ becomes assumption (A1):
\begin{description}
	\item \textbf{(A1q)} The errors $(\varepsilon_i)_{1 \leqslant i \leqslant n}$ are i.i.d random variables with a continuous positive density in a neighborhood of zero and the $\tau$th $L_q$-quantile zero: $\eE[g_\tau(\varepsilon; q)]=0$. 
We also suppose $\mathbb{E}(|\varepsilon_i|^{2q})< \infty$.
\end{description}
As in Section \ref{Sect_expectile}, the value of $\tau$ will be fixed throughout this section, such that $\eE[g_\tau(\varepsilon; q)]=0$. \\
The true non-zero parameters satisfy the following assumption:
\begin{description}
	\item \textbf{(A5q)} $\min_{1 \leqslant j \leqslant p^0} \|\eb^0_j \|_2=h_0>K>0$.
\end{description}
 Assumption   (A5q) is a particular case of assumption (A5) for $\alpha=0$ and it supposes that the norm $L_2$ of non-zero groups does not depend on $n$. 
 The proofs of the results presented in the following two subsections are postponed in Subsection \ref{proof_Lq}. 
\subsection{Case $p=O(n^c)$, $c \in [0,1/2)$}
\label{subsect_Lq1}
As in Section \ref{Sect_expectile}, we first consider   $c \in [0,1/2)$, with the particular case $c=0$ when $p$ is fixed. First of all, let's underline that, by  elementary calculations, we have, for $t \rightarrow 0$:
\begin{equation}
\label{eq2Lq}
\eE[\rho_\tau(\varepsilon -t;q) - \rho_\tau(\varepsilon;q)]= \mu_{h_\tau(q)} \cdot \frac{t^2}{2} +o(t^2)
\end{equation}
As for the penalized expectile method presented in Subsection \ref{subsect_exp_pvar} when $p=O(n^c)$, with $c \in (0,1/2)$, the tuning parameter sequence $(\lambda_n)_{n \in \N}$ satisfies conditions (\ref{E4.1}) and (\ref{E4.2}). 
In order to study the oracle properties of $\widehat{\eb}_n(q)$, we must first find the   convergence rate of the $L_q$-quantile estimator $\widetilde{\eb}_n(q)$ and afterwards of $\widehat{\eb}_n(q)$. For a model with ungrouped variables, $c \in (0,1)$ and the  design satisfying (\ref{tl}), \cite{Hu.Chen.2021} prove that the   convergence rate of $ \widehat{\eb}_n(q)$ is of order $(p/n )^{1/2}$, in exchange the rate of $\widetilde{\eb}_n(q)$ is not established. 
 \begin{theorem} 
	\label{Th2.1_CSDA}
	Under assumptions (A1q), (A2), (A3), (A4), we have:\\
	(i) $\| \widetilde{\eb}_n(q)-\ebo \|_1=O_{\PP}\pth{ (p/n )^{1/2}}$.\\
	(ii) If the  tuning parameter sequence $(\lambda_n)_{n \in \N}$ satisfies  (\ref{E4.1}) and assumption (A5q) is also fulfilled, then,   $\| \widehat{\eb}_n(q)- \ebo\|_1=O_{\PP}\pth{(p/n )^{1/2}}$.
\end{theorem}
\noindent We observe that these convergence rates don't depend on $q$.  With this two results we can now prove the sparsity property and afterwards the asymptotic normality of $\widehat{\eb}_n(q)$.
\begin{theorem} 
	\label{T4Lq1}
	Under assumptions (A1q), (A2), (A3), (A4), (A5q), conditions (\ref{E4.1}) and (\ref{E4.2}) on $(\lambda_n)_{n \in \N}$, we have: $	\underset{n \rightarrow \infty}\lim \mathbb{P}[\widehat{\mathcal{A}}_n=\mathcal{A}]=1$.
\end{theorem}
\begin{theorem} 
	\label{T5Lq1}
	Under assumptions of Theorem \ref{T4Lq1}, for all $r^0$-vector $\textbf{u} \in \mathbb{R}^{r^0}$ such that $\| \textbf{u} \|_2 =1$,  we have: $
	n^{1/2}(\textbf{u}^\top \eU_{n,\mathcal{A}}^{-1} \textbf{u})^{-1/2} \textbf{u}^\top (\widehat{\eb}_n(q)-\eb^0)_{\mathcal{A}} \overset{\mathcal{L}}{\underset{n \rightarrow \infty}{\longrightarrow}} \mathcal{N}\big(0, \sigma^2_{g_{\tau}(q)} \mu^{-2}_{h_{\tau}(q)} \big)$.
\end{theorem}
\subsection{Case $p=O(n^c)$, $c \in [1/2, 1]$}
\label{subsect_Lq2}
In this subsection, the number $p$ of groups satisfies  the following assumption:
\begin{description}
	\item \textbf{(A6)}: $p=O(n^c)$ with $c \in [1/2,1]$ and $r \leq n$.
\end{description}
In order to study the oracle properties of the estimator $\widehat{\eb}_n(q)$ we should first know the convergence rate  of $\widetilde{\eb}_n(q)$. 
We will prove that the convergence rate of the estimators  $ \widetilde{\eb}_n(q)$ and $\widehat{\eb}_n(q)$  is  strictly slower than $n^{1/2}$.  
\begin{lemma} 
	\label{Lemma3.1_CSDA}
	Under assumptions (A1q),  (A2), (A3), (A6), we have that $\| \widetilde{\eb}_n(q)- \ebo\|_1=O_{\PP}\pth{a_n}$, with the sequence $(a_n)_{n \in \N}$ such that $a_n \rightarrow 0$ and $n^{1/2} a_n \rightarrow \infty$.
\end{lemma}
\noindent  For the weights $\widehat{\omega}_{n;j}(q)$ of relation (\ref{upi}) we can take either those specified by (\ref{u1}) or by generalizing to the case of grouped explanatory variables, the weights proposed by \cite{Ciuperca.2021} in the ungrouped variable case:
\begin{equation*}
\label{u2}
\widehat{\omega}_{n;j}(q)= \min \big(\| \widetilde{\eb}_{n;j}(q) \|_2 ^{-\gamma}, n^{1/2}\big).
\end{equation*}
\begin{theorem} 
	\label{Theorem3.1_CSDA}
	Under assumptions  (A1q), (A2), (A3), (A5q), (A6) the  tuning parameter $(\lambda_n)_{n \in \N}$ and sequence $(b_n)_{n \in \N} \rightarrow 0$  satisfying , $n^{1/2} b_n \rightarrow \infty$, $\lambda_n (p^0)^{1/2} b^{-1}_n   \rightarrow 0$, as $n \rightarrow \infty$, we have,   $\| \widehat{\eb}_n(q)- \ebo\|_1=O_{\PP}\pth{b_n}$.
\end{theorem}
\noindent The convergence rate of $\widehat{\eb}_n(q)$ depends  on the choice of the tuning parameter $\lambda_n$ but also on the number of non-zero groups.
\noindent Now that we know the convergence rates of these two estimators, we can study the oracle properties of $\widehat{\eb}_n(q)$.
\begin{theorem} 
	\label{Theorem3.2_CSDA_L2} Suppose that assumptions   (A1q), (A2), (A3),   (A5q), (A6) hold, the  tuning parameter $(\lambda_n)_{n \in \N}$ and sequence $(b_n)_{n \in \N} \rightarrow 0$  satisfy  $n^{1/2} b_n \rightarrow \infty$, $ (p^0)^{1/2} b^{-1}_n \lambda_n  \rightarrow 0$ and $a_n^{-\gamma} b_n^{-1}\lambda_n \rightarrow \infty$,  as $n \rightarrow \infty$.  Then:\\
	(i)  $\PP \big[\widehat{\cal A}_n = {\cal A}\big]\rightarrow 1$,  as $n \rightarrow \infty $.\\
	(ii)   For any  vector $\eu$ of size $r^0$ such that $\| \eu\|_1=1$,   we have:  $n^{1/2} (\eu^\top \eU^{-1}_{n,{\cal A}} \eu)^{-1/2} \eu^\top ( \widehat{\eb}_n(q) - \eb^0)_{\cal{A}}  \overset{\cal L} {\underset{n \rightarrow \infty}{\longrightarrow}} {\cal N}\big(0,  {\sigma^2_{g_{\tau}(q)}}{\mu^{-2}_{h_{\tau}(q)}} \big)$.
\end{theorem}
 \section{Simulation study}
 \label{Sect_simu}
 In this section, we study by Monte Carlo simulations our adaptive group LASSO expectile estimator and compare it with the adaptive group LASSO quantile estimator, in terms of sparsity and accuracy. 
 All simulations will be performed using R language.  Two scenarios are considered: ungrouped and grouped variables. Moreover, every time, $p$ will be fixed and afterwards varied with $n$. As specified in Section \ref{Sect_expectile}, the value of $\tau$ is fixed and it must satisfy the condition:
  \begin{equation}
 \tau=\frac{\mathbb{E}(\varepsilon \e1_{\varepsilon<0})}{\mathbb{E}(\varepsilon ( \e1_{\varepsilon<0}-\e1_{\varepsilon>0}))}.
 \label{tau}
 \end{equation}
 
 \subsection{Models with ungrouped variables ($r=p$)}
 In this subsection, the linear models are with ungrouped variables, that is $d_j=1$ for any $j=1, \cdots p$. The used R language packages   are \textit{SALES} with function \textit{ernet} for expectile regression and \textit{quantreg} with function \textit{rq} for quantile regression. Based on relation (\ref{tau}), all simulations will be preceded by an estimation of $\tau$, depending on the distribution of the model error $\varepsilon$. 
 \subsubsection{Fixed $p$ case}
 \label{cas_ok2}
 \textbf{Parameters choice.} 
 Taking  into account (\ref{E3.1})    we consider:
 $
 \lambda_n=n^{ -1/2-  \gamma/4}$, with  $\gamma \in (0,1)$. 
 We choose $p^0=5,\mathcal{A}=\{1,...,5\}$ and: $
 \bm{\beta}_1^0=1$, $\bm{\beta}_2^0=-2$, $ \bm{\beta}_3^0=0.5$, $ \bm{\beta}_4^0=4$,  $\bm{\beta}_5^0=-6,$,  $\bm{\beta}_j^0=0$ for all $j > p^0. $
 For the model errors $\varepsilon$,  three distributions are considered: $\mathcal{N}(0,1)$ which is symmetrical, $\mathcal{E}xp(-1)$ and $\mathcal{N}(1.2,0.4^2)+\chi^2(1)$, the last two being asymmetrical.
 The  explanatory variables   are of normal standard distribution. By  $M=1000$ Monte Carlo replications,  the adaptive LASSO expectile estimator (\textit{ag}\_$\mathcal{E}$) is compared with the adaptive LASSO quantile estimator (\textit{ag}\_$\mathcal{Q}$). 
 For  \textit{ag}\_$\mathcal{Q}$,  the tuning parameter is of order  $n^{-3/5}$ and the value of the power in the weight of the penalty is 1.225 (see \cite{Ciuperca.2021}). 
 \newline
 \textbf{Results.}
 Looking at the sparsity property,   two cardinalities are calculated: 
  $Card(\mathcal{A} \cap \widehat{\mathcal{A}}_n)$ which is the number of true non-zeros estimated as non-zero,   $Card(\widehat{\mathcal{A}}_n \backslash \mathcal{A})$ which is the number of false non-zeros.
 Note that for a perfect estimation method we should get $Card(\mathcal{A} \cap \widehat{\mathcal{A}}_n)=p^0=5$ and $Card(\widehat{\mathcal{A}}_n \backslash \mathcal{A})=0$. Table \ref{Table1} presents these two cardinalities.
 Looking at $Card(\widehat{\mathcal{A}}_n \backslash \mathcal{A})$,   \textit{ag}\_$\mathcal{E}$ shows better performance than   \textit{ag}\_$\mathcal{Q}$, especially when $p =O(n)$. Note also that the accuracy of the evaluation of $Card(\widehat{\mathcal{A}}_n \backslash \mathcal{A})$ rises with $\gamma$ and concerning  $Card(\mathcal{A} \cap \widehat{\mathcal{A}}_n)$,  \textit{ag}\_$\mathcal{Q}$ is more accurate when $p \ll n$, even if the results for  \textit{ag}\_$\mathcal{E}$ are very close. However, when $p$ is greater and the error distribution asymmetrical,   \textit{ag}\_$\mathcal{E}$ provide better estimates. Notice also that the number of the true non-zeros estimated as non-zero decrease when $\gamma$ increase.
 \subsubsection{Case when $p$ is depending on {n}}
 \label{sect_pn}
 \textbf{Parameters choice.}   
The numbers $p$ and $p^0$ are calibrated in two ways: firstly 
$p=\lfloor  n/2 \rfloor $, $p^0=2 \lfloor n^{1/2} \rfloor $ and then $c=1$ , afterwards 
 $p=\lfloor  n(\log{n})^{-1} \rfloor $,   $p^0=2 \lfloor n^{1/4} \rfloor $ and then $c<1$. 
 For all $ j=1, \cdots,  p^0$,   $\bm{\beta}_j^0  \sim \mathcal{N}(0,2)$.
 The design, model errors are similar to those when $p$ is fixed.  
 \newline
 \textbf{Results.}  
 The sparsity is studied by calculating:  $(100/{p^0}) Card(\mathcal{A} \cap \widehat{\mathcal{A}}_n)$ which is the percentage of the   true non-zero parameters  estimated as non-zero, 
  $(100/(p-p^0)) Card(\widehat{\mathcal{A}}_n \backslash \mathcal{A})$ which is the percentage of the number of false non-zeros.
 For a perfect estimation method we should get $(100 /{p^0}) Card(\mathcal{A} \cap \widehat{\mathcal{A}}_n)=100$ and $(100/(p-p^0)) Card(\widehat{\mathcal{A}}_n \backslash \mathcal{A})=0$.  Denoting by $\widehat{\bm{\beta}}^{(m)}_{n} $ the estimation obtained for the $m$th Monte Carlo replication, we also calculate:  \textit{mean}($| \widehat{\bm{\beta}}_n - \bm{\beta}^0| ) = (Mp)^{-1} \sum_{m=1}^M \sum_{j=1}^p $ $\| \widehat{\bm{\beta}}^{(m)}_{n;j} - \bm{\beta}_j^0 \|_1 $ the accuracy of the complete estimation vectors and \textit{mean}($| (\widehat{\bm{\beta}}_n - \bm{\beta}^0)_{\mathcal{A}}| $)= $(Mp)^{-1} \sum_{m=1}^M \sum_{j=1}^{p^0} \| \widehat{\bm{\beta}}^{(m)}_{n;j} - \bm{\beta}_j^0 \|_1$ the accuracy of the estimation of the non-zero parameters.
 Remark that in Table \ref{Table2}, we are in the case where $p = O(n) $.
 Looking at the sparsity property,  \textit{ag}\_$\mathcal{Q}$ is slightly better for symmetrical error ($\mathcal{N}(0,1)$) but   \textit{ag}\_$\mathcal{E}$ is better when the distributions are asymmetrical. When $n$ is greater, there is no significant difference between the two methods. We can notice a connection between the accuracy and the sparsity property. Precisely, the more an estimator can select variable efficiently, the less it is accurate on the estimation of the non-zeros parameters.   \textit{ag}\_$\mathcal{E}$ tends to be more precise for smaller value of $n$.  Another case is presented in Table \ref{Table3}. Here,   \textit{ag}\_$\mathcal{E}$ always has a better sparsity property and   \textit{ag}\_$\mathcal{Q}$ is always more precise.
  \begin{table}[h!]
   	\begin{center}
  {\tiny  	
 	\begin{tabular}{ccc|cc|c|ccc|c|}  		\hline 
 		$\varepsilon$ & n & p & \multicolumn{3}{c|}{$Card(\mathcal{A} \cap \widehat{\mathcal{A}}_n)$}& \quad & \multicolumn{3}{c|}{$Card(\widehat{\mathcal{A}}_n \backslash \mathcal{A})$} \\ \hline 
 		\quad & \quad & \quad  & \multicolumn{2}{c|}{\textit{ag}\_$\mathcal{E}$} & \textit{ag}\_$\mathcal{Q}$ & \quad & \multicolumn{2}{c|}{\textit{ag}\_$\mathcal{E}$} & \textit{ag}\_$\mathcal{Q}$ \\ 
 		\quad & \quad & \quad  & $\gamma=5/8$ & $\gamma=11/12$ & \quad & \quad & $\gamma=5/8$ & $\gamma=11/12$ & \quad \\ \hline 
 
 		 $\mathcal{N}(0,1)$ &  50
 		& 10 & 4.999 & 4.995 & 5 & \quad & 0.007 & 0.001 & 0.221\\    
 		&& 25 & 4.998 & 4.989 & 4.999 & \quad & 0.034 & 0.001 & 1.326\\
 		&& 50 & 4.91 & 4.85 & 4.897 & \quad & 1.596 & 2.668 & 15.97\\ \cline{2-10}
 		 		
 		& 100
 		& 10 & 5 & 5 & 5 & \quad & 0 & 0 & 0.127\\
 		&& 25 & 5 & 5 & 5 & \quad & 0.002 & 0 & 0.573\\    
 		&& 100 & 4.973 & 4.901 & 4.912 & \quad & 1.438 & 2.937 & 49.28\\ \cline{2-10}
 				
 		& 200
 		& 10 & 5 & 5 & 5 & \quad & 0 & 0 & 0.048\\
 		&& 100 & 5 & 5 & 5 & \quad & 0 & 0 & 1.764\\    
 		&& 200 & 4.986 & 4.96 & 4.949 & \quad & 0.333 & 0.557 & 109.4\\ 	\hline 
 	 $\mathcal{N}(-1.2,0.4^2) +\chi^2(1)$ &  50
 		& 10 & 4.984 & 4.972 & 5 & \quad & 0.043 & 0.019 & 0.133\\ 
 	 && 25 & 4.97 & 4.958 & 4.994 & \quad & 0.284 & 0.194 & 1.242\\
 		&& 50 & 4.87 & 4.762 & 4.831 & \quad & 5.984 & 8.134 & 24.73\\ \cline{2-10}
 	&  100
 		& 10 & 5 & 5 & 5 & \quad & 0.005 & 0 & 0.035\\
 		&& 25 & 5 & 4.997 & 5 & \quad & 0.089 & 0.039 & 1.452\\    
 		&& 100 & 4.924 & 4.87 & 4.892 & \quad & 5.927 & 8.723 & 57.23\\ \cline{2-10}
 	&  200
 		& 10 & 5 & 5 & 5 & \quad & 0 & 0 & 0.013\\
 		&& 100 & 5 & 5 & 5 & \quad & 0.01 & 0 & 1.337\\    
 		&& 200 & 4.961 & 4.929 & 4.938 & \quad & 0.861 & 1.54 & 122.2\\ \hline
 	 		
 	 $\mathcal{E}xp(-1)$ & 50
 		& 10 & 4.999 & 4.993 & 5 & \quad & 0 & 0 & 0.033\\    
 		&& 25 & 4.995 & 4.986 & 4.999 & \quad & 0.008 & 0.005 & 0.558\\
 		&& 50 & 4.906 & 4.822 & 4.871 & \quad & 2.225 & 4.033 & 21\\ \cline{2-10}
  		
 		& 100
 		& 10   & 5     & 5     & 5     & \quad & 0     & 0     & 0.005\\
 		&& 25  & 5     & 5     & 5     & \quad & 0     & 0     & 0.073\\    
 		&& 100 & 4.945 & 4.902 & 4.914 & \quad & 0.883 & 1.619 & 48.48\\ \cline{2-10}
 		 		
 		&  200 & 10 & 5 & 5 & 5 & \quad & 0 & 0 & 0.001\\
 		&& 100 & 5 & 5 & 5 & \quad & 0 & 0 & 0.419\\ 
 		&& 200 & 4.979 & 4.971 & 4.939 & \quad & 0.12 & 0.21 & 106.3 \\ 	\hline
  	\end{tabular}
  	\caption{\small Sparsity study of the adaptive LASSO expectile estimator (\textit{ag}\_$\mathcal{E}$) and of the adaptive LASSO quantile estimator (\textit{ag}\_$\mathcal{Q}$) when explanatory variables are ungrouped and $p^0=5$.}
  		\label{Table1}
}
\end{center}
  \end{table}
   
 \begin{table}[h!]
 \begin{center}
{\tiny  
 	\begin{tabular}{cc|cc|c|cc|c|cc|c|cc|c|}
 		\hline 
 		\quad & \quad  & \multicolumn{3}{|c|}{100 ${p^0}^{-1} Card(\mathcal{A} \cap \widehat{\mathcal{A}}_n)$}&  \multicolumn{3}{c|}{100$(p-p^0)^{-1} Card(\widehat{\mathcal{A}}_n \backslash \mathcal{A})$} &  \multicolumn{3}{c|}{mean($| \widehat{\bm{\beta}}_n - \bm{\beta}^0|$ )} &  \multicolumn{3}{c|}{mean($| (\widehat{\bm{\beta}}_n - \bm{\beta}^0)_{\mathcal{A}}| $)} \\   
 		 
 		  &  & \multicolumn{2}{c|}{\textit{ag}\_$\mathcal{E}$} & \textit{ag}\_$\mathcal{Q}$ & \multicolumn{2}{c|}{\textit{ag}\_$\mathcal{E}$} & \textit{ag}\_$\mathcal{Q}$ & \multicolumn{2}{c|}{\textit{ag}\_$\mathcal{E}$} & \textit{ag}\_$\mathcal{Q}$ & \multicolumn{2}{c|}{\textit{ag}\_$\mathcal{E}$} & \textit{ag}\_$\mathcal{Q}$  \\  
 		 		 		
 		$\varepsilon$ & n & $\gamma=\frac{5}{8} $ & $\gamma=\frac{11}{12}$ & \quad &$\gamma=\frac{5}{8} $ & $\gamma=\frac{11}{12}$  & \quad &$\gamma=\frac{5}{8} $ & $\gamma=\frac{11}{12}$  & \quad & $\gamma=\frac{5}{8} $ & $\gamma=\frac{11}{12}$  & \quad  \\ \hline 
 		
 		\multirow{1}{1.5cm}{$\mathcal{N}(0,1)$} & 50 &  100 & 100 & 98.21 &0.154 &0.12 &0.009 &0.196 &0.223 &0.347 &0.350 &0.399 & 0.619 	\\    
 		& 100 &  100 & 99.99 & 100 &0 &0 &0 &0.161 &0.15 &0.109 &0.402 &0.375 &0.289 \\
 		& 400 &  99.99 & 98 &100 &0 &0 &0 &0.086 &0.077 &0.031 &0.43 &0.384 &0.154 \\ 	\hline 
 		
 		\multirow{1}{1.5cm}{$\mathcal{N}(-1.2,0.4^2) \newline +\chi^2(1)$} & 50 &  100 & 100 &93.05 &1.03 &0.509 &0.045 &0.215 &0.196 &0.595 &0.364 &0.384 & 1.063 
 		\\    
 		& 100 &  100 & 100 &99.5 &0.003 &0 &0 &0.166 &0.149 &0.196 &0.414 &0.37 &0.491 \\
 		& 400 &  100 & 100 &99.57 &0 &0 &0 &0.077 &0.083 &0.079 &0.385 &0.414 &0.393 \\ \hline
 		 		
 		\multirow{1}{1.5cm}{$\mathcal{E}xp(-1)$} & 50 &  100 & 100 &96.01 &0.045 &0.009 &0 &0.237 &0.217 &0.456 &0.424 &0.387 &0.814 
 		\\    
 		& 100 &  100 & 100 &99.96 &0 &0 &0 &0.17 &0.153 &0.094 &0.425 &0.382 &0.235 \\
 		& 400 &  100 & 98.9 &100 &0 &0 &0 &0.078 &0.073 &0.022 &0.389 &0.366 &0.109 \\	\hline
 	 	\end{tabular}
  	\caption{\small Sparsity study of   \textit{ag}\_$\mathcal{E}$ and of  \textit{ag}\_$\mathcal{Q}$ when explanatory variables are ungrouped,    $p=\lfloor  n/2 \rfloor $, $p^0=2 \lfloor \sqrt{n} \rfloor $.}
  	  	\label{Table2}
}
\end{center}
 \end{table}

 \begin{table}[h!]
\begin{center}
 {\tiny   
 	\begin{tabular}{cc|cc|c|cc|c|cc|c|cc|c|}
 	\hline 
 	\quad & \quad  & \multicolumn{3}{|c|}{100 ${p^0}^{-1} Card(\mathcal{A} \cap \widehat{\mathcal{A}}_n)$}&  \multicolumn{3}{c|}{100$(p-p^0)^{-1} Card(\widehat{\mathcal{A}}_n \backslash \mathcal{A})$} &  \multicolumn{3}{c|}{mean($| \widehat{\bm{\beta}}_n - \bm{\beta}^0|$ )} &  \multicolumn{3}{c|}{mean($| (\widehat{\bm{\beta}}_n - \bm{\beta}^0)_{\mathcal{A}}| $)} \\   
 	
 	&  & \multicolumn{2}{c|}{\textit{ag}\_$\mathcal{E}$} & \textit{ag}\_$\mathcal{Q}$ & \multicolumn{2}{c|}{\textit{ag}\_$\mathcal{E}$} & \textit{ag}\_$\mathcal{Q}$ & \multicolumn{2}{c|}{\textit{ag}\_$\mathcal{E}$} & \textit{ag}\_$\mathcal{Q}$ & \multicolumn{2}{c|}{\textit{ag}\_$\mathcal{E}$} & \textit{ag}\_$\mathcal{Q}$  \\  
 		
 	$\varepsilon$ & n & $\gamma=\frac{5}{8} $ & $\gamma=\frac{11}{12}$ & \quad &$\gamma=\frac{5}{8} $ & $\gamma=\frac{11}{12}$  & \quad &$\gamma=\frac{5}{8} $ & $\gamma=\frac{11}{12}$  & \quad & $\gamma=\frac{5}{8} $ & $\gamma=\frac{11}{12}$  & \quad  \\ \hline 
 		\multirow{1}{1.5cm}{$\mathcal{N}(0,1)$} & 50 &  100 & 100 & 100 &0.1 &0.013 &0.062 &0.139 &0.140 &0.073 & 0.416 &0.420 & 0.221
 		\\    
 		& 100 &  100 & 100 & 100 &0 &0 &0 & 0.115 &0.11 &0.049 &0.403 &0.383 &0.167 \\
 		& 400 &  100 & 99.3 &100 &0 &0 &0 &0.048 &0.043 &0.011 & 0.393 &0.354 &0.012 \\ \hline 
 		\multirow{1}{1.5cm}{$\mathcal{N}(-1.2,0.4^2) \newline +\chi^2(1)$} & 50 &  100 & 100 & 100 & 0.99 &0.425 &0.025 &0.137 &0.123 &0.074 &0.41 &0.37 & 0.222
 		\\    
 		& 100 &  100 & 100 & 100 &0
 		&0 &0 &0.113 &0.108 &0.04 &0.394 &0.38 &0.141 \\
 		& 400 &  100 & 100 & 100 &0 &0 &0 &0.049 &0.048 &0.008 &0.404 &0.4 &0.068 \\ 	\hline
 		\multirow{1}{1.5cm}{$\mathcal{E}xp(-1)$} & 50 &  100 & 100 & 100 &0.037 &0.05 &0 &0.124 &0.109 &0.049 &0.373 &0.328 &0.147
 		\\    
 		& 100 &  100 & 100 & 100 &0 &0 &0 &0.112 &0.105 &0.033 &0.391 &0.366 &0.115 \\
 		& 400 &  100 & 100 &100 &0 &0 &0 &0.047 &0.053 &0.007 &0.386 &0.438 &0.048 \\ 	\hline
  	\end{tabular}
  \caption{\small Sparsity study  of the adaptive LASSO expectile estimator (\textit{ag}\_$\mathcal{E}$) and of the adaptive LASSO quantile estimator (\textit{ag}\_$\mathcal{Q}$) when explanatory variables are ungrouped,    $p=\lfloor  n(\log{n})^{-1} \rfloor $, $p^0=2 \lfloor n^{1/4} \rfloor $.}
   	\label{Table3}
}
\end{center}
  \end{table}
 
 
 \subsubsection{Sparsity study depending on $\gamma$}
 \textbf{Parameters choice.}  
 We take $p \in \{10,100\}$ and $n=100$.
 Parameter $\bm{\beta}^0$ and the model errors $\varepsilon$ are similar as in the case of fixed  $p$.\\
  \noindent \textbf{Results.}  
 Results are presented for  \textit{ag}\_$\mathcal{E}$. In the case of symmetrical distribution of the error, Figures \ref{F1} and \ref{F2} show that the number of true non-zeros estimated as non-zero and the number of false non-zeros   decrease function of $\gamma$. Consequently, the choice of $\gamma$ will depend on the context. However, if $p \ll n$, taking $\gamma  $ close to 1 will always be the best option. 
\noindent In the case of asymmetrical distribution of the error, $Card(\mathcal{A} \cap \widehat{\mathcal{A}}_n)$ is still a decreasing function of $\gamma$ while  $Card(\widehat{\mathcal{A}}_n \backslash \mathcal{A})$ has a minimum value  (Figure \ref{F6}).  
 It would be interesting to choose $\gamma$ near this minimum. 
 We can still notice that, when $p \ll n$, the influence of  $\gamma$ is mainly on $Card(\widehat{\mathcal{A}}_n \backslash \mathcal{A})$ (Figure \ref{F5}).
 Finally,    $\gamma \simeq 0.6 $ is a good choice when $p=O(n)$ and for
 $p \ll n$ the choice of $\gamma  $ close to 1  is favorable.
 \begin{figure}[h!]
 	\begin{tabular}{cc}
 		\includegraphics[width=0.45\linewidth,height=3.5cm]{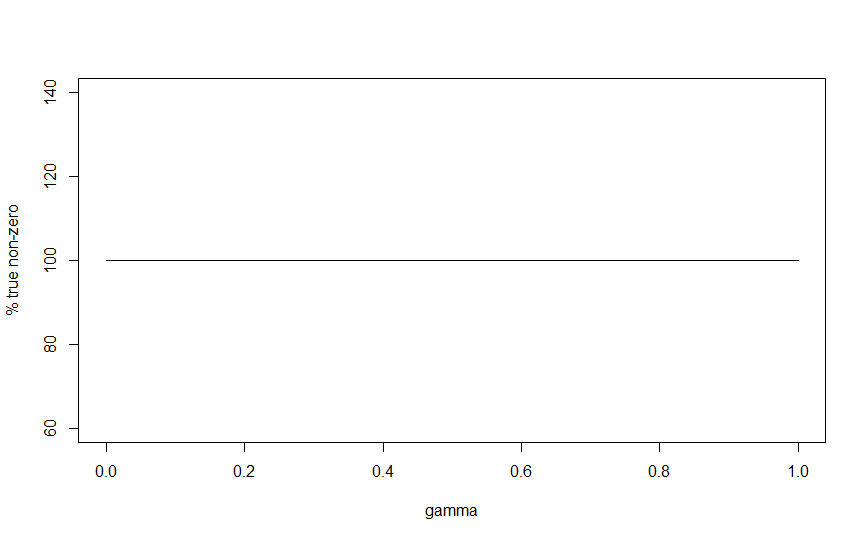} &
 		\includegraphics[width=0.45\linewidth,height=3.5cm]{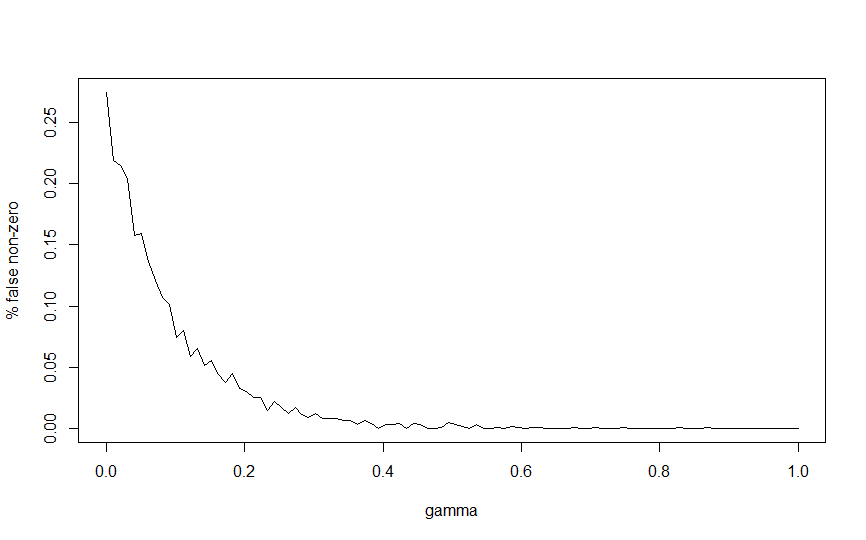} \\
 	{\small 	$(100 /{p^0})  Card(\mathcal{A} \cap \widehat{\mathcal{A}}_n)$} &
 	{\small 	$(100/(p-p^0) ) Card(\widehat{\mathcal{A}}_n \backslash \mathcal{A})$ }
 	\end{tabular}
 	\caption{\small Sparsity study of   \textit{ag}\_$\mathcal{E}$ depending on $\gamma$ when $\varepsilon \sim \mathcal{N}(0,1)$, $ p=10$, $n=100$.	}
 	\label{F1}
 \end{figure}
 \begin{figure}[h!]
 	\begin{tabular}{cc}
 		\includegraphics[width=0.45\linewidth,height=3.5cm]{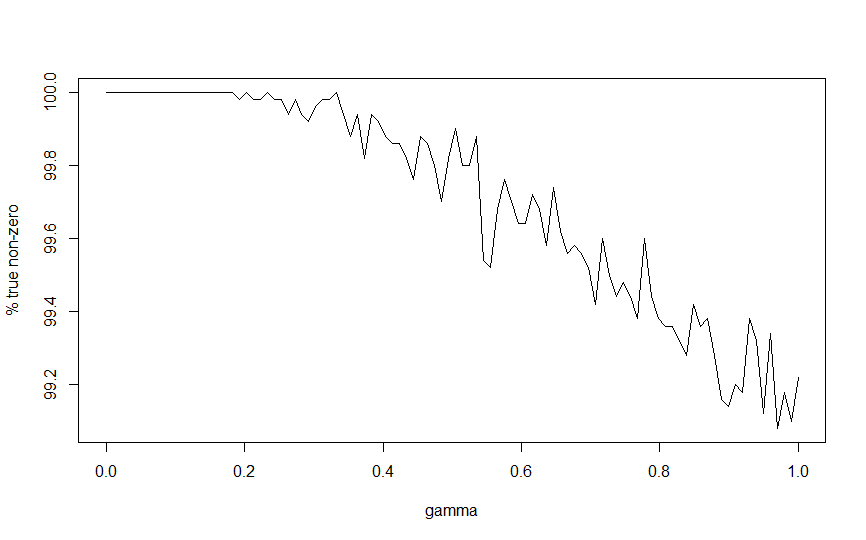} &
 		\includegraphics[width=0.45\linewidth,height=3.5cm]{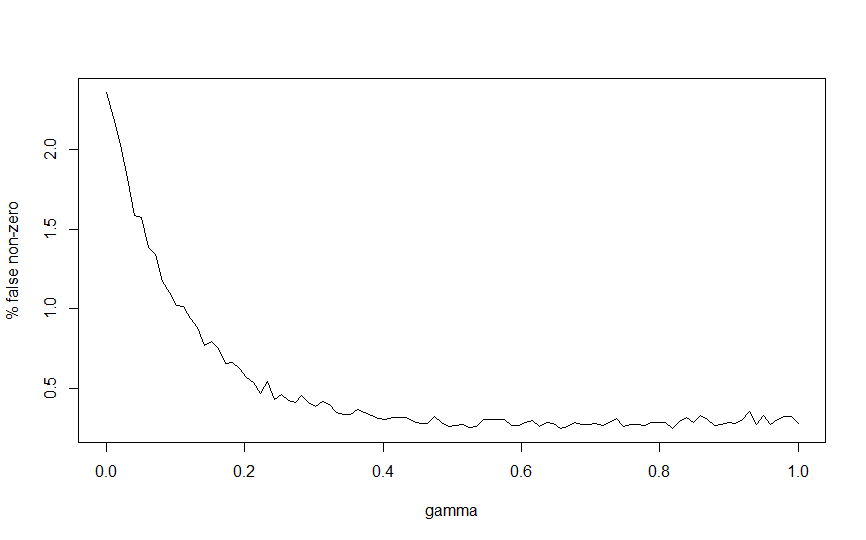} \\
 	{\small	$(100 / {p^0}) Card(\mathcal{A} \cap \widehat{\mathcal{A}}_n)$ }&
 	{\small	$(100/(p-p^0)) Card(\widehat{\mathcal{A}}_n \backslash \mathcal{A})$ } 
 	\end{tabular}
 	\caption{\small Sparsity study of   \textit{ag}\_$\mathcal{E}$  depending on $\gamma$ when $\varepsilon \sim \mathcal{N}(0,1)$, $ p=100$, $n=100$. }
 	\label{F2}
 \end{figure}
 \begin{figure}[h!]
 	\begin{tabular}{cc}
 		\includegraphics[width=0.45\linewidth,height=3.5cm]{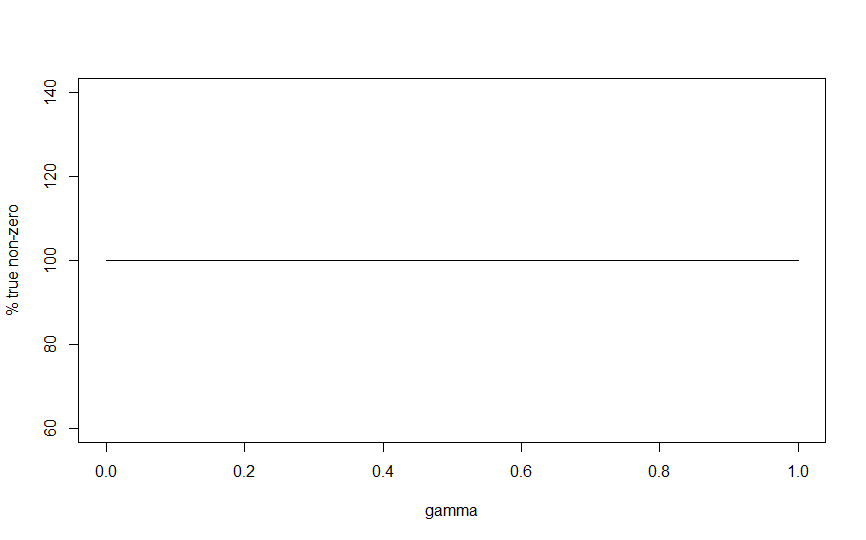} &
 		\includegraphics[width=0.45\linewidth,height=3.5cm]{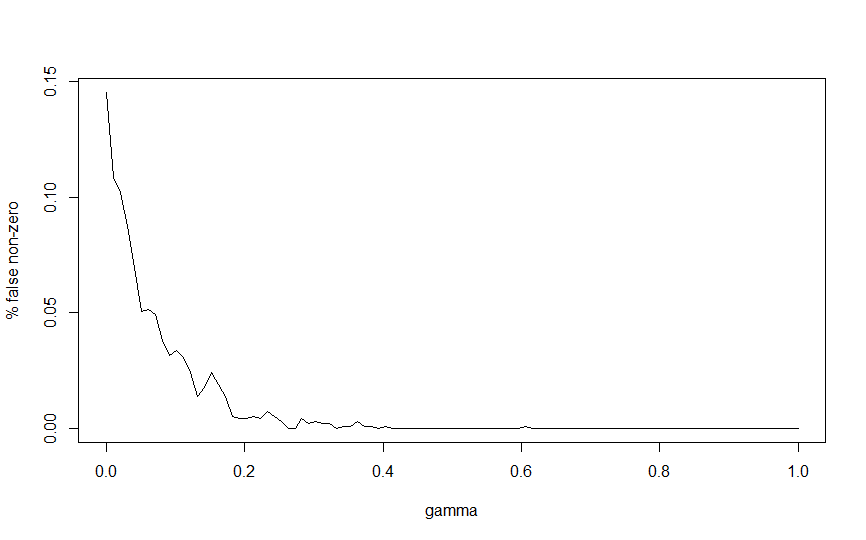} \\
 	{\small	$(100/ {p^0} ) Card(\mathcal{A} \cap \widehat{\mathcal{A}}_n)$} &
 	{\small	$(100/(p-p^0) ) Card(\widehat{\mathcal{A}}_n \backslash \mathcal{A})$ }
 	\end{tabular}
 	\caption{\small Sparsity study of   \textit{ag}\_$\mathcal{E}$ depending on $\gamma$ when $\varepsilon \sim \mathcal{E}xp(-1)$, $ p=10$, $n=100$.	}
 	\label{F5}
 \end{figure}
 \begin{figure}[h!]
 	\begin{tabular}{cc}
 		\includegraphics[width=0.45\linewidth,height=3.5cm]{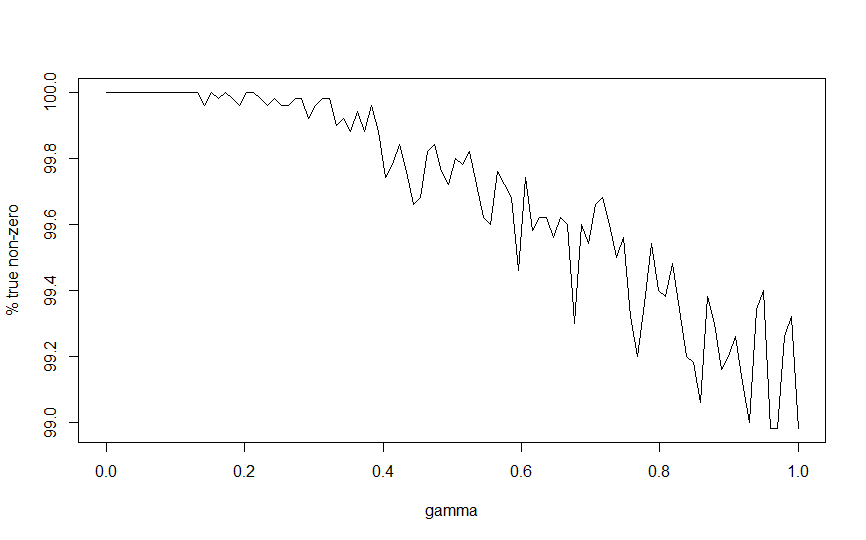} &
 		\includegraphics[width=0.45\linewidth,height=3.5cm]{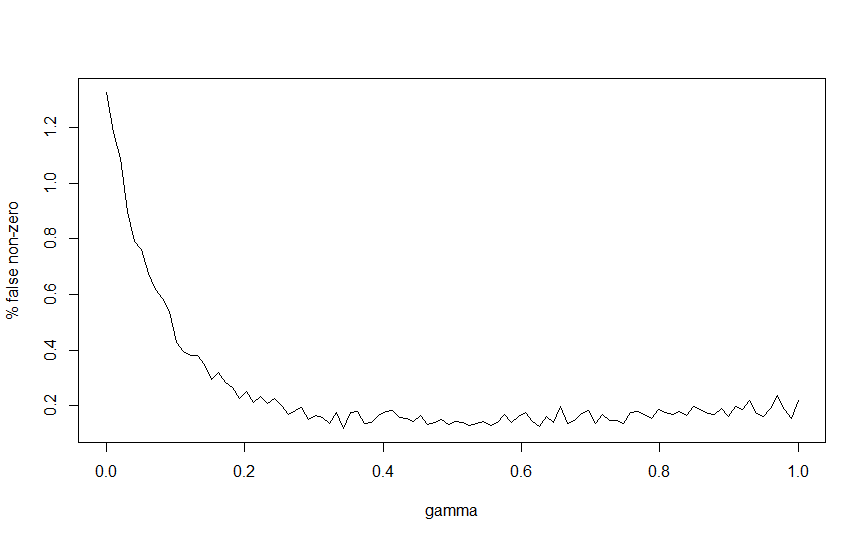} \\
 	{\small 	$(100 /{p^0}) Card(\mathcal{A} \cap \mathcal{A}_n)$ }&
 	{\small 	$(100/(p-p^0)) Card(\widehat{\mathcal{A}}_n \backslash \mathcal{A})$ }
 	\end{tabular}
 	\caption{\small Sparsity study of   \textit{ag}\_$\mathcal{E}$ depending on $\gamma$ when $\varepsilon \sim \mathcal{E}xp(-1)$, $ p=100$, $n=100$. 	}
 	\label{F6}
 \end{figure}
 \subsubsection{Effect of $\| \beta^0 \|_2$}
 Now quantify the effect of $\| \bm{\beta}^0 \|_2 $ on the  sparsity of \textit{ag}\_$\mathcal{E}$. \newline
 \textbf{Parameters choice.} 
 We take $\gamma=0.6$, $p=100$, $n=100$, while   the errors $\varepsilon$ are  similar as in the case of fixed  $p$. Concerning the model parameters, we choose: $  \bm{\beta}_1^0=v \cdot 10^{-2}$, $\bm{\beta}_j^0=0$  for all $ j>1$, 
 with $v>0$ which will be varied. 
 \newline
 \textbf{Results.}   For the true non-zeros, from Figures \ref{F7}, \ref{F8}, \ref{F9} and \ref{F10} we deduce     the value of $v$   for obtaining  a satisfactory  selection   as $p \simeq n $ and as the distribution of the error becomes asymmetrical. The effect on 100$(p-p^0)^{-1} Card(\widehat{\mathcal{A}}_n \backslash \mathcal{A})$ isn't relevant even if, when $p =O( n)$, this value seems to have a maximum for a symmetrical distribution of error.   Table \ref{Table4} presents the value of $\| \bm{\beta}^0 \|_2$ for which different values of $(100 /{p^0}) Card(\mathcal{A} \cap \widehat{\mathcal{A}}_n)$ are achieved. More precisely
 $_{99}\| \bm{\beta}^0 \|_2$ is the value of $\bm{\beta}^0$ for which $(100/p_0) Card(\mathcal{A} \cap \widehat{\mathcal{A}}_n)=99\% $ and  $_{95}\| \bm{\beta}^0 \|_2$ the value for which $(100/p_0) Card(\mathcal{A} \cap \widehat{\mathcal{A}}_n)=95\% $.
 \begin{figure}[h!]
 	\begin{tabular}{cc}
 		\includegraphics[width=0.45\linewidth,height=3.5cm]{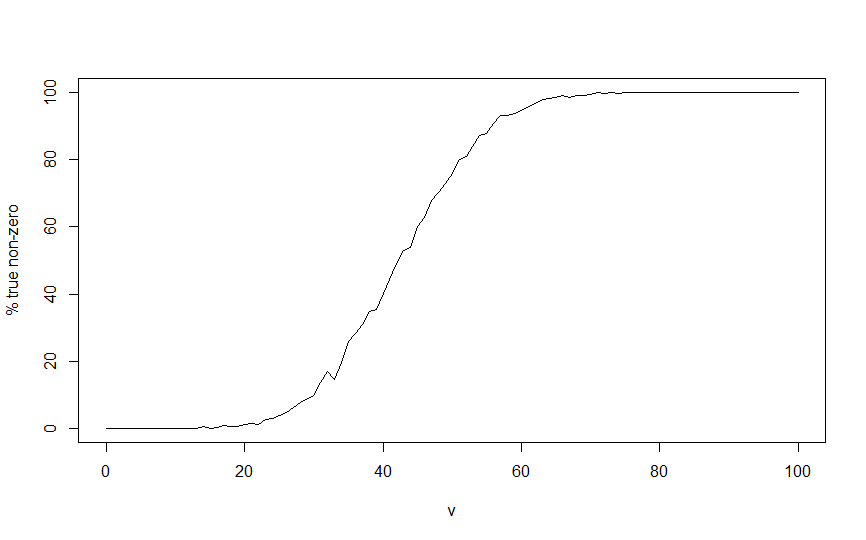} &
 		\includegraphics[width=0.45\linewidth,height=3.5cm]{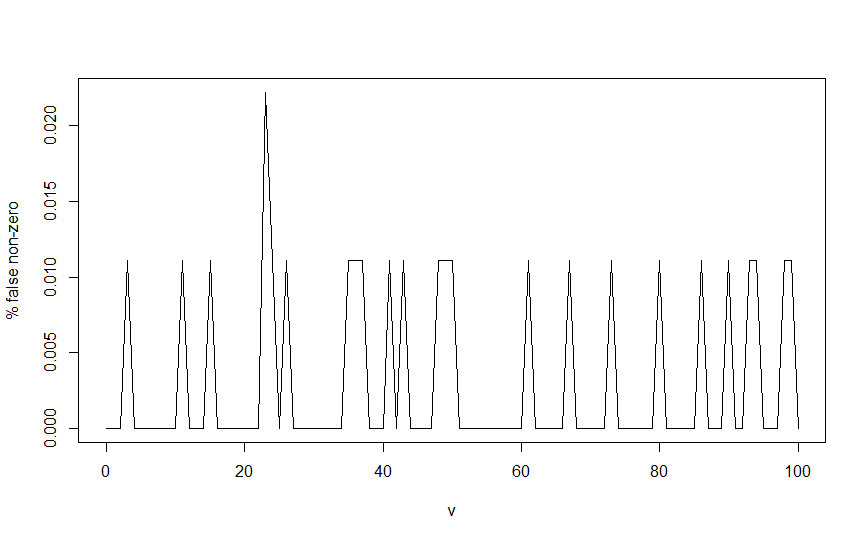} \\
 	{\small 	$(100/ {p^0} ){\normalsize } Card(\mathcal{A} \cap \widehat{\mathcal{A}}_n)$ }&
 	{\small 	$(100/(p-p^0)) Card(\widehat{\mathcal{A}}_n \backslash \mathcal{A})$ }
 	\end{tabular}
 	\caption{\small Sparsity study of  \textit{ag}\_$\mathcal{E}$ depending on $v$ when $\varepsilon \sim \mathcal{N}(0,1)$, $ p=10$, $n=100$. 	}
 	\label{F7}
 \end{figure}

 \begin{figure}[h!]
 	\begin{tabular}{cc}
 		\includegraphics[width=0.45\linewidth,height=3.5cm]{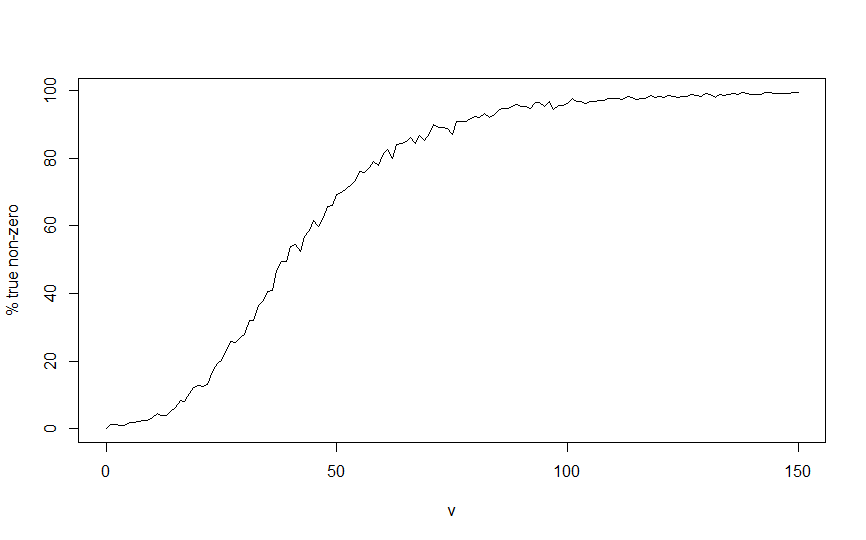} &
 		\includegraphics[width=0.45\linewidth,height=3.5cm]{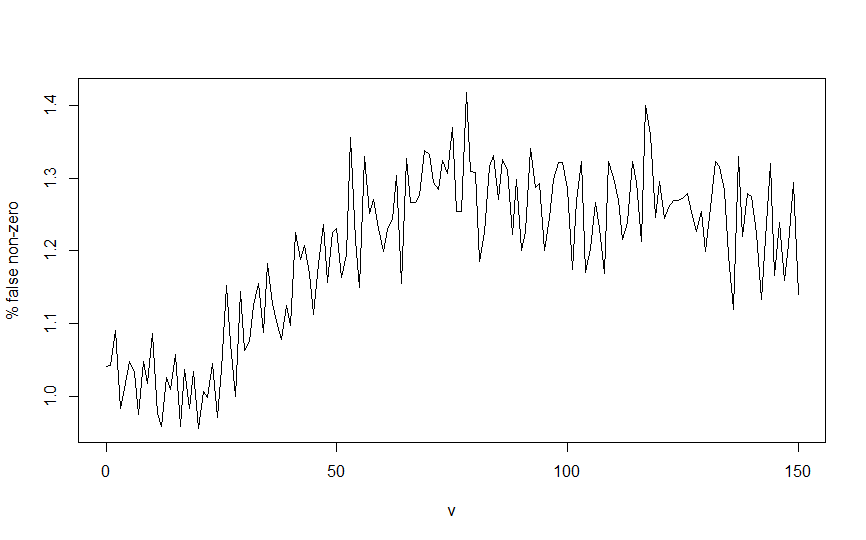} \\
 	{\small	$(100 /p_0) Card(\mathcal{A} \cap \widehat{\mathcal{A}}_n)$} &
 	{\small	$(100/(p-p_0)) Card(\widehat{\mathcal{A}}_n \backslash \mathcal{A})$ }
 	\end{tabular}
 	\caption{\small Sparsity study of \textit{ag}\_$\mathcal{E}$ depending on $v$ when $\varepsilon \sim \mathcal{N}(0,1)$, $ p=100$, $n=100$. }
 	\label{F8}
 \end{figure}

 \begin{figure}[h!]
 	\begin{tabular}{cc}
 		\includegraphics[width=0.45\linewidth,height=3.5cm]{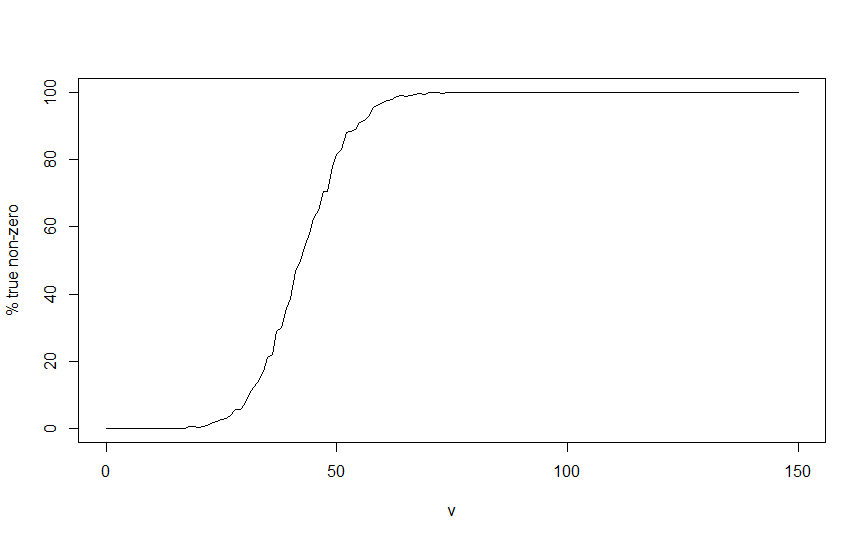} &
 		\includegraphics[width=0.45\linewidth,height=3.5cm]{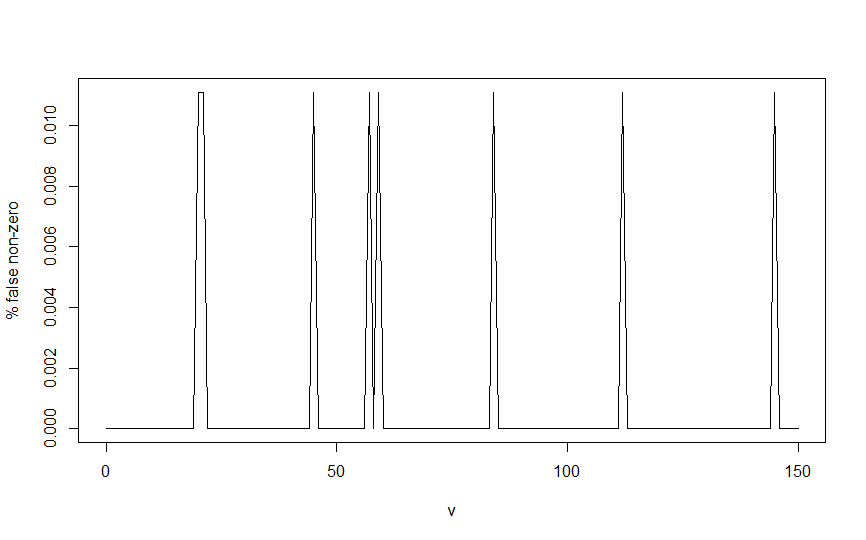} \\
 	{\small	$(100 /p_0) Card(\mathcal{A} \cap \widehat{\mathcal{A}}_n)$} &
 	{\small	$(100/(p-p_0)) Card(\widehat{\mathcal{A}}_n \backslash \mathcal{A})$ }
 	\end{tabular}
 	\caption{\small Sparsity study of   \textit{ag}\_$\mathcal{E}$ depending on $v$ when $\varepsilon \sim \mathcal{E}xp(-1)$, $ p=10$, $n=100$.	}
 	\label{F9}
 \end{figure}

 \begin{figure}[h!]
 	\begin{tabular}{cc}
 		\includegraphics[width=0.45\linewidth,height=3.5cm]{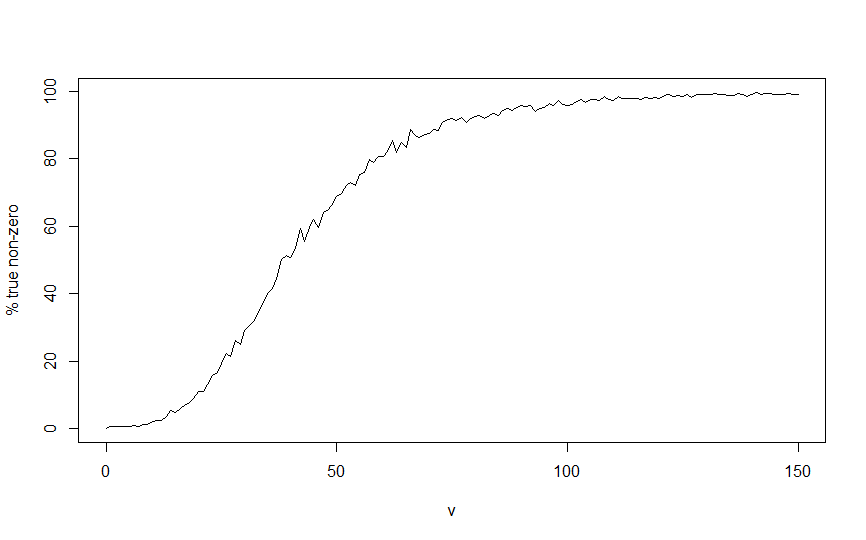} &
 		\includegraphics[width=0.45\linewidth,height=3.5cm]{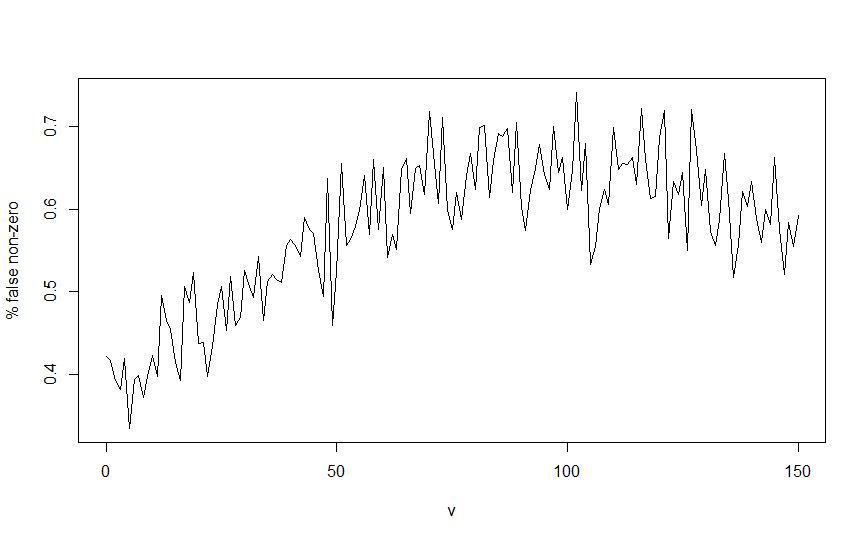} \\
 	{\small	$(100 /p_0) Card(\mathcal{A} \cap \widehat{\mathcal{A}}_n)$} &
 	{\small	$(100/(p-p_0)) Card(\widehat{\mathcal{A}}_n \backslash \mathcal{A})$ } 	  	\end{tabular}
 	\caption{\small Sparsity study of   \textit{ag}\_$\mathcal{E}$ depending on $v$ when $\varepsilon \sim \mathcal{E}xp(-1)$, $ p=100$, $n=100$. 	}
  	\label{F10}
 \end{figure}

 \begin{table}[h!]
\begin{center}
 {\tiny  
 	\begin{tabular}{cccc}
 		\hline 
 		$\varepsilon$ & p &  $_{99}\| \bm{\beta}^0 \|_2$ &  $_{95}\| \bm{\beta}^0 \|_2$ \\ 
 		\hline 
 		 {$\mathcal{N}(0,1)$} & 10
 		& 0.69 & 0.6\\    
 		& 100 &  1.19 &0.9\\ \hline 
 		 {$\mathcal{N}(-1.2,0.4^2)+\chi^2(1)$}  & 10
 		& 0.72 & 0.64\\    
 		& 100 & 1.59 & 0.99\\ 	\hline 
 		 {$\mathcal{E}xp(-1)$}  & 10
 		& 0.65 & 0.57\\    
 		& 100 & 1.24 &0.86 	\\	\hline
 	\end{tabular}
 	\caption{\small Effect of $\| \bm{\beta}^0 \|_2$ on sparsity of   \textit{ag}\_$\mathcal{E}$.}
 	 	\label{Table4}
}
\end{center}
  \end{table}
 
 \subsection{Models with grouped variables}
 In this subsection, the explanatory variables are grouped and  the expectile index is $\tau =1/2$. The R language packages   used are \textit{grpreg} with function \textit{grpreg} for expectile regression and \textit{rqPen} with function \textit{QICD.group} for quantile regression.
 \subsubsection{Fixed $p$ case}
 \label{cas_ok}
 \textbf{Parameters choice.} 
 We take $  \lambda_n=n^{-1/2- \gamma/4}$, $ \mathcal{A}=\{1,...,4\}$ and:  $ \bm{\beta}_1^0=(0.5,1,1.5,1,0.5)$, $\bm{\beta}_2^0=(1,1,1,1,1)$, $ \bm{\beta}_3^0=(-1,0,1,2,1.5)$, $\bm{\beta}_4^0=(-1.5,1,0.5,0.5,0.5)$,  $\bm{\beta}_j^0=\bm{0}$, for any  $j > p^0$.
  The   errors $\varepsilon$   follow   $\mathcal{N}(0,1)$ or    Cauchy   $\mathcal{C}(0,0.1)$ distributions while the explanatory variables are  normal standard distributed.
  For   \textit{ag}\_$\mathcal{Q}$,   the tuning parameter is of order $n^{-3/5}$ and the power  value in the weight of the penalty is 1.225 (see \cite{Ciuperca.2020}).\\
\noindent  \textbf{Results.} The same sparsity measurements as in Subsubsection \ref{sect_pn} are presented,  through 1000 Monte Carlo replications.
From Table \ref{Table5} we deduce in the case  $\varepsilon \sim \mathcal{N}(0,1)$ that   \textit{ag}\_$\mathcal{E}$ has a better sparsity property than   \textit{ag}\_$\mathcal{Q}$, especially when $n$ is small. 
Since Cauchy distribution has no mean, when $\varepsilon \sim \mathcal{C}(0,0.1)$, if $n$ is small, \textit{ag}\_$\mathcal{E}$ is still better for the true non-zeros, but worse for the false non-zeros ($Card(\widehat{\mathcal{A}}_n \backslash \mathcal{A})$) in either case.  If $n$ increases, the penalized quantile estimation is always better.
 
 \subsubsection{Case when $p$ is depending on $n$}
 \textbf{Parameters choice.} 
 We calibrate $p$ and $p^0$ in two ways: firstly 
 $p=\lfloor  n/5 \rfloor $,    $p^0=2 \lfloor 5^{-1}n^{1/2} \rfloor $ and then  $c=1$, afterwards 
  $p=\lfloor  n(2\log{n})^{-1} \rfloor $,  $p^0=2 \lfloor n^{1/4} \rfloor $ and then  $c<1$.   We consider  $\bm{\beta}_j^0  \sim \mathcal{N}(\bm{0}_5,2\bm{I}_5)$, for any $j= 1, \cdots ,    p_0$, with $\bm{I}_5$ the identity matrix of order 5. 
 All the other parameters are similar as in the case of fixed  $p$. \newline
 \noindent \textbf{Results.} 
   The same measurements of  sparsity and accuracy   as in Subsubsection \ref{sect_pn} are presented.
  The results of Table \ref{Table6}  are for the case $p =O( n)$.  
 When  $\varepsilon \sim \mathcal{N}(0,1)$,   \textit{ag}\_$\mathcal{E}$ is always better in the variable selection and more accurate for the two values of $\gamma$ considered: 1/10 and 9/10 and for all values of $n$. When $\varepsilon \sim \mathcal{C}(0,0.1)$, for $\gamma =1/10$ in adaptive weight of  \textit{ag}\_$\mathcal{E}$, the   results are similar for the two estimation methods, while for  $\gamma =9/10$ the obtained results by \textit{ag}\_$\mathcal{E}$ are better than by \textit{ag}\_$\mathcal{Q}$.   
 Moreover,  in Table \ref{Table7}, the sparsity and accuracy results of the  penalized expectile estimation are better  corresponding to the two considered values of $\gamma$.
 \begin{table}[h!]
 	\begin{center}
 {\tiny  
 	\begin{tabular}{ccc|cc|c|ccc|c|}
 		\hline 
 		$\varepsilon$ & n & p & \multicolumn{3}{c|}{100 ${p^0}^{-1} Card(\mathcal{A} \cap \mathcal{A}_n)$}& \quad & \multicolumn{3}{c|}{100$(p-p^0)^{-1} Card(\widehat{\mathcal{A}}_n \backslash \mathcal{A})$} \\  
 		\quad & \quad & \quad  & \multicolumn{2}{c|}{\textit{ag}\_$\mathcal{E}$} & \textit{ag}\_$\mathcal{Q}$ & \quad & \multicolumn{2}{c|}{\textit{ag}\_$\mathcal{E}$} & \textit{ag}\_$\mathcal{Q}$  \\  		
 		\quad & \quad & \quad  & $\gamma=1/10$ & $\gamma=9/10$ & \quad & \quad & $\gamma=1/10$ & $\gamma=9/10$ & \quad \\ \hline 
 		
 		\multirow{1}{1cm}{$\mathcal{N}(0,1)$} & \multirow{1}{1cm}{50} 
 		& 5 & 98.53 & 99 & 18.35 & \quad & 4.3 & 3.9 & 0.5\\    
 		&& 7 & 98.63 & 98.53 & 18.425 & \quad & 4.3 & 4.8 & 0.4\\
 		&& 9 & 98.55 & 98.68 & 18.6 & \quad & 5.06 & 4.9 & 1.22\\ \cline{2-10}
 			
 		& \multirow{1}{1cm}{250} 
 		& 10 & 100 & 100 & 93.62 & \quad & 0 & 0 & 0.15\\
 		&& 25 & 100 & 100 & 93 & \quad & 0 & 0 & 0.187\\    
 		&& 49 & 100 & 100 & 93.3 & \quad & 0 & 0 & 0.204\\  \cline{2-10}
 			
 		& \multirow{1}{1cm}{500} 
 		& 10 & 100 & 100 & 100 & \quad & 0 & 0 & 0.033\\
 		&& 25 & 100 & 100 & 100 & \quad & 0 & 0 & 0.062\\    
 		&& 99 & 100 & 100 & 100 & \quad & 0 & 0 & 0.081\\ \hline 
 		
 		\multirow{1}{1cm}{$\mathcal{C}(0,0.1)$} & \multirow{1}{1cm}{50} 
 		& 5 & 94.7 & 94.43 & 29.8 & \quad & 2 & 3.4 & 2.3\\    
 		&& 7 & 93.8 & 94.73 & 29.48 & \quad & 2.93 & 2.63 & 1.97\\
 		&& 9 & 94.5 & 94.63 & 29.58 & \quad & 2.58 & 3.12 & 1.46\\ \cline{2-10}
 	 		
 		& \multirow{1}{1cm}{250} 
 		& 10 & 100 & 99.98 & 99.63 & \quad & 1.15 & 1.12 & 0.7\\
 		&& 25 & 99.98 & 100 & 99.85 & \quad & 1.94 & 1.19 & 0.61\\    
 		&& 49 & 99.98 & 99.99 & 99.75 & \quad & 0.48 & 1.42 & 0.7\\ \cline{2-10}
 			
 		& \multirow{1}{1cm}{500} 
 		& 10 & 100 & 99.95 & 100 & \quad & 1.02 & 1.08 & 0.017\\
 		&& 25 & 99.95 & 99.95 & 100 & \quad & 0.98 & 0.5 & 0.019\\    
 		&& 99 & 99.98 & 99.98 & 100 & \quad & 0.91 & 0.75 & 0.008\\ \hline
  	\end{tabular}
 	\caption{\small Sparsity study of  \textit{ag}\_$\mathcal{E}$ when  $\tau=1/2$,    $p^0=4$. Comparison with  \textit{ag}\_$\mathcal{Q}$.}
 	  	\label{Table5}
}  
\end{center}
 \end{table}

 \begin{table}[h!]
 	\begin{center}
 {\tiny   
 \begin{tabular}{cc|cc|c|cc|c|cc|c|cc|c|}
 	\hline 
 	\quad & \quad  & \multicolumn{3}{|c|}{100 ${p^0}^{-1} Card(\mathcal{A} \cap \widehat{\mathcal{A}}_n)$}&  \multicolumn{3}{c|}{100$(p-p^0)^{-1} Card(\widehat{\mathcal{A}}_n \backslash \mathcal{A})$} &  \multicolumn{3}{c|}{mean($| \widehat{\bm{\beta}}_n - \bm{\beta}^0|$ )} &  \multicolumn{3}{c|}{mean($| (\widehat{\bm{\beta}}_n - \bm{\beta}^0)_{\mathcal{A}}| $)} \\   
 	
 	&  & \multicolumn{2}{c|}{\textit{ag}\_$\mathcal{E}$} & \textit{ag}\_$\mathcal{Q}$ & \multicolumn{2}{c|}{\textit{ag}\_$\mathcal{E}$} & \textit{ag}\_$\mathcal{Q}$ & \multicolumn{2}{c|}{\textit{ag}\_$\mathcal{E}$} & \textit{ag}\_$\mathcal{Q}$ & \multicolumn{2}{c|}{\textit{ag}\_$\mathcal{E}$} & \textit{ag}\_$\mathcal{Q}$  \\  
 	
 	$\varepsilon$ & n & $\gamma=\frac{1}{10} $ & $\gamma=\frac{9}{10}$ & \quad &$\gamma=\frac{1}{10} $ & $\gamma=\frac{9}{10}$  & \quad &$\gamma=\frac{1}{10} $ & $\gamma=\frac{9}{10}$ & \quad & $\gamma=\frac{1}{10} $ & $\gamma=\frac{9}{10}$  & \quad  \\ \hline  		 		
 		\multirow{1}{1cm}{$\mathcal{N}(0,1)$} & 51 &  100 & 100 & 92.55 & 0.025 & 0 & 10.2 &  0.011 & 0.033 & 0.21 & 0.041 &0.099 & 1.27
 		\\    
 		& 101 &  100 & 100 & 99.325 & 0.025 & 0  & 9.7 &  0.009 & 0.031 & 0.21 & 0.044 &0.216 & 1.07\\
 		& 249 &  100 & 100 & 100 & 0 &0 & 9 &  0.0037 & 0.024 & 0.041 & 0.0375
 		&0.217 & 0.306 \\ 	\hline
 		\multirow{1}{1cm}{$\mathcal{C}(0,0.01)$} & 51 &  99.9 & 100 & 99.1 & 10.21 & 0 & 10.17 &  0.231 & 0.055 & 0.27 & 0.24 &0.349 & 1.83
 		\\    
 		& 101 &  99.97 & 100 & 96.35 & 12.1 & 0& 9.47 &   0.0385 & 0.033 & 0.145 & 0.063 & 0.237 & 0.78\\
 		& 249 &  99.98 & 100 & 100 & 12.38 &0&  8.12  &  0.19 & 0.028 & 0.056 & 0.197 &0.259 & 0.461 \\ \hline
 		\end{tabular}
 		\caption{\small Sparsity study of \textit{ag}\_$\mathcal{E}$ when $\tau=1/2$,  $p=\lfloor  n/5 \rfloor $,  $p^0=2 \lfloor 5^{-1} n^{1/2} \rfloor $. Comparison with \textit{ag}\_$\mathcal{Q}$.}
 		 	\label{Table6}
}
\end{center}
  \end{table}
 
   \begin{table}[hbt!]
 	\begin{center}
 		{\tiny  
 	 \begin{tabular}{cc|cc|c|cc|c|cc|c|cc|c|}
 		\hline 
 		\quad & \quad  & \multicolumn{3}{|c|}{100 ${p^0}^{-1} Card(\mathcal{A} \cap \widehat{\mathcal{A}}_n)$}&  \multicolumn{3}{c|}{100$(p-p^0)^{-1} Card(\widehat{\mathcal{A}}_n \backslash \mathcal{A})$} &  \multicolumn{3}{c|}{mean($| \widehat{\bm{\beta}}_n - \bm{\beta}^0|$ )} &  \multicolumn{3}{c|}{mean($| (\widehat{\bm{\beta}}_n - \bm{\beta}^0)_{\mathcal{A}}| $)} \\   
 		
 		&  & \multicolumn{2}{c|}{\textit{ag}\_$\mathcal{E}$} & \textit{ag}\_$\mathcal{Q}$ & \multicolumn{2}{c|}{\textit{ag}\_$\mathcal{E}$} & \textit{ag}\_$\mathcal{Q}$ & \multicolumn{2}{c|}{\textit{ag}\_$\mathcal{E}$} & \textit{ag}\_$\mathcal{Q}$ & \multicolumn{2}{c|}{\textit{ag}\_$\mathcal{E}$} & \textit{ag}\_$\mathcal{Q}$  \\  
 		
 		$\varepsilon$ & n & $\gamma=\frac{1}{10} $ & $\gamma=\frac{9}{10}$ & \quad &$\gamma=\frac{1}{10} $ & $\gamma=\frac{9}{10}$  & \quad &$\gamma=\frac{1}{10} $ & $\gamma=\frac{9}{10}$ & \quad & $\gamma=\frac{1}{10} $ & $\gamma=\frac{9}{10}$  & \quad  \\ \hline 	
 		\multirow{1}{1cm}{} & 51 &  99.95 & 100 & 95.15 & 6.55 & 0 & 32.5 &  0.0411 & 0.12 & 0.87 & 0.061 & 0.227 & 1.13
 		\\    
 		& 101 &  99.87 & 100 & 99.9 & 1.375 & 0 & 35.6 &  0.027 & 0.101 & 0.42 & 0.0347 & 0.114 & 0.8\\
 		& 249 &  100 & 100 & 100 & 0 &0 & 38.37 &  0.007 & 0.041 & 0.056 & 0.0157 &0.158 & 0.195 \\ \hline
 		
 		\multirow{1}{1cm}{$\mathcal{C}(0,0.01)$} & 51 &  99.15 & 100 & 96.32 & 19.65 & 0  & 30.9 & 0.062 & 0.0893 & 0.7 &  0.0756 &0.177 & 0.921
 		\\    
 		& 101 &  99.98 & 99.99 & 99.5 & 18.55 &0.05& 36.2 & 0.061 & 0.437 & 0.401 & 0.07 & 0.185 & 0.674\\
 		& 249 &  100 & 99.85 & 100 & 17.44& 0.019 & 32.12 &  0.037 & 0.172 & 0.076 & 0.0575& 0.173& 0.385 \\ \hline
 	 	\end{tabular}
  		\caption{\small Sparsity study of \textit{ag}\_$\mathcal{E}$ when $\tau=1/2$, $p=\lfloor  n(2\log{n})^{-1} \rfloor $,  $p^0=2 \lfloor n^{1/4} \rfloor $. Comparison with \textit{ag}\_$\mathcal{Q}$.}
  			\label{Table7}
  	}
  \end{center}
   \end{table}
 \section{Application on real data}
 \label{Sect_appli}
 In addition to the simulation study presented in the previous section, we demonstrate in this section the practical utility of the proposed estimator. Thus,   \textit{ag}\_$\mathcal{E}$ will be used on real data concerning air pollution, especially two gases: ozone and nitrogen dioxide.
 \subsection{Data}
 We consider the following data:  \href{http://archive.ics.uci.edu/ml/datasets/Air+Quality}{\textit{http://archive.ics.uci.edu/ml/datasets/Air+Quality}}.   
 The studied variables are the concentrations of: carbon monoxide (CO), benzene  (C6H6), nitrogen oxides  (NOx), nitrogen dioxide (NO2), ozone  (O3), temperature (T), relative humidity (RH) and absolute humidity (AH).
 The explained variables are O3 and NO2. We will add to the explanatory variables, the concentrations of O3 and NO2 up to three day before.   More precisely, for ozone, we will add the variables 
 $\text{O3}_{-1},\text{O3}_{-2},\text{O3}_{-3}$  which are respectively the daily ozone concentration one, two and three days before the observation.   
 All usable data are taken from March 13, 2004 to May 4, 2005.
 Learning database is from May 1 to September 15, 2004 (115 observations) while the test database is from September 16 to 30,  2004 (15 observations).
 \subsection{Grouped case}
 We will first determine the groups of variables influencing the explained variable.\\
 \textbf{ Model and parameters choice}:   We consider $\tau=1/2$,  $\gamma=1$ (as $p \ll n$), $\lambda_n=n^\xi$, with $\xi$ chosen so that the accuracy of the estimation of the non-zero parameters is minimal. 
 Three group of variables are formed: the "pollutant" group: CO, C6H6, NOx and O3 or NO2 depending on the studied explained variable; the "weather" group: T, RH, AH; the "past" group: $\text{O3}_{-1},\text{O3}_{-2},\text{O3}_{-3}$ or  $\text{NO2}_{-1},\text{NO2}_{-2},\text{NO2}_{-3}$, depending on the studied explained variable.\\
  \textbf{Results when the response variable is  Ozone.}  
 Following values  for the euclidean norm of the grouped variables are obtained:
 $ \| \hat{\bm{\beta}}_n ({\text{pollutant}})\|_2=0.46$, $ \| \hat{\bm{\beta}}_n({\text{weather}} )\|_2=0.15$, $ \| \hat{\bm{\beta}}_n({\text{past}}) \|_2=0.29$.
  Hence, all groups are selected. \newline
 \textbf{Results when the response variable is Nitrogen dioxide.} 
 In this case: $ \| \hat{\bm{\beta}}_n ({\text{pollutant}})\|_2=0.15$, $ \| \hat{\bm{\beta}}_n ({\text{weather}} ) \|_2=0.25$, $ \| \hat{\bm{\beta}}_n ({\text{past}})\|_2=0.0091$. 
  \subsection{Case ungrouped variables}
Because the three groups of variables are relevant, to better study the influence of each variable of a group, we now consider models with ungrouped  explanatory variables.\\
\textbf{Parameters choice}:  We take $\gamma=1$ (since $p \ll n$).    For a better approach of the model,  the expectile index $\tau$ is estimated for each model, taking into account relation (\ref{tau}). More precisely, using $(\widetilde{y}_i)_i$ the normalized observations of $Y$, $\tau$ is estimated by:
 $$
 \widehat{\tau}=\frac{n^{-1} \sum_{i=1}^n \widetilde{y}_i \e1_{\widetilde{y}_i<0}}{n^{-1} (\sum_{i=1}^n \widetilde{y}_i \e1_{\widetilde{y}_i<0}-\sum_{i=1}^n \widetilde{y}_i \e1_{\widetilde{y}_i>0})}.
 $$
 Tables \ref{Table8} and  \ref{Table9} present the  \textit{MAD} which is the empirical mean of the absolute values of residuals and the empirical variance of the residuals. \\
 \textbf{Results when the response variable is  Ozone.}     For $\hat{\tau}=0.42$, $\lambda_n=n^{-0.999}$,   the selected explanatory variables by the adaptive LASSO expectile method are:
 $
 \text{C6H6},\text{N02},\text{T},\text{AH}$, $\text{O3}_{-1}$:  $\widehat{\bm{\beta}}_{n;\text{C6H6}}=0.78$, $ \widehat{\bm{\beta}}_{n;\text{N02}}=0.016$,  $\widehat{\bm{\beta}}_{n;\text{T}}=0.14$, $ \widehat{\bm{\beta}}_{n;\text{AH}}=-0.024$,  $\widehat{\bm{\beta}}_{n;\text{O3}_{-1}}=0.18$.  
 The results by the adaptive LASSO expectile method are compared with  those obtained by the adaptive LASSO  quantile  and the classical LS methods  (Table \ref{Table8}). Note that the adaptive LASSO expectile and LS estimators  are more precise than the adaptive LASSO quantile estimator on the learning and test data.\\
\noindent \textbf{Results when the response variable is Nitrogen dioxide.} 
 For $\hat{\tau}=0.31$, $\lambda_n=n^{-0.999}$,   the selected variables are: $\text{CO},\text{NOx},\text{T},\text{AH},\text{NO2}_{-1},\text{NO2}_{-2},\text{NO2}_{-3}$:
 $ \widehat{\bm{\beta}}_{n;\text{CO}}=0.42$,  $
 \widehat{\bm{\beta}}_{n;\text{NOx}}=0.25$,  $
 \widehat{\bm{\beta}}_{n;\text{T}}=0.37$, $
 \widehat{\bm{\beta}}_{n;\text{AH}}=-0.2$,  $
 \widehat{\bm{\beta}}_{n;\text{NO2}_{-1}}=0.066$,  $
 \widehat{\bm{\beta}}_{n;\text{NO2}_{-2}}=0.0001$,  $
 \widehat{\bm{\beta}}_{n;\text{NO2}_{-3}}=0.0096$. 
 The precision of the three estimation methods is roughly the same, with a slight advantage for the adaptive LASSO expectile estimation (Table \ref{Table9}).

 \begin{table}[h!]
 	\begin{center}
 		{\tiny  
 	\begin{tabular}{ccccccc}
 		\hline 
 		\tiny Estimator & \tiny MAD on all data & \tiny  MAD on learning data &  \tiny MAD on test data & \tiny Variance on all data & \tiny Variance learning & \tiny Variance test \\ 
 		\hline 
 		 Adaptive LASSO  expectile & 0.504 & 0.303 & 0.43 & 0.484 & 0.162 & 0.432 \\ \hline 
 		 Adaptive LASSO quantile  & 0.473 & 0.355 & 0.612 & 0.404 & 0.231 & 0.551 \\ \hline 
 		  Least squares & 0.522 & 0.31 & 0.428 & 0.516 & 0.152 & 0.464 \\ \hline
 	\end{tabular}
  	\caption{\small Comparison of the performance of estimators, for Ozone modelisation, ungrouped variable.}
  		\label{Table8}
}
\end{center}
  \end{table}

 \begin{table}[h!]
 	\begin{center} 	{\tiny  
 	\begin{tabular}{ccccccc}
 		\hline 
 		\tiny Estimator & \tiny MAD on all data & \tiny  MAD on learning data &  \tiny MAD on test data & \tiny Variance on all data & \tiny Variance learning & \tiny Variance test \\ 	\hline 
 		  adaptive  LASSO expectile  & 0.880 & 0.347 & 0.416 & 15.54 & 0.194 & 0.331 \\ \hline 
 		  adaptive  LASSO quantile & 0.853 & 0.349 & 0.445 & 16.18 & 0.219 & 0.363 \\ 		
 		\hline 
 		  Least squares & 0.84 & 0.334 & 0.532 & 11.8 & 0.174 & 0.515 \\ \hline
  	\end{tabular}
  		\caption{\small Comparison of the performance of estimator for NO2 modelization, ungrouped variables.}
  	  	\label{Table9}
}
\end{center}
 \end{table}
 
 \section{Proofs}  
 \label{Sect_proofs}
 In this section we present the proofs of the results stated in Sections \ref{Sect_expectile} and \ref{Sect_generalisation}.
 \subsection{Proofs of Section \ref{Sect_expectile}}
 \label{proof_expectile}
 \subsubsection{Proofs of Subsection \ref{subsect_pfix}}
 \begin{proof}[\textbf{Proof of Lemma \ref{L1}}]
 	According to \cite{Wu.Liu.2009}, it suffices to show that, for all $\epsilon >0$ , there exists $B_{\epsilon} >0$ large enough such that, for $n$ large enough:
 	$
 	\PP\big[ \underset{\textbf{u} \in \mathbb{R}^r, \| \textbf{u} \|_2 =1 }{\inf} \mathcal{Q}_n(\eb^0+B_{\epsilon} n^{-1/2} \textbf{u}) > \mathcal{Q}_n(\eb^0 )\big] \geq 1- \epsilon$.  
 	Let $B>0$ be and  $\eu \in \R^r$  such that  $ \| \eu \|_2 =1$. Otherwise, since for $j > p^0$, $\| \eb^0_j+B n^{-1/2} \textbf{u}_j \|_2 - \| \eb^0_j \|_2=\| B n^{-1/2} \textbf{u}_j \|_2  \geq 0$, then we have:
 	\begin{equation}
  	\begin{split}
 &	\mathcal{Q}_n(\eb^0+B n^{-1/2} \textbf{u})-\mathcal{Q}_n(\eb^0 )  \\ 
 	 & \geq n^{-1} (\mathcal{G}_n(\eb^0+B n^{-1/2} \textbf{u}) - \mathcal{G}_n(\eb^0 ) ) + \lambda_n \sum_{j=1}^{p^0} \widehat{\omega}_{n;j} ( \| \eb^0_j+B n^{-1/2} \textbf{u}_j \|_2 - \| \eb^0_j \|_2 ).
 		\label{Enew1}
 	\end{split}
 	\end{equation}
 By the proof of Theorem 1 of \cite{Liao.2018}, under assumptions (A1), (A2), (A3), we have:
 	\begin{align}
 	\mathcal{G}_n(\eb^0+B n^{-1/2} \textbf{u}) - \mathcal{G}_n(\eb^0 )&=\sum_{i=1}^n  \big(\rho_{\tau}(\varepsilon_i-\frac{B\mathbb{X}_i^\top\textbf{u}}{\sqrt{n}})-\rho_{\tau}(\varepsilon_i)\big) \nonumber \\
 	  &	=\sum_{i=1}^n \big(g_{\tau}(\varepsilon_i) \frac{B\mathbb{X}_i^\top\textbf{u}}{\sqrt{n}} + \frac{h_{\tau}(\varepsilon_i)}{2}(\frac{B\mathbb{X}_i^\top\textbf{u}}{\sqrt{n}})^2\big) +o_\PP(1).
 \label{gg}
 	\end{align}
In order to study (\ref{gg}), we will prove the following two asymptotic results:
 	\begin{equation}
 	\label{E8.2}
 	n^{-1/2}\sum_{i=1}^n g_{\tau}(\varepsilon_i)  {\mathbb{X}_i}  \overset{\mathcal{L}}{\underset{n \rightarrow \infty}{\longrightarrow}} \mathcal{N}(0,\sigma_{g_{\tau}}^2 \eU).
 	\end{equation}
 	\begin{equation}
 	\label{E8.4}
 	n^{-1}\sum_{i=1}^n  h_{\tau}(\varepsilon_i)  \mathbb{X}_i \eeX_i^\top  \overset{\mathbb{P}}{\underset{n \rightarrow \infty}{\longrightarrow}}  \mu_{h_{\tau}}  \eU.
 	\end{equation}
 By assumption (A1), we get:  	$	\eE  [ n^{-1/2}g_{\tau}(\varepsilon_i) \mathbb{X}_i ] =0$. 
 	 Moreover, 	by assumptions (A1), (A3) and relation (\ref{Gg}): $\Var\big[n^{-1/2} \sum_{i=1}^n g_{\tau}(\varepsilon_i)  \mathbb{X}_i \big]   \underset{n \rightarrow +\infty}{\longrightarrow } \sigma_{g_{\tau}}^2 \eU$. On the other hand,  for $\xi >0$, we have:
 	\begin{equation*}
 	\begin{split}
 	\eE\big[  \big\| g_{\tau}(\varepsilon_i) \frac{\mathbb{X}_i}{\sqrt{n}} \big\|_2^2 \cdot \e1_{\| g_{\tau}(\varepsilon_i) \frac{\mathbb{X}_i}{\sqrt{n}} \|_2 > \xi}\big]  = & \eE \bigg[ \frac{\| g_{\tau}(\varepsilon_i) \frac{\mathbb{X}_i}{\sqrt{n}} \|_2^4}{\| g_{\tau}(\varepsilon_i) \frac{\mathbb{X}_i}{\sqrt{n}} \|_2^2} \e1_{\| g_{\tau}(\epsilon_i) \frac{\mathbb{X}_i}{\sqrt{n}} \|_2 > \xi}\bigg]  \\ 
 	\leq &\frac{\eE \big[ \| g_{\tau}(\varepsilon_i) \frac{\mathbb{X}_i}{\sqrt{n}} \|_2^4\e1_{\| g_{\tau}(\varepsilon_i) \frac{\mathbb{X}_i}{\sqrt{n}} \|_2 > \xi}\big]}{\xi^2}  \leq \frac{\eE\big[ \| g_{\tau}(\varepsilon_i) \frac{\mathbb{X}_i}{\sqrt{n}} \|_2^4\big]}{\xi^2}.
 	\end{split}
 	\end{equation*}
 	Then, by assumptions (A1) and (A3), we obtain:
 	$$
 	\sum_{i=1}^n \eE\big[  \big\| g_{\tau}(\varepsilon_i) \frac{\mathbb{X}_i}{\sqrt{n}} \big\|_2^2 \cdot \e1_{\| g_{\tau}(\varepsilon_i) \frac{\mathbb{X}_i}{\sqrt{n}} \|_2 > \xi}\big] \leq  \frac{1}{\xi^2} \eE\big[g_{\tau}(\varepsilon)^4\big] \sum_{i=1}^n \bigg(\frac{\mathbb{X}_i^\top \mathbb{X}_i}{n}\bigg)^2 \underset{n \rightarrow +\infty}{\longrightarrow } 0.
 	$$
 Thus, by the Central Limit Theorem (CLT) of Linderberg-Feller we obtain (\ref{E8.2}). \\
 \noindent	Relation (\ref{E8.4}) follows using the decomposition $
 	n^{-1}\sum_{i=1}^n h_{\tau}(\varepsilon_i) \eeX_i \eeX_i^\top =	n^{-1}\sum_{i=1}^n \big[h_{\tau}(\varepsilon_i) -\eE[h_{\tau}(\varepsilon_i)]+ \eE[h_{\tau}(\varepsilon_i)] \big] \eeX_i \eeX_i^\top$ and coupling 
 	assumption (A3) with  the strong Law of Large Numbers (LLN). Then, relations (\ref{gg}), (\ref{E8.2}) and (\ref{E8.4}) imply:
 \begin{equation}
 \label{ggg}
 \mathcal{G}_n(\eb^0+B n^{-1/2} \textbf{u}) - \mathcal{G}_n(\eb^0) =O_\PP(B^2).
 \end{equation}
 	Let us now study the second term of the right-hand side of (\ref{Enew1}).
 	For any $j \leq p^0$ we have $\| \eb_j^0 \|_2 \ne 0$ and, by the consistency of the expectile estimator, we obtain:
 	\begin{equation}
 	\label{onj}
 	\widehat{\omega}_{n;j} \overset{\mathbb{P}}{\underset{n \rightarrow \infty}{\longrightarrow}} \| \eb_j^0 \|_2 ^{-\gamma} \ne 0.
 	\end{equation}
 	By elementary calculations:
 	$
 n^{1/2} \big|\| \eb_j^0+ n^{-1/2}{\textbf{u}_j}  \|_2 - \| \eb_j^0 \|_2 \big|  \leq   \| \eu_j \|_2 $. Using assumption (\ref{E3.1})(a), relation (\ref{onj}) and    Slutsky's Lemma, we obtain: $	 n \lambda_n\sum_{j=1}^{p^0} \widehat{\omega}_{n;j} (\| \eb_j^0+n^{-1/2} {\textbf{u}_j}  \|_2 - \| \eb_j^0 \|_2 ) \overset{\mathbb{P}}{\underset{n \rightarrow \infty}{\longrightarrow}} 0$, from where  
 \begin{equation}
 	\label{pi}
 	\lambda_n\sum_{j=1}^{p^0} \widehat{\omega}_{n;j} (\| \eb_j^0+n^{-1/2} {\textbf{u}_j}  \|_2 - \| \eb_j^0 \|_2 ) =o(n^{-1}).
\end{equation}
 	Finally, relations (\ref{Enew1}), (\ref{ggg}) and (\ref{pi}) imply: $
 	\mathcal{Q}_n(\eb^0+B n^{-1/2} \textbf{u})-\mathcal{Q}_n(\eb^0 ) \geq O_\PP(B^2 n^{-1})+o_\PP(Bn^{-1}) $.  
 	So, the lemma follows by choosing $B$ and $n$ large enough.
 \end{proof}

 \begin{proof}[\textbf{Proof of Theorem  \ref{T1}}]
 	By Lemma \ref{L1}, we can set $\widehat{\textbf{u}}_n \equiv n^{1/2}(\widehat{\eb}_n-\ebo)$ and more generally $\eu=n^{1/2}(\eb-\ebo)$ for $\eb \in \R^r$, $\eu \in \R^r $ such that  $  \| \eu \|_2 \leq C. $
	We have $Y_i-\mathbb{X}_i^\top\eb= \varepsilon_i- n^{-1/2}{\mathbb{X}_i^\top\eu} $. 
 	By elementary calculations, for $t \rightarrow 0$, we obtain  $\rho_\tau(\varepsilon - t)=\rho_\tau(\varepsilon)+g_\tau(\varepsilon)t+t^2/2 h_\tau(\varepsilon^2)+o_\PP(t^2)$. Using also the Cauchy-Schwarz inequality $|\eeX_i^\top \eu | \leq \| \eeX_i\|_2 \|\eu\|_2$, together with assumption (A2), we have: $n^{-1/2} {\mathbb{X}_i^\top\textbf{u}}  \underset{n \rightarrow +\infty}{\longrightarrow } 0$, $ \| \eu \|_2 \leq C$. Thus, we get:
 	\begin{equation}
 	\label{E8}
 	\rho_{\tau}\big(\varepsilon_i-n^{-1/2}\mathbb{X}_i^\top\eu\big)-\rho_{\tau}(\varepsilon_i)=g_{\tau}(\varepsilon_i) \frac{\mathbb{X}_i^\top\eu}{\sqrt{n}} + \frac{h_{\tau}(\varepsilon_i)}{2}\bigg(\frac{\mathbb{X}_i^\top\eu}{\sqrt{n}}\bigg)^2+o(n^{-1}).
 	\end{equation}
 	On the other hand,  $\widehat{\textbf{u}}_n$  minimises the process: $
 	L_n(\textbf{u})\equiv\sum_{i=1}^n \big( \rho_{\tau}\big(\varepsilon_i-n^{-1/2}{\mathbb{X}_i^\top\textbf{u}}\big)-\rho_{\tau}(\varepsilon_i)\big)+n \lambda_n\sum_{j=1}^p \widehat{\omega}_{n;j} \big(\| \eb_j^0+ n^{-1/2}{\textbf{u}_j}  \|_2 - \| \eb_j^0 \|_2 \big) $, 
 	which can be written, taking into account (\ref{E8}):
 	\begin{equation}
 	\label{E8.1}
 	L_n(\textbf{u})= \sum_{i=1}^n \big( g_{\tau}(\varepsilon_i) \frac{\mathbb{X}_i^\top\textbf{u}}{\sqrt{n}} + \frac{h_{\tau}(\varepsilon_i)}{2}\big(\frac{\mathbb{X}_i^\top\textbf{u}}{\sqrt{n}}\big)^2\big) +o_\PP(1) + n\lambda_n\sum_{j=1}^p \widehat{\omega}_{n;j} \big(\| \eb_j^0+\frac{\textbf{u}_j}{\sqrt{n}} \|_2 - \| \eb_j^0 \|_2 \big),
 	\end{equation}
 Remark that relations (\ref{E8.2}) and (\ref{E8.4}) hold for the first term of the right-hand side of   (\ref{E8.1}). \\
 \noindent 	We are now interested in the second term of the right-hand side of  relation (\ref{E8.1}). We will prove: 
 	\begin{equation}
 	\label{E8.5}
 	n	\lambda_n\sum_{j=1}^p \widehat{\omega}_{n;j} (\| \eb_j^0+n^{-1/2}{\textbf{u}_j}  \|_2 - \| \eb_j^0 \|_2) \overset{\mathbb{P}}{\underset{n \rightarrow \infty}{\longrightarrow}} \sum_{j=1}^p W(\eb_j^0,\textbf{u}), 
 	\end{equation}
   	with:
 	$$
 	W(\eb_j^0,\textbf{u}) = \left\{
 	\begin{array}{ll}
 	0 &\quad \text{if} \quad j \in {\cal A}\\
 	0 &\quad \text{if} \quad  j \in {\cal A}^c \quad  \text{and} \quad \textbf{u}_j= \textbf{0}_{d_j} \\
 	\infty& \quad \text{if} \quad j \in {\cal A}^c \quad \text{and} \quad  \textbf{u}_j \ne \textbf{0}_{d_j}.
 	\end{array}
 	\right. 
 	$$
 For  $\forall j \in {\cal A}$, by the proof of Lemma  \ref{L1},  we have:  $n \lambda_n \widehat{\omega}_{n;j} (\| \eb_j^0+n^{-1/2}{\textbf{u}_j}  \|_2 - \| \eb_j^0 \|_2 ) \overset{\mathbb{P}}{\underset{n \rightarrow \infty}{\longrightarrow}} 0$. \\
\noindent   For  $\forall j \in {\cal A}^c$, we have  $ \| \eb_j^0+ n^{-1/2}{\textbf{u}_j}  \|_2 - \| \eb_j^0 \|_2 = n^{-1/2}\|\textbf{u}_j \|_2$ and $\widehat{\omega}_{n;j}  =O_\PP(n^{\gamma /2})$.
 	Using  condition  (\ref{E3.1})(b),  	since $\| n^{1/2}\widetilde{ \eb}_{n;j} \|=O_\PP(1)$ we have, if $\|\textbf{u}_j \|_2 \neq 0$: $
 	 n \lambda_n \widehat{\omega}_{n;j}(\| \eb_j^0+ n^{-1/2}{\textbf{u}_j}  \|_2 - \| \eb_j^0 \|_2 ) =  n^{ (\gamma+1)/2} \lambda_n \| n^{-1/2}\widetilde{ \eb}_{n;j} \|_2 ^{-\gamma} \overset{\mathbb{P}}{\underset{n \rightarrow \infty}{\longrightarrow}} \infty$.
Then,  
 	$$
 n	\lambda_n \widehat{\omega}_{n;j} (\| \eb_j^0+n^{-1/2}{\textbf{u}_j}  \|_2 - \| \eb_j^0 \|_2 ) \overset{\mathbb{P}}{\underset{n \rightarrow \infty}{\longrightarrow}} \left\{
 	\begin{array}{ll}
 	0 &\quad \text{if} \quad \eu_j = \textbf{0}_{d_j}\\
 	\infty & \quad \text{if} \quad \|\eu_j\|_2 \ne 0.
 	\end{array}
 	\right.
 	$$
Relation (\ref{E8.5}) is proved.  From relations (\ref{E8.2}), (\ref{E8.4}), (\ref{E8.1}), (\ref{E8.5}) and the  Slutsky's Lemma, we get:
 	$
 	L_n(\textbf{u}) \overset{\mathcal{L}}{\underset{n \rightarrow \infty}{\longrightarrow}} \textbf{z}^\top\textbf{u} + 2^{-1}\mu_{h_{\tau}} \textbf{u}^\top \eU \textbf{u} + \sum_{j=1}^p W(\eb_j^0,\textbf{u})\equiv L(\textbf{u})$,
with the random vector $\textbf{z} \sim \mathcal{N}(\textbf{0}_r,\sigma_{g_{\tau}}^2 \eU)$.
Now, for all 	$  \textbf{u} \in \R^r$, we write $\textbf{u}=(\textbf{u}_{\mathcal{A}},\textbf{u}_{\mathcal{A}^c})$ with $\textbf{u}_{\mathcal{A}} \in \mathbb{R}^{r^0}$ and $\widehat{\textbf{u}}_n=(\widehat{\textbf{u}}_{n,\mathcal{A}},\widehat{\textbf{u}}_{n\mathcal{A}^c})$.	Since $\widehat{\textbf{u}}_n$ is the   minimizer of $L_n(\textbf{u})$, taking into account (\ref{E8.5}), we  have: $
 	\widehat{\textbf{u}}_{n\mathcal{A}^c} \overset{\mathbb{P}}{\underset{n \rightarrow \infty}{\longrightarrow}} \textbf{0}_{r-r^0}$ and then $L(\textbf{u})= 	\textbf{z}_{\mathcal{A}}^\top\textbf{u}_{\mathcal{A}} + 2^{-1}\mu_{h_{\tau}} \textbf{u}_{\mathcal{A}}^\top \eU_{\mathcal{A}} \textbf{u}_{\mathcal{A}}$,
with $\textbf{z}_{\mathcal{A}} \sim \mathcal{N}(\textbf{0}_{r^0},\sigma_{g_{\tau}}^2 \eU_{\mathcal{A}})$. We can then  find the minimizer of $L(\textbf{u})$, by a result on strictly convex quadratic forms, is $
   -\mu^{-1}_{h_{\tau}} \eU_{\mathcal{A}}^{-1} \textbf{z}_{\mathcal{A}}$. 
 	We have then, by an epi-convergence result of \cite{Geyer.1994} and \cite{KnightFu.2000} that: $
 	\widehat{\textbf{u}}_{n,\mathcal{A}} \overset{\mathcal{L}}{\underset{n \rightarrow \infty}{\longrightarrow}} \mathcal{N}\big(\textbf{0}_{r^0},  \sigma^2_{g_{\tau}} \mu^{-2}_{h_{\tau}}  \eU_{\mathcal{A}}^{-1}\big) $.
 	The proof is finished. 
 \end{proof}
 \begin{proof}[\textbf{Proof of Theorem  \ref{T2}}]
 	By   Lemma \ref{L1}: $\widehat{\eb}_{n,\mathcal{A}} = \eb_{\mathcal{A}}^0 + O_\PP(n^{-1/2})$, $	\widehat{\eb}_{n,\mathcal{A}^c} = O_\PP(n^{-1/2})$. 
 	Then, we consider  $  \eb \in \mathbb{R}^r$, such that $\| \eb -\ebo \|_2=O(n^{-1/2})$, $\eb=(\eb_{\mathcal{A}},\eb_{\mathcal{A}^c})$, $ \eb_{\mathcal{A}} \in \mathbb{R}^{r^0} $, $\eb_{\mathcal{A}^c} \in \mathbb{R}^{r-r^0}$. 
 	In order to prove the theorem,   we  consider an $\eb$ such that  $\| \eb_{\mathcal{A}^c} \|_2 \ne 0$. Then,
 	\begin{equation}
 	\label{E8.7}
 	\begin{split}
 n \big(	\mathcal{Q}_n((\eb_{\mathcal{A}},\textbf{0}_{r-r^0}))-\mathcal{Q}_n((\eb_{\mathcal{A}},\eb_{\mathcal{A}^c})) \big)=&\sum_{i=1}^n \rho_{\tau}(Y_i-\mathbb{X}_{i,\mathcal{A}}^\top\eb_{\mathcal{A}})-\sum_{i=1}^n \rho_{\tau}(Y_i-\mathbb{X}_i^\top\eb)\\
 & \qquad \qquad	-n \lambda_n\sum_{j =p^0+1}^p \widehat{\omega}_{n;j} \| \eb_j \|_2.
 	\end{split}
 	\end{equation}
 	Now,   by assumption (A2) we have,  for any $i=1, \cdots, n$ that $\mathbb{X}_{i,\mathcal{A}}^\top ( \eb_{\mathcal{A}} - \eb_{\mathcal{A}}^0) \underset{n \rightarrow \infty}{\longrightarrow} 0$. Then, we deduce  the following in the same way as for (\ref{E8}): $
 	\sum_{i=1}^n \big(\rho_{\tau}(Y_i-\mathbb{X}_{i,\mathcal{A}}^\top\eb_{\mathcal{A}}) - \rho_\tau(\varepsilon_i) \big)=\sum_{i=1}^n g_{\tau}(\varepsilon_i) \mathbb{X}_{i,\mathcal{A}}^\top ( \eb_{\mathcal{A}} - \eb_{\mathcal{A}}^0) + 2^{-1} h_{\tau}(\varepsilon_i) (\mathbb{X}_{i,\mathcal{A}}^\top ( \eb_{\mathcal{A}} - \eb_{\mathcal{A}}^0))^2 +o_\PP((\mathbb{X}_{i,\mathcal{A}}^\top ( \eb_{\mathcal{A}} - \eb_{\mathcal{A}}^0))^2)$.
 	By similar arguments as for the proof of (\ref{E8.2}), using assumptions (A1), (A2) and since $\eb_{\mathcal{A}} - \eb_{\mathcal{A}}^0 = O(n^{-1/2})$, we deduce: $
 	\sum_{i=1}^n g_{\tau}(\varepsilon_i) \mathbb{X}_{i,\mathcal{A}}^\top ( \eb_{\mathcal{A}} - \eb_{\mathcal{A}}^0)=O_\PP\big(\big(  n^{1/2}( \eb_{\mathcal{A}} - \eb_{\mathcal{A}}^0)^\top n^{-1}\sum_{i=1}^n \mathbb{X}_{i,\mathcal{A}} \mathbb{X}_{i,\mathcal{A}}^\top $ $\cdot n^{1/2}( \eb_{\mathcal{A}} - \eb_{\mathcal{A}}^0) \sigma_{g_{\tau}}^2\big)^{1/2}\big)=O_\PP(1)$ and $
 	\sum_{i=1}^n  h_{\tau}(\varepsilon_i) (\mathbb{X}_{i,\mathcal{A}}^\top ( \eb_{\mathcal{A}} - \eb_{\mathcal{A}}^0))^2$ $ = \mu_{h_{\tau}}  n^{1/2}( \eb_{\mathcal{A}} - \eb_{\mathcal{A}}^0)^\top n^{-1}\sum_{i=1}^n \mathbb{X}_{i,\mathcal{A}} \mathbb{X}_{i,\mathcal{A}}^\top  n^{1/2}( \eb_{\mathcal{A}} - \eb_{\mathcal{A}}^0)+ o_\PP(1) =O_\PP(1)$. We obtain then:
 	\begin{equation}
 	\label{la1}
 	\sum_{i=1}^n \big( \rho_{\tau}(Y_i-\mathbb{X}_{i,\mathcal{A}}^\top\eb_{\mathcal{A}}) - \rho_\tau(\varepsilon_i) \big)=O_\PP(1).
 	\end{equation}
 	Similarly, using also  $\eb_{\mathcal{A}^c} = O(n^{-1/2})$, we prove:
 	\begin{equation}
 	\label{la2}
 	\sum_{i=1}^n \big( \rho_{\tau}(Y_i-\mathbb{X}_i^\top\eb) - \rho_\tau(\varepsilon_i) \big)=O_\PP(1).
 	\end{equation}
 	Since $n^{1/2} \widetilde{\eb}_{n;j}=O_\PP(1)$ for any $ j \in {\cal A}^c$, $\eb_{\mathcal{A}^c} = O(n^{-1/2})$ and  $n^{  (\gamma+1)/2} \lambda_n \underset{n \rightarrow +\infty}{\longrightarrow} \infty$, then the last term of (\ref{E8.7}) can be written: $
  n^{(\gamma+1)/2}  \lambda_n  \sum_{j=p^0+1}^p \| n^{1/2} \widetilde{\eb}_{n;j} \|_2 ^{-\gamma} n^{1/2} \| \eb_j \|_2 \overset{\mathbb{P}}{\underset{n \rightarrow \infty}{\longrightarrow}} \infty$. 
 Combining this last relation together with (\ref{la1}) and (\ref{la2}), we obtain for   (\ref{E8.7}): $
 n \big(	\mathcal{Q}_n((\eb_{\mathcal{A}},\textbf{0}_{r-r^0}))-\mathcal{Q}_n((\eb_{\mathcal{A}},\eb_{\mathcal{A}^c}))\big) \overset{\mathbb{P}}{\underset{n \rightarrow \infty}{\longrightarrow}} -\infty$.
 This implies $	\underset{n \rightarrow \infty}{\text{lim}} \mathbb{P}[\mathcal{A}^c \subseteq \widehat{\mathcal{A}}^c_n] = 1$, from where  $
 	\underset{n \rightarrow \infty}{\text{lim}} \mathbb{P}[\widehat{\mathcal{A}}_n \subseteq \mathcal{A}]= 1$.
  	On the other hand,  for any $ j \in \mathcal{A} $, by Theorem \ref{T1}, $
 	n^{1/2}(\widehat{\eb}_{n;j}-\eb_j^0) \overset{\mathcal{L}}{\underset{n \rightarrow \infty}{\longrightarrow}} \mathcal{N}\big(0,\sigma^2_{g_{\tau}}\mu^{-2}_{h_{\tau}} \eU_{\mathcal{A}_j}^{-1}\big)$, 
 	with $\eU_{\mathcal{A}_j}$ the sub-matrix of $\eU$ with the index $\{d_{j-1}+1,...,d_j\}$.
 Since $\| \eb_j^0 \|_2 \ne 0 $ then  $
 	\underset{n \rightarrow \infty}{\text{lim}} \mathbb{P}[\mathcal{A} \subseteq \widehat{\mathcal{A}}_n] = 1$ which  together with 
 	 $
 	\underset{n \rightarrow \infty}{\text{lim}} \mathbb{P}[\widehat{\mathcal{A}}_n \subseteq \mathcal{A}]= 1$ imply the theorem.
 \end{proof}
 \subsubsection{Proofs of Subsection \ref{subsect_exp_pvar} }
 \label{proofs_subsect_7.1.2}
 Recall first a result given in  \cite{Ciuperca.2021} needed in the following proofs.
 \begin{lemma}[Theorem 2.1.(i) of \cite{Ciuperca.2021}]
 	\label{L2}
 	Under assumptions (A1), (A2), (A3), (A4), we have: $
 	\| \widetilde{\eb}_n - \eb^0 \|_2 = O_\PP\big((p/n)^{1/2}\big)$.
  	More precisely, for $\textbf{u} \in \mathbb{R}^r$   such that $ \| \textbf{u} \|_2 =1$, then for a $B>0$ large enough, we have: $0 <
 	\mathcal{G}_n\big(\eb^0+B  {(p/n)}^{1/2} \textbf{u}\big) - \mathcal{G}_n(\eb^0 ) =O_\PP(B^2 p)$. 
 \end{lemma}

 \begin{proof}[\textbf{Proof of Theorem \ref{T3}}]
 	We show that, for all $  \epsilon >0 $, there exists $ B_{\epsilon} >0$ large enough such that, for  $n$ large enough:
 	\begin{equation}
 	\label{qqn}
 	\mathbb{P}\big[ \underset{\textbf{u} \in \mathbb{R}^r, \| \textbf{u} \|_2 =1 }{\inf} \mathcal{Q}_n\big(\eb^0+B_{\epsilon} {(p/n)}^{1/2} \textbf{u}\big) > \mathcal{Q}_n(\eb^0 )\big] \geq 1- \epsilon.
 	\end{equation}
 	Let $\textbf{u} \in \mathbb{R}^r$ be such that $ \| \textbf{u} \|_2 =1$ and $B>0$. In order to prove   relation (\ref{qqn}), we will study:
 	\begin{equation}
 	\label{E8.8}
 	\begin{split}
 	\mathcal{Q}_n(\eb^0+B {(p/n)}^{1/2} \textbf{u})-\mathcal{Q}_n(\eb^0 )= & n^{-1} ( \mathcal{G}_n(\eb^0+B {(p/n)}^{1/2} \textbf{u}) - \mathcal{G}_n(\eb^0 ) )\\
 	&\quad  + \lambda_n \sum_{j=1}^p \widehat{\omega}_{n;j} ( \| \eb^0_j+B {(p/n)}^{1/2} \textbf{u}_j \| - \| \eb^0_j \| ).
 	\end{split}
 	\end{equation}
 Let us first study the penalty of the right-hand side of relation (\ref{E8.8}).  The following inequality holds, with probability 1: $
 	\underset{j \in \mathcal{A}}{\min} \| \eb^0_j \|_2 \leq  \underset{j \in \mathcal{A}}{\max} \| \widetilde{\eb}_{n;j} - \eb^0_j \|_2  + \underset{j \in \mathcal{A}}{\min} \| \widetilde{\eb}_{n;j} \| _2$. By Lemma \ref{L2}, $\underset{j \in \mathcal{A}}{\max} \| \widetilde{\eb}_{n;j} - \eb^0_j \|_2 =O_\PP(n^{(c-1)/2})$ from which,  since $\alpha > (c-1)/2$ by supposition (A5), it results $n^{-\alpha} \underset{j \in \mathcal{A}}{\max} \| \widetilde{\eb}_{n;j} - \eb^0_j \|_2 =o_\PP(1)$.
 	On the other hand, still using assumption (A5), we have:
 	$
 	K \leq h_0 n^{-\alpha} \leq n^{-\alpha} \underset{j \in \mathcal{A}}{\max} \| \widetilde{\eb}_{n;j} - \eb^0_j \|_2 + n^{-\alpha} \underset{j \in \mathcal{A}}{\min} \| \widetilde{\eb}_{n;j} \|_2 .
 	$
 Then, the two last relations imply:
 	\begin{equation}
 	\label{E8.9}
 	\underset{n \rightarrow +\infty}{\lim} \mathbb{P} \big[\underset{j \in \mathcal{A}}{\min} \| \widetilde{\eb}_{n;j} \|_2  > \frac{Kn^{\alpha}}{2}\big]=1.
 	\end{equation}
 	Taking into account relation (\ref{E8.9}) and  the Cauchy-Schwarz inequality, we get for a constant $B >0$, with probability converging to one, that:
 	\begin{align}
 &	\sum_{j=1}^p \| \widetilde{\eb}_{n;j} \|_2^{-\gamma} ( \| \eb^0_j+B {(p/n)}^{1/2} \textbf{u}_j \|_2 - \| \eb^0_j \|_2 ) \geq  \sum_{j=1}^{p^0} \| \widetilde{\eb}_{n;j} \|_2^{-\gamma} ( \| \eb^0_j+B {(p/n)}^{1/2} \textbf{u}_j \|_2 - \| \eb^0_j \|_2 ) \nonumber \\ 
 &	\geq  -B {(p/n)}^{1/2} \sum_{j=1}^{p^0} \| \widetilde{\eb}_{n;j} \|_2^{-\gamma}  \| \textbf{u}_j \|_2  \geq  -B{(p/n)}^{1/2} (\sum_{j=1}^{p^0} \| \widetilde{\eb}_{n;j} \|_2^{-2\gamma})^{\frac{1}{2}}  \| \textbf{u}\|_2  \nonumber \\ 
 &  \geq  -B {(p/n)}^{1/2} (p^0)^{1/2} \underset{j \in \mathcal{A}}{\min} \| \widetilde{\eb}_{n;j} \|_2^{-\gamma} \geq  -B {(p/n)}^{1/2}  p^{1/2} ( 2^{-1}Kn^{\alpha} )^{-\gamma}.
 \label{iel}
 	\end{align}
 On the other hand, using  (\ref{E4.1}) together with    (\ref{iel}) we obtain:
 $ 	\lambda_n \sum_{j=1}^p \widehat{\omega}_{n;j} ( \| \eb^0_j+B  ({p}/{n})^{1/2}  \textbf{u}_j \|_2 - \| \eb^0_j \|_2 )  \geq - B  (p/n)^{-1/2} p^{1/2}O_\PP \big( \lambda_n  n^{-\alpha \gamma}\big)  =- B  p n^{-1} O_\PP \big( \lambda_n   n^{1/2-\alpha \gamma}\big) = o_\PP\big(Bp/n\big)$. Thus,  
 	\begin{equation}
 	\label{tt}
 	\lambda_n \sum_{j=1}^p \widehat{\omega}_{n;j} ( \| \eb^0_j+B ({p}/{n})^{1/2} \textbf{u}_j \|_2 - \| \eb^0_j \|_2 )\geq - o_\PP (B p/n).
 	\end{equation}
 On the other hand, by Lemma \ref{L2}:
 	$
 	0 < n^{-1}\big( \mathcal{G}_n(\eb^0+B ({p}/{n})^{1/2} \textbf{u}) - \mathcal{G}_n(\eb^0 ) \big)=O_\PP(B^2  p n^{-1})$.
  	Therefore, taking into account (\ref{tt}) by choosing $B$ and  $n$ large enough, we deduce that the first term of the right-hand side of (\ref{E8.8}) dominates the right-hand side and relation (\ref{qqn}) follows.
 \end{proof}
 \begin{proof}[\textbf{Proof of Theorem \ref{T4}}]
 	By Theorem  \ref{T3}, we have that $\widehat{\eb}_n$ belongs, with a probability converging to one, to the set: $
 	\mathcal{V}_p(\eb^0)=\{ \eb \in \R^r , \| \eb - \eb^0 \|_2 \leq B (p/n)^{1/2} \}$, 
 	for $B >0$ large enough.
 	Consider also the parameter set: $
 	\mathcal{W}_n=\{ \eb \in \mathcal{V}_p(\eb^0) , \| \eb_{\mathcal{A}^c} \|_2 >0 \}$.
 	If $\eb \in \mathcal{W}_n$, then $p>p^0$.
 	In order to show the property of sparsity, we will first show:
 	\begin{equation}
 	\label{lolo}
 	\underset{n \rightarrow \infty}{\lim} \mathbb{P}[\widehat{\eb}_n \in \mathcal{W}_n]=0.
 	\end{equation}
 	For this, we consider two vectors $\eb=(\eb_{\mathcal{A}},\eb_{\mathcal{A}^c}) \in \mathcal{W}_n$ and $\eb^{(1)}=(\eb_{\mathcal{A}}^{(1)},\eb_{\mathcal{A}^c}^{(1)})$ such that $\eb_{\mathcal{A}}^{(1)}=\eb_{\mathcal{A}}$ and $\eb_{\mathcal{A}^c}^{(1)}=\bm{0}_{r-r^0}$.
 	Then, in order to prove relation (\ref{lolo}), we will study the difference 
 	\begin{equation}
 	\label{E8.10}
   \mathcal{Q}_n(\eb)-\mathcal{Q}_n(\eb^{(1)} )= n^{-1} ( \mathcal{G}_n(\eb) - \mathcal{G}_n(\eb^{(1)}) ) + \lambda_n \sum_{j=p^0+1}^p  \widehat{\omega}_{n;j} \| \eb_j \|_2.
 	\end{equation}
 	By elementary calculations, in the same way as for relation (\ref{E8}), we can rewrite:
 	\begin{equation*}
 	\begin{split}
 &	\mathcal{G}_n(\eb) - \mathcal{G}_n(\eb^{(1)})   =    \sum_{i=1}^n \big( \rho_{\tau}(\varepsilon_i-\mathbb{X}_{i,\mathcal{A}}^\top (\eb_{\mathcal{A}}-\eb_{\mathcal{A}}^0 )) - \rho_{\tau}(\varepsilon_i-\mathbb{X}_i^\top (\eb_{\mathcal{A}}-\eb_{\mathcal{A}}^0 ,\eb_{\mathcal{A}^c}))\big) \\
     &  =   \sum_{i=1}^n \big(g_{\tau}(\varepsilon_i) \mathbb{X}_{i,\mathcal{A}}^\top (\eb_{\mathcal{A}}-\eb_{\mathcal{A}}^0 ) + \frac{h_{\tau}(\varepsilon_i)}{2} (\mathbb{X}_{i,\mathcal{A}}^\top (\eb_{\mathcal{A}}-\eb_{\mathcal{A}}^0 ))^2 
     + o_\PP((\mathbb{X}_{i,\mathcal{A}}^\top (\eb_{\mathcal{A}}-\eb_{\mathcal{A}}^0 ))^2) \big) \\
    &  - \sum_{i=1}^n \big(g_{\tau}(\varepsilon_i) \mathbb{X}_i^\top (\eb_{\mathcal{A}}-\eb_{\mathcal{A}}^0 ,\eb_{\mathcal{A}^c}) + \frac{h_{\tau}(\varepsilon_i)}{2} (\mathbb{X}_i^\top (\eb_{\mathcal{A}}-\eb_{\mathcal{A}}^0 ,\eb_{\mathcal{A}^c}))^2 
        +o_\PP((\mathbb{X}_i^\top (\eb_{\mathcal{A}}-\eb_{\mathcal{A}}^0 ,\eb_{\mathcal{A}^c}))^2) \big).
 	\end{split}
 	\end{equation*}
 By the Markov inequality:
 	$
 	 n^{-1} \sum_{i=1}^n g_{\tau}(\varepsilon_i) \mathbb{X}_{i,\mathcal{A}}^\top   (\eb_{\mathcal{A}}-\eb_{\mathcal{A}}^0 )  =    \mathbb{E}\big[  n^{-1} \sum_{i=1}^n g_{\tau}(\varepsilon_i) \mathbb{X}_{i,\mathcal{A}}^\top (\eb_{\mathcal{A}}-\eb_{\mathcal{A}}^0 )\big]  + O_\PP( \big( \Var \big[ n^{-1} \sum_{i=1}^n g_{\tau}(\varepsilon_i) \mathbb{X}_{i,\mathcal{A}}^\top (\eb_{\mathcal{A}}-\eb_{\mathcal{A}}^0 )\big]\big)^{1/2}).$ 
 	By assumption (A1) we have $\mathbb{E}[g_{\tau}(\varepsilon_i)]=0$. On the other hand,  by assumptions (A1), (A3) and since $\| \eb_{\mathcal{A}}-\eb_{\mathcal{A}}^0 \|_2 \leq C  (p/n)^{1/2}$, we have: $  	\Var\big[ n^{-1}\sum_{i=1}^n g_{\tau}(\varepsilon_i) \mathbb{X}_{i,\mathcal{A}}^\top (\eb_{\mathcal{A}}-\eb_{\mathcal{A}}^0 )\big]=  n^{-2} \sigma_{g_{\tau}}^2 \sum_{i=1}^n (\eb_{\mathcal{A}}-\eb_{\mathcal{A}}^0)^\top  \mathbb{X}_{i,\mathcal{A}} \mathbb{X}_{i,\mathcal{A}}^\top ( \eb_{\mathcal{A}}-\eb_{\mathcal{A}}^0 )  = O_\PP\big(  p n^{-2} \big)$. 
 	Thus: 
 	$
 	 n^{-1} \sum_{i=1}^n g_{\tau}(\varepsilon_i) \mathbb{X}_{i,\mathcal{A}}^\top   (\eb_{\mathcal{A}}-\eb_{\mathcal{A}}^0 )=O_\PP\big( p^{1/2}n\big)$.
  	Likewise, we find:  $ 	n^{-1} \sum_{i=1}^n g_{\tau}(\varepsilon_i) \mathbb{X}_i^\top (\eb_{\mathcal{A}}-\eb_{\mathcal{A}}^0 ,\eb_{\mathcal{A}^c})=O_\PP\big( p^{1/2}n\big)$.
 	Moreover, since $h_{\tau}(\varepsilon_i)$ is bounded with  probability one, using assumption (A3), we have:
 	$
 	0< n^{-1} \sum_{i=1}^n  h_{\tau}(\varepsilon_i)  (\mathbb{X}_{i,\mathcal{A}}^\top (\eb_{\mathcal{A}}-\eb_{\mathcal{A}}^0 ))^2 = O_\PP( \| \eb_{\mathcal{A}}-\eb_{\mathcal{A}}^0 \|_2 ^2 ) = O_\PP\big( pn^{-1}\big)$.
  	Likewise, we obtain: $ 0< n^{-1} \sum_{i=1}^n  2^{-1}h_{\tau}(\varepsilon_i) (\mathbb{X}_i^\top (\eb_{\mathcal{A}}-\eb_{\mathcal{A}}^0 ,\eb_{\mathcal{A}^c}))^2=O_\PP (p/n)$. 
 	Finally, we have shown that:
 	\begin{equation}
 	\label{E8.11}
 	n^{-1} ( \mathcal{G}_n(\eb) - \mathcal{G}_n(\eb^{(1)}) ) = O_\PP \big(p/n\big).
 	\end{equation}
 	We are now interested of the penalty term of the right-hand side of (\ref{E8.10}). By  Lemma \ref{L2},  for all $j \in \mathcal{A}^c$, we have: $
 	\widehat{\omega}_{n:j}= \| \widetilde{\eb}_{n:j} \|_2 ^{-\gamma} =  \| \widetilde{\eb}_{n:j} - \eb^0_j \|_2 ^{-\gamma} = O_\PP\big(\big( {p}/{n}\big)^{-{\gamma}/{2}}\big)$
 	and since $\eb \in \mathcal{W}_n$ we get:
 	$
 	0< \| \eb_j \|_2 = O_\PP\big( (p/n)^{1/2}\big)$.
 	Therefore, since $p>p^0$, we obtain:
 	\begin{equation}
 	\label{peni}
 	0<\sum_{j=p^0+1}^p \lambda_n \widehat{\omega}_{n;j} \| \eb_j \|_2 = \sum_{j=p^0+1}^p O_\PP\big(\lambda_n\big(p/n\big)^{(1-\gamma)/2}\big).
 	\end{equation}
 	Relations (\ref{E8.11}) and (\ref{peni}) imply for relation (\ref{E8.10}):
 	\begin{equation}
 	\label{16b}
 	 \mathcal{Q}_n(\eb)-\mathcal{Q}_n(\eb^{(1)} )=O_\PP\big(p/n\big) +  (p-p^0) O_\PP\big(\lambda_n\big(p/n\big)^{(1-\gamma)/2}\big).
 	\end{equation}
 Using condition (\ref{E4.2}),  we have for (\ref{16b}) that the second term of the right-hand side that dominates, and therefore:
 	$ \mathcal{Q}_n(\eb)-\mathcal{Q}_n(\eb^{(1)} )= (p-p^0) O_\PP\big(\lambda_n\big( {p}/{n}\big)^{(1-\gamma)/2}\big)$. On the other hand, by   (\ref{E8.10}) and (\ref{E8.11}), we have: $
 	 \mathcal{Q}_n(\ebo)-\mathcal{Q}_n(\eb^{(1)} )=O_\PP\big( {p}/{n}\big)$. Therefore, using (\ref{E4.2}), we get :
 	\begin{equation}
 	\label{gl}
 	 \mathcal{Q}_n(\eb)-\mathcal{Q}_n(\eb^{(1)} ) \gg  \mathcal{Q}_n(\ebo)-\mathcal{Q}_n(\eb^{(1)} ),
 	\end{equation}
 	with probability converging to one. On the other hand, the estimator $\widehat{\eb}_n$ is the minimizer of $\mathcal{Q}_n(\eb)$. Then, relation (\ref{gl}) implies relation (\ref{lolo}), which in turn implies:
 \begin{equation}
  \label{pp}
    \underset{n \rightarrow \infty}{\text{lim}} \mathbb{P}[\widehat{\mathcal{A}}_n \subseteq \mathcal{A}]=1.
 \end{equation}
 	 Finally, by Theorem \ref{T3} and a similar  relation to (\ref{E8.9}) we have:
 $ 	\underset{n \rightarrow \infty}{\text{lim}} \mathbb{P}[\underset{j \in \mathcal{A}}{\min} \| \widehat{\eb}_{n;j} \|_2 >0 ]=1$ 
  which implies $\underset{n \rightarrow \infty}{\text{lim}} \mathbb{P}[\widehat{\mathcal{A}}_n \supseteq \mathcal{A}]=1$ and which  together with (\ref{pp}) complete the proof. 
 \end{proof}
 \begin{proof}[\textbf{Proof of Theorem \ref{T5}}]
 	By Theorems \ref{T3} and \ref{T4}, we consider the parameter $\eb$ under the form: $
 	 \eb=\eb^0+ (p/n)^{1/2}\bm{\eta} $, with $ \bm{\eta}=(\bm{\eta}_{\mathcal{A}},\bm{\eta}_{\mathcal{A}^c}) $,  $\bm{\eta}_{\mathcal{A}^c}=\textbf{0}_{r-r^0} $, $ \| \bm{\eta}_{\mathcal{A}} \|_2 \leq C$. Then, we will study:
 	\begin{equation}
 	\label{E8.12}
 	\mathcal{Q}_n(\eb^0+ (p/n)^{1/2} \bm{\eta})-\mathcal{Q}_n(\eb^0 )=n^{-1} \sum_{i=1}^n [ \rho_{\tau}(\varepsilon_i-(p/n)^{1/2} \mathbb{X}_i^\top \bm{\eta} ) - \rho_{\tau}(\varepsilon_i) ] + \mathcal{P},
 	\end{equation}
 	with $\mathcal{P}=\sum_{j=1}^{p^0} \lambda_n \widehat{\omega}_{n;j} (\| \eb_j \|_2-\| \eb^0_j \|_2)$.
  	For all $j \in {1,...,p^0}$, we have that $| \| \eb_j \|_2-\| \eb^0_j \|_2 | \leq  (p/n)^{1/2} \| \bm{\eta}_j \|_2$. 
 	On the other hand, using assumption (A5), we have $\| \eb^0_j \|_2 \geq h_0=O(n^{\alpha})$ and  $\| \widetilde{\eb}_{n;j} \|_2-\| \eb^0_j \|_2=O_\PP((p/n)^{1/2})=O_\PP(n^{ (c-1)/2})$. Since, by the same assumption, $\alpha > (c-1)/2 $, then, we deduce: $	\widehat{\omega}_{n;j}=O_\PP(n^{-\alpha \gamma})$, for all $j \in \mathcal{A}$. 
 We have then, with probability equal to 1: $
 	| \mathcal{P} | \leq \lambda_n \sum_{j=1}^{p^0}  \widehat{\omega}_{n;j} | \| \eb_j \|_2-\| \eb^0_j \|_2 | \leq 
 	\lambda_n (p/n)^{1/2} \sum_{j=1}^{p^0}  \widehat{\omega}_{n;j}  \|\bm{\eta}_j\|_2$. Applying the  Cauchy-Schwarz inequality  and  condition  (\ref{E4.1}), we get with probability one: $| \mathcal{P} | \leq \lambda_n (p/n)^{1/2}  \big(\sum_{j=1}^{p^0} \widehat{\omega}^2_{n;j} \big)^{1/2} \big(\sum_{j=1}^{p^0} \|\bm{\eta}_j\|^2_2 \big) =\lambda_n (p/n)^{1/2}  \big(\sum_{j=1}^{p^0} \widehat{\omega}^2_{n;j} \big)^{1/2}  \| \bm{\eta}_{\mathcal{A}}\|_2 $. This relation together with condition (\ref{E4.1}) and with $\widehat{\omega}_{n;j}=O_\PP(n^{-\alpha \gamma})$, imply with probability converging to one:  $ | \mathcal{P} |\leq C \lambda_n (p/n)^{1/2}  \big(\sum_{j=1}^{p^0}  n^{-2 \alpha \gamma} \big)^{1/2} \leq C (p/n)^{1/2} (p^0)^{1/2} \lambda_n  n^{-\alpha \gamma}
 	 \leq C (p/n)^{1/2} p^{1/2} \lambda_n  n^{-\alpha \gamma} $ $=  o\big({p}/{n}\big)$.
 	We are now interested in the first term of the right-hand side of (\ref{E8.12}). By assumption (A3) and the Markov inequality, similarly as in the proof of Theorem \ref{T1}, we have:
 	\begin{align*}
 	& \frac{1}{n} \sum_{i=1}^n \big( \rho_{\tau}(\varepsilon_i-\sqrt{\frac{p}{n}} \mathbb{X}_i^\top \bm{\eta} ) - \rho_{\tau}(\varepsilon_i)\big)  \\
  	= & -\frac{1}{n} \sqrt{\frac{p}{n}}  \sum_{i=1}^n g_{\tau}(\varepsilon_i)( \mathbb{X}_{i,\mathcal{A}}^\top \bm{\eta}_{\mathcal{A}})  + \big[  \frac{1}{2n} \frac{p}{n}  \sum_{i=1}^n  \big(\mathbb{E}[h_{\tau}(\varepsilon_i)]  +  h_{\tau}(\varepsilon_i) - \mathbb{E}[h_{\tau}(\varepsilon_i)] \big)\bm{\eta}_{\mathcal{A}} ^\top \mathbb{X}_{i,\mathcal{A}} \mathbb{X}_{i,\mathcal{A}}^\top \bm{\eta}_{\mathcal{A}} \big]( 1+o_\PP(1))   \\ 
 	= & -\frac{1}{n} \sqrt{\frac{p}{n}}  \sum_{i=1}^n g_{\tau}(\varepsilon_i)( \mathbb{X}_{i,\mathcal{A}}^\top \bm{\eta}_{\mathcal{A}})  +  \big[\frac{1}{2n} \frac{p}{n}  \sum_{i=1}^n  \mathbb{E}[h_{\tau}(\varepsilon_i)] \bm{\eta}_{\mathcal{A}} ^\top \mathbb{X}_{i,\mathcal{A}} \mathbb{X}_{i,\mathcal{A}}^\top \bm{\eta}_{\mathcal{A}} \big]( 1+o_\PP(1))=O_\PP\big(p/n\big).
 	\end{align*}
Thus, the minimization of the right-hand side of (\ref{E8.12}) in respect to $\bm{\eta}$, amounts to minimizing  $n^{-1} \sum_{i=1}^n [ \rho_{\tau}(\varepsilon_i-(p/n)^{1/2} \mathbb{X}_i^\top \bm{\eta} ) - \rho_{\tau}(\varepsilon_i) ]$. Moreover, the minimizer of  
 	$-n^{-1} (p n^{-1})^{1/2}  \sum_{i=1}^n g_{\tau}(\varepsilon_i)$ $\cdot ( \mathbb{X}_{i,\mathcal{A}}^\top \bm{\eta}_{\mathcal{A}})  +  (2n)^{-1} p n^{-1}  \sum_{i=1}^n  \mathbb{E}[h_{\tau}(\varepsilon_i)] $ $\cdot \bm{\eta}_{\mathcal{A}} ^\top \mathbb{X}_{i,\mathcal{A}} \mathbb{X}_{i,\mathcal{A}}^\top \bm{\eta}_{\mathcal{A}} $  verifies: $
 	-n^{-1}(p/n)^{1/2}  \sum_{i=1}^n g_{\tau}(\varepsilon_i) \mathbb{X}_{i,\mathcal{A}}$ $ + \eU_{n,\mathcal{A}}  (p/n)^{1/2}\bm{\eta}_{\mathcal{A}} \mu_{h_{\tau}}=\textbf{0}_{p^0}$, 
 	from where:
 	\begin{equation}
 	\label{E8.14}
 	\sqrt{\frac{p}{n}}\bm{\eta}_{\mathcal{A}}=\frac{\eU_{n,\mathcal{A}}^{-1}}{\mu_{h_{\tau}}} \frac{1}{n} \sum_{i=1}^n g_{\tau}(\varepsilon_i) \mathbb{X}_{i,\mathcal{A}} . 
 	\end{equation}
 	Let $\textbf{u} \in \mathbb{R}^{r^0}$ be such that $\| \textbf{u} \|_2 =1$.
  	In order to study $\bm{\eta}_{\mathcal{A}}$ of  (\ref{E8.14}), let us consider the random variables, for $i=1 , \cdots , n$: $
 	R_i \equiv \mu^{-1}_{h_{\tau}} g_{\tau}(\varepsilon_i) \textbf{u}^\top  \eU_{n,\mathcal{A}}^{-1}  \mathbb{X}_{i,\mathcal{A}} $. 
 Taking into account assumption (A1), we have: $\mathbb{E}[R_i]=0$, $\Var[R_i]=  \sigma_{g_{\tau}}^2 \mu_{h_{\tau}}^{-2}  \textbf{u}^\top \eU_{n,\mathcal{A}}^{-1} \mathbb{X}_{i,\mathcal{A}}  \mathbb{X}_{i,\mathcal{A}}^\top \eU_{n,\mathcal{A}}^{-1} \textbf{u}$. Using assumptions   (A1),   (A3), we obtain: $
 	\Var\big[\sum_{i=1}^n R_i\big]  = n\mu_{h_{\tau}}^{-2}  {\sigma_{g_{\tau}}^2 \textbf{u}^\top \eU_{n,\mathcal{A}}^{-1} \textbf{u}} $.
 	Under assumption (A3), we are in the conditions of the CLT for $(R_i)_{1 \leqslant i \leqslant n}$. Moreover, since $(\widehat{\eb}_n-\eb^0)_{\mathcal{A}} = (p/n)^{1/2}\bm{\eta}_{\mathcal{A}}$, taking into account   (\ref{E8.14}), we deduce: $
 	n^{1/2}(\textbf{u}^\top \eU_{n,\mathcal{A}}^{-1} \textbf{u})^{-1/2} \textbf{u}^\top (\widehat{\eb}_n-\eb^0)_{\mathcal{A}} \overset{\mathcal{L}}{\underset{n \rightarrow \infty}{\longrightarrow}} \mathcal{N}\big(0,  \sigma^2_{g_{\tau}} \mu^{-2}_{h_{\tau}}  \big)$.
 \end{proof}
  \subsection{Proofs of Section \ref{Sect_generalisation}}
 \label{proof_Lq}
 \subsubsection{Proofs of Subsection \ref{subsect_Lq1}}
 \begin{proof}[\textbf{Proof of Theorem \ref{Th2.1_CSDA}}]
 \textit{(i)} We prove that for all $\epsilon >0$, there exists a constant $B_\epsilon >0$  large enough, such that when $n$ is large: $
 \PP \big[ \inf_{\eu \in \R^r, \; \| \eu \|_1 =1} \mathcal{G}_n\big(\ebo+B_\epsilon (p/n)^{1/2} \eu;q\big)>\mathcal{G}_n(\ebo;q)\big] \geq 1-\epsilon $.
Consider $B>0$   a constant to be determined later and   $\eu \in \R^p$ such that $\| \eu \|_1=1$. Then:
 \begin{equation}
 \mathcal{G}_n\big(\ebo+B \sqrt{\frac{p}{n}}  \eu ;q\big) -\mathcal{G}_n(\ebo;q)\equiv \Delta_1+\Delta_2 ,
  \label{eq8Lq}
 \end{equation}
 with $\Delta_1 \equiv \sum^n_{i=1}\big(\rho_\tau\big(\varepsilon_i -B (p/n)^{1/2} \mathbb{X}^\top_i  \eu ;q\big)- \rho_\tau(\varepsilon_i;q ) - \eE\big[ \rho_\tau\big(\varepsilon_i -B (p/n)^{1/2} \eeX^\top_i  \eu ;q\big)- \rho_\tau(\varepsilon_i ;q) \big]\big)$, $\Delta_2 \equiv \sum^n_{i=1} \eE\big[ \rho_\tau\big(\varepsilon_i -B (p/n)^{1/2} \eeX^\top_i  \eu;q \big)- \rho_\tau(\varepsilon_i ;q) \big]$. 
By   relation (\ref{eq2Lq}) together with  assumptions   (A3) and (A4), we have:
 \begin{equation}
 \label{eq3Lq}
 0 < \Delta_2 =\big[\frac{1}{2} \mu_{h_{\tau}(q)} B^2 \frac{p}{n} \sum^n_{i=1}  (\mathbb{X}_i^\top \eu)\big] \big(1+o_\PP(1) \big)=O_\PP(B^2p).
 \end{equation}
For $i=1, \cdots, n$, we consider  the following sequence of random variables: $D_i \equiv \rho_\tau \big(\varepsilon_i - B  (p/n)^{1/2} \eeX_i^\top \eu ;q\big) - \rho_\tau(\varepsilon_i;q) +B (p/n)^{1/2} g_\tau(\varepsilon_i;q) \eeX_i^\top \eu $. Thus, $\Delta_1$ can be written:
 \begin{equation}
 \label{eq7Lq}
 \Delta_1 =\sum^n_{i=1} \big[ -B (p/n)^{1/2} g_\tau(\varepsilon_i;q) \eeX_i^\top \eu  +D_i -\eE [D_i] \big].
 \end{equation}
 On the other hand, using assumptions (A1q), (A3),  we have: $\eE\big[ B (p/n)^{1/2} g_\tau(\varepsilon_i;q) \eeX_i^\top \eu\big]=0$ and $\Var \big[ B (p/n)^{1/2} \sum_{i=1}^n g_\tau(\varepsilon_i;q) \eeX_i^\top \eu\big]=B^2 p n^{-1} \Var[g_\tau(\varepsilon;q)] \eu^\top \sum_{i=1}^n \eeX_i \eeX_i^\top \eu=O(B^2 p)$. Then, using  $
\sum_{i=1}^n B (p/n)^{1/2} g_\tau(\varepsilon_i;q) \eeX_i^\top \eu =\eE\big[ \sum_{i=1}^n B (p/n)^{1/2} g_\tau(\varepsilon_i;q) \eeX_i^\top \eu\big]  +  O_\PP \big( \big(\Var \big[
 \sum_{i=1}^n $ $  B (p/n)^{1/2} g_\tau(\varepsilon_i;q) \eeX_i^\top \eu	\big] \big)^{1/2}\big),
$ we obtain:
 \begin{equation}
 \label{eq3bL}
 B (p/n)^{1/2} \sum_{i=1}^n g_\tau(\varepsilon_i;q) \eeX_i^\top \eu  =O_\PP (B p^{1/2}).
 \end{equation}
By the same reasoning as in the proof of Theorem  \ref{T1}, we get $D_i =O_\PP \big( (p/n)^{1/2} \eeX_i^\top \eu\big)^2$. On the other hand, by the Cauchy-Schwarz inequality, taking also into account assumption (A4) and the fact that $\| \eeX_i\|_2 \leq r^{1/2} \| \eeX_i\|_\infty$, we obtain  $(p/n)^{1/2} \max_{1 \leqslant i \leqslant n} \| \eeX_i\|_2=o(1)$, which implies $\big(  (p/n)^{1/2} \eeX_i^\top \eu\big)^2 \leq (p/n)^{1/2} \big| \eeX_i^\top \eu\big| $, for any $i=1, \cdots, n$. Then, there exists a constant $c_1>0$ such that   $ |D_i| \leq c_1 (p/n)^{1/2}| \eeX_i^\top \eu | $. For the sake of writing clarity, we assume without reducing the generality that $c_1=1$. Then $ | D_i| \e1_{|\varepsilon_i| \leq (p/n)^{1/2} |\eeX_i^\top \eu |} \leq (p/n)^{1/2}| \eeX_i^\top \eu | \e1_{|\varepsilon_i| \leq (p/n)^{1/2} |\eeX_i^\top \eu |}$. 
 Since $(\varepsilon_i)_{1 \leqslant i \leqslant n}$ are independent by assumption (A1q): 
 \begin{align} 
 &\eE\big[ \sum_{i=1}^n (D_i - \eE[D_i])\big]^2 = \sum_{i=1}^n \eE[D_i - \eE[D_i] ]^2 \leq \sum_{i=1}^n \eE[D_i^2] \nonumber \\ 
& =\sum_{i=1}^n \eE[D_i^2 \e1_{|\varepsilon_i| \leq (p/n)^{1/2} |\eeX_i^\top \eu |}]+ \sum_{i=1}^n \eE[D_i^2 \e1_{|\varepsilon_i| > (p/n)^{1/2}|\eeX_i^\top \eu |}]  \equiv T_1+T_2.
 \label{A12}
\end{align}
Using Hölder's inequality, assumptions (A2), (A3),  the fact that $\|\eu \|_2 \leq \|\eu \|_1$, we obtain:
  \begin{align} 
  T_1 & \leq \frac{p}{n} \sum_{i=1}^n  (\eeX_i^\top \eu)^2 \PP \big[ |\varepsilon_i| \leq \sqrt{\frac{p}{n}} |\eeX_i^\top \eu | \big] \leq \frac{p}{n} \sum_{i=1}^n  (\eeX_i^\top \eu)^2 \PP \big[|\varepsilon_i| \leq \sqrt{\frac{p}{n}}  \|\eu \|_1 \max_{1 \leqslant i \leqslant n} \| \eeX_i \|_\infty \big]  \nonumber 
    \end{align}
   \begin{align} 
 & \leq  \frac{p}{n} \sum_{i=1}^n  (\eeX_i^\top \eu)^2 \cdot O \big( F\big(\sqrt{\frac{p}{n}} \max_{1 \leqslant i \leqslant n} \| \eeX_i \|_\infty \big) - F\big(- \sqrt{\frac{p}{n}} \max_{1 \leqslant i \leqslant n} \| \eeX_i \|_\infty \big)\big) \nonumber \\
 & =O\big(p \sqrt{\frac{p}{n}} \big) =o(p).
 \label{A1o}
 \end{align}
 In order to study $T_2$  we get: $
 D_i \e1_{|\varepsilon_i| > (p/n)^{1/2} |\eeX_i^\top \eu |} = h_\tau(\varepsilon_i ; q) \big((p/n)^{1/2}\eeX_i^\top \eu \big)^2 \e1_{|\varepsilon_i| > (p/n)^{1/2} |\eeX_i^\top \eu |} (1+o(1))$. Then: $
 T_2  \leq \big( \big( q (q-1)\big)^2 \sum_{i=1}^n \big|\sqrt{\frac{p}{n}} \eeX_i^\top \eu \big|^4 \int_{|x| > \sqrt{\frac{p}{n}} | \eeX_i^\top \eu |} |x|^{2(q-2)} f(x) dx \big) (1+o(1))$.\\
\noindent  If $1  < q < 2$,  using Hölder's inequality and assumption (A2), we obtain: 
 $ T_2  \leq    \big[ \big( q (q-1)\big)^2 \sum_{i=1}^n \big|(p/n)^{1/2} \eeX_i^\top \eu \big|^{2q} \big] (1+o(1)) = O \big( p^qn^{-(q-1)}\big) $. 
 Then, taking into account relations (\ref{A1o}), (\ref{A12}), we have in the case $q \in (1,2)$:
 \begin{equation}
 \label{eq4Lq}
 \sum_{i=1}^n [D_i -\eE[D_i]]=o_\PP(p^{1/2})+O_\PP\big( p^{1/2}  {p^{(q-1)/2}}{n^{-(q- 1)/2}} \big)=o_\PP(p^{1/2}).
 \end{equation}
 If $q \geq 2$, then $
 T_2  \leq   \big[   \sum_{i=1}^n \big|(p/n)^{1/2} \eeX_i^\top \eu \big|^{4} \big] (1+o(1)) = O \big( p^2n^{-1} \big) $.
  Then, using also  relations (\ref{A12}), (\ref{A1o}),  we obtain:
  \begin{equation}
  \label{eq5Lq}
  \sum_{i=1}^n (D_i -\eE[D_i])=O_\PP\big(   {p}{n^{-1/2}}\big)=o_\PP(p^{1/2}).
  \end{equation}
  Thus relations (\ref{eq4Lq}) and (\ref{eq5Lq}) imply that for all $q >1$: $  \sum_{i=1}^n (D_i -\eE[D_i])=o_\PP(p^{1/2})$. This, together with (\ref{eq3bL}), implies for (\ref{eq7Lq}):
    $ \Delta_1 = - O_\PP(B p^{1/2})+o_\PP(p^{1/2})=-O_\PP(B p^{1/2})$. This last relation, together with relations (\ref{eq8Lq}) and (\ref{eq3Lq}) imply: $
  \mathcal{G}_n\big(\ebo+B  (p/n)^{1/2} \eu ;q\big) -\mathcal{G}_n(\ebo;q) =O_\PP(B^2 p )-O_\PP(B p^{1/2}) >0$,   with probability converging to 1 as $n \rightarrow \infty$, for $B$ large enough. \\
\noindent   \textit{(ii)} For $\eu \in \R^r$, $\| \eu \|_1=1$, $B>0$, we consider the difference:
  \begin{equation}
  \begin{split}
     \mathcal{Q}_n\big(\ebo+B  \sqrt{\frac{p}{n}} \eu;q\big)-\mathcal{Q}_n(\ebo;q)= n^{-1} \big(\mathcal{G}_n\big(\ebo+B  \sqrt{\frac{p}{n}} \eu;q\big)-\mathcal{G}_n(\ebo;q)\big) \\
    +\lambda_n \sum^p_{j=1} \widehat {\omega}_{n;j}(q) \big(\| \eb^0_j +B \sqrt{\frac{p}{n}} \eu_j \|_2 - \| \eb^0_j\|_2 \big).
    \label{loro}
    \end{split}
  \end{equation}
  By claim \textit{(i)}: $  \mathcal{G}_n\big(\ebo+B  (p/n)^{1/2} \eu;q\big)-\mathcal{G}_n(\ebo;q) =O_\PP(B^2p)$. 
  For the penalty of (\ref{loro}),  similarly to (\ref{tt}), using \textit{(i)}, assumptions (A5q) and  (\ref{E4.1}), taking also into account the equivalence of norms in finite dimension,  we have: $ \lambda_n \sum^p_{j=1} \widehat {\omega}_{n;j}(q) \big[\| \eb^0_j +B (p/n)^{1/2} \eu_j \|_2 - \| \eb^0_j\|_2 \big] \geq \lambda_n \sum^{p^0}_{j=1} \widehat {\omega}_{n;j}(q) \big[\| \eb^0_j +B (p/n)^{1/2} \eu_j \|_2 - \| \eb^0_j\|_2 \big]\geq -o_\PP \big( B pn^{-1}\big)$ and claim \textit{(ii)} follows. 
\end{proof}
\begin{proof}[\textbf{Proof of Theorem \ref{T4Lq1}}]
The idea of the proof is similar to that of Theorem \ref{T4}.	Let us consider the parameter sets ${\cal V}_p(\ebo) \equiv \big\{ \eb \in \R^p; \| \eb-\ebo\|_1 \leq B (p/n)^{1/2}\big\}$, with $B>0$ large enough and ${\cal W}_n \equiv \acc{\eb \in {\cal V}_p(\ebo) ; \|\eb_{{\cal A}^c}\|_1 >0}$.  Let us consider two parameter vectors: $\eb=(\eb_{\cal A}, \eb_{{\cal A}^c}) \in {\cal W}_n$ and $ {\eb}^{(1)}=({\eb}_{\cal A}^{(1)}, {\eb}_{{\cal A}^c}^{(1)}) \in {\cal V}_p(\ebo)$,   such that  ${\eb}^{(1)}_{\cal A}=\eb_{\cal A}$ and $ {\eb}_{{\cal A}^c}^{(1)}=\textbf{0}_{r-r^0}$ for which we study the difference: $
	  \mathcal{Q}_n\big(\eb ;q\big)-\mathcal{Q}_n(\eb^{(1)};q)= n^{-1}\big[ \mathcal{G}_n\big(\eb;q\big)-\mathcal{G}_n(\eb^{(1)};q)\big] +\lambda_n \sum^p_{j=p^0+1} \widehat {\omega}_{n;j}(q)   \| \eb_j \|_2   $.
By Lemma A1 of \cite{Hu.Chen.2021}
	\begin{equation}
	\label{eq12Lq}
	\begin{split}
	\sum_{i=1}^n \rho_\tau \big(\varepsilon_i - \sqrt{\frac{p}{n}} \eeX_i^\top \eu ;q\big)=&\sum_{i=1}^n \rho_\tau(\varepsilon_i ;q)+\sqrt{\frac{p}{n}}  \sum_{i=1}^n g_\tau(\varepsilon_i ; q) \eeX_i^\top \eu\\
	 &+\frac{\eE[h_\tau(\varepsilon ; q)]}{2} \frac{p}{n} \eu^\top \sum_{i=1}^n \eeX_i \eeX_i^\top \eu+ o_\PP(p^{1/2}).
	\end{split}
	\end{equation}
	Then, using  (\ref{eq12Lq}) we have, similarly to proof of Theorem \ref{T4}:
	$\mathcal{G}_n\big(\eb;q\big)-\mathcal{G}_n(\eb^{(1)};q)=\big[\mathcal{G}_n\big(\eb;q\big)-\mathcal{G}_n(\ebo;q) \big] -\big[\mathcal{G}_n(\eb^{(1)};q) - \mathcal{G}_n(\ebo;q)\big]  
	=\sum_{i=1}^n g_\tau(\varepsilon_i ; q) \eeX_{i,{\cal A}}(\eb - \ebo)_{\cal A} =O_\PP(p)$. Thus,  this last relation together with a   reasoning similar to (\ref{peni}), taking also into account   (\ref{E4.2}), lead to:
	$
	\mathcal{Q}_n\big(\eb ;q\big)-\mathcal{Q}_n(\eb^{(1)};q) =O_\PP(p/n)+\sum^p_{j=p^0+1} O_\PP \big(\lambda_n (p/n)^{(1-\gamma)/2}\big)=\sum^p_{j=p^0+1} O_\PP \big(\lambda_n ( p/n)^{(1-\gamma)/2}\big)$.
	Moreover, since $\mathcal{Q}_n\big(\ebo ;q\big)-\mathcal{Q}_n(\eb^{(1)};q) =O_\PP\big(p/n\big)$, using also condition (\ref{E4.2}), we have, with probability converging to one, $\mathcal{Q}_n\big(\eb ;q\big)-\mathcal{Q}_n(\eb^{(1)};q) \gg \mathcal{Q}_n\big(\ebo ;q\big)-\mathcal{Q}_n(\eb^{(1)};q)$. This implies $\lim_{n \rightarrow \infty} \PP[\widehat{\eb}_n(q) \in {\cal W}_n]=0$, from where $\lim_{n \rightarrow \infty} \PP[\widehat{\mathcal{A}}_n \subseteq {\cal A}]=1$. Finally, $\lim_{n \rightarrow \infty} \PP[{\cal A}  \subseteq \widehat{\mathcal{A}}_n ]=1$ is proved as in the proof of Theorem \ref{T4}.
\end{proof}

\begin{proof}[\textbf{Proof of Theorem \ref{T5Lq1}}]
 The proof,  similar to that of Theorem \ref{T5},  is omitted.
\end{proof}
 
\noindent  \subsubsection{Proofs of Subsection \ref{subsect_Lq2}}
\begin{proof}[\textbf{Proof of Lemma \ref{Lemma3.1_CSDA}}]
The   proof idea is similar to that of Theorem \ref{Th2.1_CSDA}\textit{(i)}. Consider $\eu \in \R^r$ such that $\| \eu \|_1=1$. Using Hölder's inequality: $|\eeX_i^\top \eu | \leq \| \eeX_i \|_\infty \| \eu \|_1$, relation (\ref{eq3Lq}) becomes:
  \begin{equation}
 \label{eq3LLq}
 0 < \Delta_2 =O_\PP\big(a_n^2 \mu_{h_{\tau}(q)} B^2 \sum^n_{i=1}  (\mathbb{X}_i^\top \eu)\big) \big(1+o_\PP(1) \big)=O_\PP(B^2 a^2_n n).
 \end{equation}
 With respect to the proof of Theorem  \ref{Th2.1_CSDA}\textit{(i)}, now $\Delta_1 \equiv \sum^n_{i=1}\big(\rho_\tau\big(\varepsilon_i -B a_n \mathbb{X}^\top_i  \eu ;q\big)- \rho_\tau(\varepsilon_i;q ) - \eE\big[ \rho_\tau\big(\varepsilon_i -B a_n \eeX^\top_i  \eu ;q\big)- \rho_\tau(\varepsilon_i ;q) \big]\big) $, $D_i \equiv \rho_\tau \big(\varepsilon_i - B a_n \eeX_i^\top \eu ;q\big) - \rho_\tau(\varepsilon_i;q) +B a_n g_\tau(\varepsilon_i;q) \eeX_i^\top \eu $. Relation (\ref{eq3bL}) becomes: $B a_n \sum^{n}_{i=1} g_\tau(\varepsilon_i;q) \eeX_i^\top \eu=O_\PP(B n^{1/2}a_n)$. 
 By  (\ref{A12}), we have:   $\eE\big[ \sum_{i=1}^n (D_i - \eE[D_i])\big]^2 \leq  T_1+T_2$. Regarding $T_1$, as in the proof of Theorem \ref{Th2.1_CSDA}(i): $T_1=O(a_n)=o(n^{1/2} a_n)$. We study now $T_2$.
  If $1 < q < 2$, then $T_2 \leq O(a_n^{2q} n)$ and  $\sum_{i=1}^n (D_i - \eE[D_i]) =O_\PP(a_n^{1/2})+O_\PP(a_n^q n^{1/2})= o_\PP(a_n n^{1/2})$.
  If $q >2$, then $T_2 \leq O(a_n^{4} n)$, from where $\sum_{i=1}^n (D_i - \eE[D_i]) =O_\PP(a_n^{1/2})+O_\PP(a_n^2 n^{1/2})= o_\PP(a_n n^{1/2})$.
  Therefore, in the two cases,  $ \sum_{i=1}^n (D_i - \eE[D_i])=o_\PP(a_n n^{1/2})$ 
and then: $\Delta_1 =- O_\PP(B n^{1/2} a_n) + o_\PP(a_n n^{1/2})= -  O_\PP(B n^{1/2} a_n)$.
In conclusion, since $n^{1/2} a_n \rightarrow \infty$, we have: $
\mathcal{G}_n\big(\ebo+B  a_n \eu;q\big)-\mathcal{G}_n(\ebo;q) =O_\PP(B^2 a_n^2 n) - O_\PP (B a_n n^{1/2})=O_\PP(B^2 a_n^2 n) >0$, with probability converging to 1 as $n \rightarrow \infty$ and for $B$ large enough.
\end{proof}
 \begin{proof}[\textbf{Proof of Theorem \ref{Theorem3.1_CSDA}}]
As in the proof of Theorem \ref{Th2.1_CSDA}\textit{(ii)}, we study the difference ${\cal Q}_n\big(\ebo+B  b_n \eu;q\big)-\mathcal{Q}_n(\ebo;q)$ for $\eu \in \R^r$, $\| \eu \|_1=1$ and $B>0$. By the proof of Lemma \ref{Lemma3.1_CSDA} we have with probability converging to 1 as $n \rightarrow \infty$ and for $B$ large enough:
 \begin{equation}
 \label{eq10Lq}
n^{-1}\big( \mathcal{G}_n\big(\ebo+B  b_n \eu;q\big)-\mathcal{G}_n(\ebo;q)\big) =O_\PP(B^2 b_n^2 ) >0.
  \end{equation}
  Similarly to (\ref{iel}), taking into account assumption (A5q), the equivalence of norms in finite dimension, we obtain:
\begin{equation}
\label{eq11Lq}
 \lambda_n \sum^p_{j=1} \widehat {\omega}_{n;j}(q) \big(\| \eb^0_j +B b_n \eu_j \|_2 - \| \eb^0_j\|_2 \big) \geq  -B (p^0)^{1/2} \lambda_n b_n .
\end{equation}
  By assumption $ \lambda_n (p^0)^{1/2} b_n^{-1}   \rightarrow 0$, then  $B^2 b^2_n  \gg B \lambda_n (p^0)^{1/2} b_n$ which implies that  (\ref{eq10Lq}) dominates (\ref{eq11Lq}). Theorem follows.
 \end{proof}
    
\begin{proof}[\textbf{Proof of Theorem \ref{Theorem3.2_CSDA_L2}}]  
 By the Karush-Kuhn-Tucker optimality  condition we have:
 \begin{equation}
\frac{1}{n} \sum_{i=1}^n g_\tau(Y_i - \eeX_i^\top \eb;q) \eX_{ij}=\lambda_n \widehat {\omega}_{n;j}(q)  \frac{\eb_j}{\| \eb_j\|_2}, \qquad \forall j \in \widehat{\cal A}_n,
 \label{KKT1}
 \end{equation}
 \textit{(i)} By Theorem \ref{Theorem3.1_CSDA} we get $\| \widehat{\eb}_n(q) -\ebo \|_1=O_\PP(b_n)$, with $b_n$ such that $n^{1/2} b_n \rightarrow \infty$ and $\lambda_n (p^0)^{1/2} b_n^{-1}   \rightarrow 0$.  For all $j \in  {\cal A}$, since $\eb^0_j \neq \textbf{0}_{d_j}$, we have  for $n$ large enough that $\| \widehat{\eb}_{n;j}(q) -\eb^0_j \|_1 \leq C b_n$ and then, $\PP[\widehat{\eb}_{n;j}(q) \neq \textbf{0}_{d_j}] {\underset{n \rightarrow \infty}{\longrightarrow}} 1$, which implies $\PP[j \in  \widehat{\cal A}_n ]{\underset{n \rightarrow \infty}{\longrightarrow}} 1 $. Thus, 
 \begin{equation}
 \label{eq13Lq2}
 \PP[{\cal A} \subseteq \widehat{\cal A}_n] {\underset{n \rightarrow \infty}{\longrightarrow}} 1.
 \end{equation}
 We prove now $\PP[\widehat{\cal A}_n \subseteq {\cal A} ] {\underset{n \rightarrow \infty}{\longrightarrow}} 1$, that is $\PP[\widehat{\cal A}_n  \cap {\cal A}^c =\emptyset ] {\underset{n \rightarrow \infty}{\longrightarrow}} 0$. \\
 \noindent  Let us consider $j \in \widehat{\cal A}_n  \cap {\cal A}^c$. Since $j \in\widehat{\cal A}_n $, then relation (\ref{KKT1}) holds.
 On the other hand, by elementary calculations, we have for $t \rightarrow 0$, $g_\tau(\varepsilon -t; q) =g_\tau(\varepsilon;q)+h_\tau(\varepsilon;q) t +o_\PP(t)$, which implies $
  \sum^n_{i=1} g_\tau(Y_i-\eeX_i^\top \eb;q) \eX_{ij}=\sum^n_{i=1} g_\tau(\varepsilon_i ;q)\eX_{ij} +\sum^n_{i=1} \eeX_i^\top (\eb -\ebo) h_\tau(\varepsilon_i ; q) \eX_{ij}+\sum^n_{i=1}  o_\PP \big(\eeX_i^\top (\eb -\ebo) h_\tau(\varepsilon_i ; q) \eX_{ij}  \big)$.
  Similarly to (\ref{eq3bL}), we have: $\sum^n_{i=1} g_\tau(\varepsilon_i ;q)\eX_{ij}=O_\PP(n^{1/2})$.\\
 \noindent  Similarly as the proof of  $\eE[h_\tau(\varepsilon;q)] < \infty$ in \cite{Hu.Chen.2021} we  show that   $\Var[h_\tau(\varepsilon;q)]<\infty$. Then, considering $\eu=\eb-\ebo$ with $\| \eu \|_1=O(b_n)$, using the Markov inequality together with assumptions (A1q), (A2), (A3), we have: $
 \sum^n_{i=1}  \eeX_i^\top \eu h_\tau(\varepsilon_i;q) \eX_{ij}=\eE  \big[\sum^n_{i=1}  \eeX_i^\top \eu h_\tau(\varepsilon_i;q) \eX_{ij} \big]$ $ +O_\PP \big( \big( \Var \big[\sum^n_{i=1}  \eeX_i^\top \eu h_\tau(\varepsilon_i;q) \eX_{ij} \big]\big)^{1/2}\big)
  =O(nb_n)+O_\PP(nb^2_n)=O_\PP(nb^2_n)$.
  Thus,
 \begin{equation}
 \label{eq14Lq}
 \sum^n_{i=1} g_\tau(Y_i-\eeX_i^\top \eb;q) \eX_{ij}=O_\PP(n^{1/2})+O_\PP(n b_n)=O_\PP(n b_n).
 \end{equation}
 On the other hand, since $j \in \widehat{\cal A}_n  \cap {\cal A}^c$, we have that $\lambda_n \widehat {\omega}_{n;j}(q) =\lambda_n \|\widetilde{\eb}_{n;j}(q) \|_2^{-\gamma}=\lambda_n O_\PP(a_n^{-\gamma})$. Since, $\lambda_na_n^{-\gamma}  b_n^{-1}\rightarrow \infty$, thus, this is in contradiction with (\ref{eq14Lq}), since the left-hand side is much smaller than  the one on the right of relation (\ref{KKT1}). Then, $\PP[\widehat{\cal A}_n  \cap {\cal A}^c =\emptyset ] {\underset{n \rightarrow \infty}{\longrightarrow}} 0$.\\
 \noindent Taking also into account (\ref{eq13Lq2}) we have then, $\PP[\widehat{\cal A}_n  = {\cal A}  ] {\underset{n \rightarrow \infty}{\longrightarrow}} 1$.\\
 \textit{(ii)} By Theorem \ref{Theorem3.1_CSDA} and \textit{(i)} of the present theorem, we consider the parameters of the form $\eb=\ebo +b_n \ed$, with $\ed \in \R^r$, $\ed=(\ed_{\cal A}, \ed_{{\cal A}^c}) $, $\ed_{{\cal A}^c}=\textbf{0}_{r-r^0}$, $\| \ed_{\cal A}\|_1 \leq C$. Then,
 \begin{equation}
 \label{eq18Lq}
 \mathcal{Q}_n\big(\ebo+b_n \ed ;q\big)-\mathcal{Q}_n(\ebo;q)=n^{-1}\sum^n_{i=1} \big( \rho_\tau(Y_i-\eeX_i^\top (\ebo+b_n \ed);q) - \rho_\tau(\varepsilon_i ;q) \big) +{\cal P},
 \end{equation}
 with ${\cal P}= \lambda_n \sum^{p^0}_{j=1}  \widehat {\omega}_{n;j}(q) \big( \|\eb^0_j +b_n \ed_j \|_2-\| \eb^0_j \|_2\big) $.   Relation (\ref{eq11Lq}) implies, using, the Cauchy-Schwarz inequality:
 \begin{equation}
 |{\cal P}| =O_\PP(\lambda_n b_n (p^0)^{1/2}).
   \label{eq15Lq}
 \end{equation}
 On the other hand, for the first term for the right-hand side of (\ref{eq18Lq}), we have:
 \begin{align}
 \sum^n_{i=1} \big( \rho_\tau(Y_i-\eeX_i^\top (\ebo+b_n \ed);q) - \rho_\tau(\varepsilon_i ;q) \big)=\sum^n_{i=1} \big( \rho_\tau(\varepsilon_i-b_n \eeX_{i, {\cal A}}^\top \ed_{\cal A};q) - \rho_\tau(\varepsilon_i ;q) \big)  \nonumber \\ 
  = - b_n \sum^n_{i=1} g_\tau(\varepsilon_i ; q) \eeX_{i, {\cal A}}^\top \ed_{\cal A} +\frac{b_n^2}{2} \sum^n_{i=1} \ed^\top_{\cal A} \eeX_{i, {\cal A}} \eeX_{i, {\cal A}}^\top \ed_{\cal A} \eE[h_\tau(\varepsilon_i;q)] \big(1+o_\PP(1)\big), \qquad \qquad 
 \label{eq16Lq}
  \end{align}
 which has as minimizer, the solution of: $
 -b_n \sum^n_{i=1} g_\tau(\varepsilon_i ; q) \eeX_{i, {\cal A}} +n b_n^2 \eU_{n,{\cal A}}  \ed_{\cal A} \eE[h_\tau(\varepsilon_i;q)] =\textbf{0}_{r^0}$,
 from where, using assumption (A3)
 \begin{equation}
 \label{eq17Lq}
 b_n \ed_{\cal A} =\frac{\eU^{-1}_{n,{\cal A}}}{\eE[h_\tau(\varepsilon_i;q)]} \cdot \frac{1}{n} \sum^n_{i=1} g_\tau(\varepsilon_i ; q) \eeX_{i, {\cal A}}.
 \end{equation}
 Thus, the minimum of (\ref{eq16Lq}) is $O(nb_n^2)$. Combining  assumption $ \lambda_n b_n^{-1}    (p^0)^{1/2}   {\underset{n \rightarrow \infty}{\longrightarrow}} 0$,   relation (\ref{eq15Lq}) and supposition (A3), we have that for $\mathcal{Q}_n\big(\ebo+b_n \ed ;q\big)-\mathcal{Q}_n(\ebo;q)$ of   (\ref{eq18Lq}) it is the first term of the right-hand side  which dominates. Then, the minimizer of (\ref{eq18Lq}) is (\ref{eq17Lq}).   Consider  $
 R_i \equiv \big(\eE[h_\tau(\varepsilon_i;q)]\big)^{-1} g_\tau(\varepsilon_i ; q)  \eU^{-1}_{n,{\cal A}}  \eu^\top \eeX_{i, {\cal A}}$  and the rest of the proof is similar
 to the proof of Theorem \ref{T5}.
\end{proof}

\end{document}